\documentclass[a4paper,11pt]{amsart}
\usepackage{etex}

\usepackage{a4wide}
\usepackage{stmaryrd}
\usepackage{enumerate}
\usepackage[all]{xy}
\usepackage[linktocpage=true]{hyperref}
\usepackage{breakurl}

\usepackage{latexsym, amssymb, amsthm,pictex}
\usepackage[centertags]{amsmath}
\usepackage[dvips]{color}
\usepackage{graphicx}
\usepackage[small]{caption}
\usepackage{mathtools}
\usepackage{pstricks}
\usepackage{multido}
\usepackage{dsfont}
\usepackage{color}

\definecolor{usablegreen}{rgb}{0,.5,0}
\definecolor{usablecyan}{rgb}{0,.5,.5}

\newcommand{\ee}{\end{equation}}
\newcommand{\eea}{\end{eqnarray}}
\newcommand{\bean}{\begin{eqnarray*}}
\newcommand{\eean}{\end{eqnarray*}}

 \newif\ifpctex

 \newcommand{\eps}{\epsilon}

\newcommand{\smallx}{\mathpzc{x}}
\newcommand{\smally}{\mathpzc{y}}

\DeclareMathAlphabet{\mathpzc}{OT1}{pzc}{m}{it}

  \newtheorem{lemma}{Lemma}[section]
  \newtheorem{theorem}{Theorem}
  \newtheorem{proposition}[lemma]{Proposition}
  \newtheorem{cor}[lemma]{Corollary}
  \newtheorem{corollary}[lemma]{Corollary}

  \theoremstyle{definition}
  \newtheorem{definition}[lemma]{Definition}
  \theoremstyle{definition}
  \newtheorem{remarkx}[lemma]{Remark}
  \newtheorem{examplex}[lemma]{Example}
  \numberwithin{equation}{section}
  \newtheoremstyle{step}{3pt}{0pt}{\itshape}{}{\bf}{}{.5em}{}

\theoremstyle{step} \newtheorem{step}{Step}

\newcommand{\R}{\mathbb{R}} 
\newcommand{\N}{\mathbb{N}}

\newcommand{\CM}{\mathcal{M}}

\newcommand{\Rand}[1]{\marginpar{#1}} 
\newcommand{\be}[1]{\begin{equation}\label{#1}}
\newcommand{\bew}[1]{\Rand{\vspace{0,6cm}\tt
#1}\begin{equation*}\label{#1}}
\newcommand{\bea}[1]{\Rand{\vspace{0,6cm}\tt
#1}\begin{eqnarray}\label{#1}}

\newcommand{\beL}[2]{\Rand{\vspace{0,6cm}\tt
#1}\begin{lemma}[#2]\label{#1}}
\newcommand{\beD}[2]{\Rand{\vspace{0,6cm}\tt
#1}\begin{definition}[#2]\label{#1}}
\newcommand{\beT}[2]{\Rand{\vspace{0,6cm}\tt
#1}\begin{theorem}[#2]\label{#1}}
\newcommand{\beP}[2]{\Rand{\vspace{0,6cm}\tt
#1}\begin{proposition}[#2]\label{#1}}
\newcommand{\beC}[2]{\Rand{\vspace{0,6cm}\tt #1}\begin{cor}[#2]\label{#1}}

\newcommand{\axiom}{\phantom{vi}\bf}	

\newcommand{\circseg}{\nabla}
\newcommand{\newatop}[2]{\underset{\scriptscriptstyle #2}{#1}}	
\newcommand{\convto}[2][\infty]{\,\newatop{\longrightarrow}{#2 \to #1}\,}	

\providecommand{\tno}{}
\renewcommand{\tno}{\convto{n}}

\newcommand{\folge}[2][n]{(#2_{#1})_{#1\in\N}}
\newcommand{\floor}[1]{\lfloor #1 \rfloor}

\newcommand{\define}[1]{\emph{#1}}
\newcommand{\comment}[1]{}

\newcommand{\proofcase}[1]{\removelastskip\smallbreak\par\emph{#1}\hspace{0.25ex}}
\newcommand{\pstep}[1]{\proofcase{#1.}}
\newcommand{\istep}[2]{\pstep{``#1$\,\Rightarrow\,$#2''}}

\newcommand{\nbd}{\protect\nobreakdash-\hspace{0pt}}
\newcommand{\nclad}[1][m]{$#1$\nbd labelled cladogram}
\newcommand{\nclads}[1][m]{\nclad[#1]s}

\newcommand{\Exp}{\mathbb{E}}
\newcommand{\one}{\mathds{1}}
\newcommand{\Q}{\mathbb{Q}}

\newcommand{\testtree}{\mathfrak{t}}

\newcommand{\uu}{\underline{u}}
\newcommand{\ueta}{\underline{\eta}}
\newcommand{\ux}{\underline{x}}

\newcommand{\toln}{\,\xrightarrow[n\to\infty]{\law}\,}
\newcommand{\Clad}[1][m]{\mathfrak{C}_{#1}}
\newcommand{\shape}[1][T]{\mathfrak{s}_{#1}}
\newcommand{\shapedist}[1][m]{\mathfrak{S}_{#1}}
\newcommand{\sPol}{\Pi_{\shape[]}}
\newcommand{\tensor}[1][m]{[0,1]^{3\cdot \binom{#1}{3}}}

\DeclareMathOperator{\compop}{comp}
\newcommand{\compplain}{\compop}
\newcommand{\comp}[2]{\compplain_{#1}(#2)}
\newcommand{\m}[1][\smallx]{\mathfrak{m}_{#1}}

\newcommand{\massdist}[1][m]{\vartheta_{#1}}
\newcommand{\mPol}{\Pi_{\m[]}}

\newcommand{\gentree}[1]{\llbracket #1 \rrbracket}
\newcommand{\ePol}{\Pi_\iota}

\newcommand{\I}[1][T]{\mathcal{I}_{#1}}
\newcommand{\J}{\mathcal{I}}
\newcommand{\Sset}[1][T]{{\mathcal{S}}_{#1}}
\DeclareMathOperator{\dimVC}{dim_{VC}}
\newcommand{\law}{\mathcal{L}}

\renewcommand{\(}{\bigl(} 	\renewcommand{\)}{\bigr)}

\newcommand{\dGP}{d_\mathrm{GP}}
\newcommand{\dPr}{d_\mathrm{Pr}}

\newcommand{\diota}{d_\mathrm{BGP}}
\newcommand{\dH}[1][]{d^{#1}_\mathrm{H}}
\newcommand{\CCRT}{C_{\mathrm{CRT}}}
\newcommand{\F}{\mathcal{F}}
\newcommand{\B}{\mathcal{B}}

\newcommand{\U}{\mathcal{U}}
\newcommand{\A}{\mathcal{A}}
\newcommand{\MC}{\mathcal{M}}
\newcommand{\T}{\mathbb{T}}
\newcommand{\Tbin}{\T_{2}}

\renewcommand{\H}{\mathbb{H}}
\newcommand{\sxh}{\hat{\smallx}}
\newcommand{\rh}{\hat{r}}

\newcommand{\ch}{\hat{c}}

\newcommand{\Th}{\widehat{T}}
\newcommand{\Tb}{\overline{T}}
\newcommand{\mub}{\bar{\mu}}
\newcommand{\Vc}{\closure{V}}
\newcommand{\Vt}{\tilde{V}}
\newcommand{\ct}{\tilde{c}}
\newcommand{\cb}{\bar{c}}

\DeclareMathOperator{\supp}{supp}
\DeclareMathOperator{\lf}{lf}
\newcommand{\lfatom}{\lf_\mathrm{atom}}

\DeclareMathOperator{\br}{br}

\DeclareMathOperator{\edge}{edge}

\newcommand{\tri}{\Delta}

\DeclareMathOperator{\at}{at}
\newcommand{\triang}{\mathcal{T}}
\newcommand{\tree}{\tau}

\DeclareMathOperator{\conv}{conv}

\newcommand{\closure}{\overline}
\newcommand{\interior}{\mathring}

\newcommand{\restricted}[1]{{\mspace{-1mu}\upharpoonright}_{#1}}
\newcommand{\set}[2]{\{#1:\, #2\}}
\newcommand{\bset}[2]{\bigl\{#1 :\, #2\bigr\}}
\newcommand{\Bset}[2]{\Bigl\{#1 : #2\Bigr\}}
\newcommand{\openint}[2]{\mathopen(#1, #2\mathclose)}
\newcommand{\lopenint}[2]{\mathopen(#1, #2\mathclose]}
\newcommand{\ropenint}[2]{\mathopen[#1, #2\mathclose)}

\newcommand{\blopenint}[2]{\bigl(#1,\, #2\bigr]}

\newcommand{\Rtree}{$\R$\nobreakdash-tree}
\newcommand{\Rtrees}{$\R$\nobreakdash-trees}
\newcommand{\ngon}{$n$\nobreakdash-gon}

\newcommand{\disc}{\mathbb{D}}
\renewcommand{\S}{\mathbb{S}}
\newcommand{\Sub}{\mathcal{S}}
\newcommand{\dx}{\mathrm{d}}

\newcommand{\inta}[3]{\int_{#1} #2 \,\dx #3}
\newcommand{\intax}[4]{\int_{#1} #2 \,#3(\dx #4)}
\newcommand{\intamu}[3][m]{\inta{#2}{#3}{\mu^{\otimes #1}}}
\newcommand{\intamuu}[4][m]{\intax{#2}{#3}{\mu^{\otimes #1}}{#4}}
\newcommand{\C}{\mathcal{C}}
\newcommand{\Cb}{\mathcal{C}_b}

\newenvironment{remark}
  {\pushQED{\qed}\remarkx}
  {\popQED\endremarkx}
\newenvironment{example}
  {\pushQED{\qed}\examplex}
  {\popQED\endexamplex}

\newcommand{\node}{\bullet}
\newcommand{\xyedge}{\ar@{-}}
\newcommand{\leaf}{{}\phantom{\bullet}}
\newcommand{\Bnode}{\circ}
\newcommand{\Rnode}{\times}
\newcommand{\lfdr}{\leaf\ar@{-}[dr]}
\newcommand{\lfd}{\leaf\ar@{-}[d]}
\newcommand{\lfur}{\leaf\ar@{-}[ur]}
\newcommand{\lful}{\leaf\ar@{-}[ul]}
\newcommand{\lfdl}{\leaf\ar@{-}[dl]}
\newcommand{\brd}{\Bnode\ar@{-}[d]}
\newcommand{\brrr}{\node\ar@{-}[rrr]}
\newcommand{\brr}{\Bnode\ar@{-}[r]}
\newcommand{\brl}{\Rnode\ar@{-}[l]}
\newcommand{\brdr}{\Bnode\ar@{-}[dr]}
\newcommand{\brdl}{\node\ar@{-}[dl]}
\newcommand{\brul}{\Rnode\ar@{-}[ul]}
\newcommand{\Rbrdl}{\Rnode\ar@{-}[dl]}

\newcommand{\colsepdefault}{0.75pc}
\xymatrixrowsep{0.75pc} \xymatrixcolsep{\colsepdefault}

\newcommand{\cfigure}[3]{
	\begin{figure}
	\centerline{
	 #2 
	 }
	\caption{#3}
	\label{#1}
	\end{figure}
}
\newcommand{\xymatfig}[3]{\cfigure{#1}{$\xymatrix{#2}$}{#3}}

\psset{linewidth=0.05pt, dotsize=3.6pt, hatchsep=0.8pt, hatchwidth=0.6pt}
\newcommand{\deffillstyle}{vlines}
\newcommand{\dotcircle}{\pscircle[linestyle=dotted,linewidth=.75pt](0,0){1}}
\newcommand{\dotarc}[2]{\psarc[linestyle=dotted,linewidth=.75pt]{-}(0,0){1}{#1}{#2}}
\newcommand{\filledarc}[2]{
	\psline(1;#1)(1;#2)
	\psarc[fillstyle=\deffillstyle]{-}(0,0){1}{#1}{#2}
}

\newcommand{\pictpretri}{{
\psset{fillstyle=\deffillstyle}
\begin{pspicture}(-1.01,-1.01)(1.01,1.01)
		\pspolygon(0,1)(-0.866,-0.5)(0.866,-0.5)
	\SpecialCoor
		\filledarc{30}{90}
		\filledarc{90}{150}
		\filledarc{150}{210}
		\filledarc{210}{270}
		\filledarc{270}{330}
		\filledarc{330}{30}
\end{pspicture}
}}

\newcommand{\pictBox}{{
\psset{fillstyle=\deffillstyle}
\begin{pspicture}(-1.01,-1.01)(1.01,1.01)
	\SpecialCoor
	\filledarc{30}{150}
	\filledarc{150}{270}
	\filledarc{270}{30}
\end{pspicture}
}}

\newcommand{\pictsubtri}{
\begin{pspicture}(-1.01,-1.01)(1.01,1.01)
	\SpecialCoor
	\psarc{-}(0,0){1}{-30}{210}
	\dotarc{210}{-30}
	\pspolygon(1;-30)(1;90)(1;210)
	\multido{\n=-30+60,\ne=30+60}{4}{ \psline{-}(1;\n)(1;\ne) }
	\multido{\n=-30+30,\ne=0+30}{8}{ \psline{-}(1;\n)(1;\ne) }
	\multido{\n=-30+15,\ne=-15+15}{16}{ \psline{-}(1;\n)(1;\ne) }
\end{pspicture}
}

\newcommand{\picttrinobr}[1][0.09\textwidth]{{	
\psset{unit=#1}
\begin{pspicture}(-1.05,-1.02)(1.05,1.02)
	\SpecialCoor
	\dotcircle
	\psdot(1;0)
\end{pspicture}
\hfil
\begin{pspicture}(-1.05,-1.02)(1.05,1.02)
	\SpecialCoor
	\dotcircle
	\psline(1;-90)(1;45)
\end{pspicture}
\hfil
\begin{pspicture}(-1.05,-1.02)(1.05,1.02)
	\SpecialCoor
	\dotarc{105}{-105}
	\dotarc{-90}{45}
	\psline(1;-90)(1;45)
	\psline(1;-105)(1;105)
	\pspolygon[fillstyle=\deffillstyle,linestyle=none](1;-90)(1;45)(1;105)(1;-105)
	\psarc[fillstyle=\deffillstyle]{-}(0,0){1}{45}{105}
	\psarc[fillstyle=\deffillstyle]{-}(0,0){1}{-105}{-90}
\end{pspicture}
\hfil
\begin{pspicture}(-1.05,-1.02)(1.05,1.02)
	\SpecialCoor
	\dotarc{-90}{45}
	\filledarc{45}{-90}
\end{pspicture}
\hfil
\begin{pspicture}(-1.05,-1.02)(1.05,1.02)
	\pscircle[fillstyle=\deffillstyle](0,0){1}
\end{pspicture}
}}

\makeatletter 
\renewcommand\p@enumii{}
\makeatother

\newcommand{\due}

\title[Spaces of algebraic measure trees and triangulations of the circle]{Spaces of algebraic measure trees\\ and triangulations of the circle}
\author{Wolfgang L\"ohr} \author{Anita Winter}
\thanks{Research supported by \emph{DFG-RTG 2131} and by \emph{DFG Priority Programme SPP 1590}. Wolfgang L\"ohr was supported by the \emph{DFG project 415705084}.}
\address{University of Duisburg-Essen, Mathematics, 45117 Essen, Germany}
\email{wolfgang.loehr@uni-due.de (corresponding author), anita.winter@uni-due.de}

\begin{document}
\begin{abstract}
In this paper we present with \emph{algebraic trees} a novel notion of (continuum) trees which generalizes countable
graph-theoretic trees to (potentially) uncountable structures. For that purpose we focus on the tree structure
given by the branch point map which assigns to each triple of points their branch point. We give
an axiomatic definition of algebraic trees, define a natural topology,
and equip them with a probability measure on the Borel-$\sigma$-field. 
Under an order-separability condition, algebraic (measure) trees can be considered as
tree structure equivalence classes of metric (measure) trees (i.e.\ subtrees of \Rtree s). 
Using Gromov-weak convergence (i.e.\ sample distance convergence) of the particular representatives given by the
metric arising from the distribution of branch points, we define a metrizable topology on the space of
equivalence classes of algebraic measure trees.

In many applications, binary trees are of particular interest. We introduce on that subspace with the
sample shape and the sample subtree mass convergence two additional, natural topologies.
Relying on the connection to triangulations of the circle, we
show that all three topologies are actually the same, and the space of binary algebraic measure trees is compact. 
To this end, we provide a formal definition of triangulations of the circle, and show that the coding map
which sends a triangulation to an algebraic measure tree is a continuous surjection
onto the subspace of binary algebraic non-atomic measure trees.
\end{abstract}

\maketitle

\vspace*{-\baselineskip}
\begin{quote}
{\footnotesize  \tableofcontents }
\end{quote}

\section{Introduction}
\label{S:intro}

Graph-theoretic trees are abundant in mathematics and its applications, from computer science to theoretical
biology. A natural question is how to define limits and limit objects as the size of the trees tends to
infinity. On the one hand, there are \emph{local} approaches yielding countably infinite graphs, or generalized
so-called graphings with a Benjamini-Schramm-type approach (going back to \cite{BenjaminiSchramm2001}, see
\cite[Part~4]{Lovasz2012}). On the other hand, if one takes a more \emph{global} point of view, as we are doing
here, the predominant approach is to consider graph-theoretic trees as metric spaces equipped with the
(rescaled) graph distance. Then the limit objects are certain ``tree-like'' metric spaces, most prominently
so-called \Rtree s introduced in \cite{Tits1977}.
They are also of independent interest, e.g.\ for studying isometry groups of
hyperbolic space (\cite{MorganShalen1984}), or as generalized universal covering spaces in the study of the
fundamental groups of one-dimensional spaces (\cite{FischerZastrow2013}). Characterizing the topological
structures induced by \Rtree s has received considerable attention
(\cite{MayerOversteegen90,MayerNikielOversteegen92,Fabel15}). Here, instead of the topological structures, we are
more interested in the ``tree structures'' induced by \Rtree s. We formalize the tree structure with a branch point
map and call the resulting axiomatically defined objects \emph{algebraic trees}.
While, unlike for metric spaces, we do not know any useful notion of convergence for topological spaces or
topological measure spaces, it is essential for us that we can define a very useful convergence of algebraic
measure trees.

Our main motivation lies in suitable state spaces for tree-valued stochastic processes.
The construction and investigation of scaling limits of tree-valued Markov chains within a metric space setup
started with the continuum analogs of the Aldous-Broder-algorithm for sampling a uniform spanning tree from the
complete graph (\cite{EvansPitmanWinter2006}), and of the tree-valued subtree-prune and regraft Markov chain used
in the reconstruction of phylogenetic trees (\cite{EvansWinter2006}).
It continued with the construction of evolving genealogies of infinite size populations in population genetics
(\cite{GrevenPfaffelhuberWinter2013,DepperschmidtGrevenPfaffelhuber2012,KliemLoehr2015,Piotrowiak2010,GrevenSunWinter2016}) and
in population dynamics (\cite{Gloede2012,KliemWinter19}). Moreover, continuum analogues of pruning procedures
were constructed
(\cite{AbrahamDelmasVoisin2010,AbrahamDelmas2012,LoehrVoisinWinter2015,HeWinkel2014,HeWinkel2017}).
All these constructions have in common that they encode trees as metric (measure) spaces or bi-measure
$\R$-trees, and equip the respective space of trees with the Gromov-Hausdorff (\cite{Gromov2000}),
Gromov-weak (\cite{Fukaya1987,GrevenPfaffelhuberWinter2009,Loehr2013}), Gromov-Hausdorff-weak
(\cite{Villani2009,AthreyaLohrWinter16}), or leaf-sampling weak-vague topology
(\cite{LoehrVoisinWinter2015}).

In the present paper, we shift the focus from the metric to the tree structure for several reasons. First,
checking compactness or tightness criteria for (random) metric (measure) spaces is not always easy, and some natural
sequences of trees do not converge as metric (measure) spaces with a uniform rescaling of edge-lengths. At least
for the subspace of binary algebraic measure trees we introduce, the situation is much more favorable, because
it turns out to be compact. Second, the metric is often less canonical than the tree structure in situations
where it is not clear that every edge should have the same length, e.g.\ in a phylogenetic tree, where edges
might correspond to very different evolutionary time spans. Third, one might want to preserve certain
functionals of the tree structure in the limit. For instance, the limit of binary trees is not always binary in
the metric space setup, while this will be the case for our algebraic measure trees. Also, the centroid
function used in \cite{Aldous2000} is not continuous on spaces of metric measure trees, but it is continuous on
our space.

The starting point of our construction is the notion of an $\R$-tree (see \cite{Tits1977,DreMouTer96,Chiswell2001,Evans2008}).
There are many equivalent definitions, but the following one is the most convenient for us:

\begin{definition}[\Rtree s]
A metric space $(T,r)$ is an \define{$\R$-tree}  iff it satisfies the following:
\begin{enumerate}[\axiom(RT1)]
\item $(T,r)$ satisfies the so-called \emph{$4$-point condition}, i.e., for all $x_1,x_2,x_3,x_4\in T$,
    \begin{equation}
    \label{e:4point}
       r(x_1,x_2)+r(x_3,x_4)
       \le
       \max\big\{r(x_1,x_3)+r(x_2,x_4),\,r(x_1,x_4)+r(x_2,x_3)\big\}.
    \end{equation}
\item $(T,r)$ is a connected metric space.
\end{enumerate}
\label{d:realGH}
\end{definition}

Notice that any metric space $(T,r)$ satisfying (RT1) and (RT2) admits a \emph{branch point map} $c\colon T^3\to
T$, i.e., for all $x_1,x_2,x_3\in T$ there exists a unique point $c(x_1,x_2,x_3)\in T$ such that
\begin{equation}
\label{e:bp}
   \bigl\{c(x_1,x_2,x_3)\bigr\} = [x_1,x_2]\cap[x_1,x_3]\cap[x_2,x_3],
\end{equation}
where for $x,y\in T$ the \emph{interval} $[x,y]$ is defined as
\begin{equation}
\label{e:arc}
  [x,y]:=\bset{z\in T}{r(x,z)+r(z,y)=r(x,y)}.
\end{equation}
Given the branch point map $c$, we can recover the intervals via the identity
\begin{equation}\label{e:arc2}
  [x,y]=\bset{z\in T}{c(x,y,z)=z}.
\end{equation}
\xymatfig{f:4pt}{\node\xyedge[dr]^(0){x_1} & & & & \node\\
	& \node\xyedge[rr]^<{c_1}^>{c_2} & & \node\xyedge[ur]^(1){x_3}\xyedge[dr]_(1){x_4} & \\
	\node\xyedge[ur]_(0){x_2}& & & & \node}
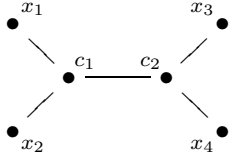
    {The only possible tree shape spanned by four points separates them into two pairs. Here,
    $r(x_1,x_2)+r(x_3,x_4)<\max\{r(x_1,x_3)+r(x_2,x_4),\,r(x_1,x_4)+r(x_2,x_3)\}$, while any other permutation
    yields equality. Furthermore, $c_1=c(x_1,x_2,x_3)=c(x_1,x_2,x_4)$ and $c_2=c(x_1,x_3,x_4)=c(x_2,x_3,x_4)$.
    }
While condition (RT1) is crucial for trees as it reflects the fact that there is only one possible shape for the
subtree spanned by four points (as shown in Figure~\ref{f:4pt}), the assumption of connectedness can be relaxed.
In \cite{AthreyaLohrWinter17}, the notion of a
\emph{metric tree} was introduced to allow for a unified set-up in discrete and continuous situations.
A metric tree $(T,r)$ is defined as a metric space
which can be embedded isometrically into an $\R$-tree such that it contains all branch points
$c(x_1,x_2,x_3)$, $x_1,x_2,x_3\in T$, as defined by \eqref{e:bp}.
To exclude non-tree graphs satisfying the $4$\nbd point condition (see Figure~\ref{f:nontree}), we have to require the property of containing the branch points explicitly.

{
\xymatrixrowsep{1.6pc}
\xymatfig{f:nontree}{& \node\xyedge[dl]\xyedge[dr] & \\ \node\xyedge[rr] & & \node}
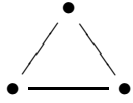
    {The graph shown here is not a tree, but the vertices satisfy the $4$\nbd point condition
     with respect to the graph-distance.
     Condition~(MT\ref{MT:cex}) fails.}
}

\begin{definition}[metric trees]
A metric space $(T,r)$ is a \define{metric tree} if the following holds:
\begin{enumerate}[\axiom(MT1)]
\item\label{MT:4pt} $(T,r)$ satisfies the $4$-point condition (RT1).
\item\label{MT:cex} $(T,r)$ admits all branch points, i.e., for all $x_1,x_2,x_3\in T$ there exists a
	(necessarily unique) $c(x_1,x_2,x_3)\in T$ such that
	\begin{equation} \label{e:brapoi}
	   r\big(x_i,\,c(x_1,x_2,x_3)\big)+r\big(c(x_1,x_2,x_3),\,x_j\big)=r(x_i,x_j)\quad\forall\,i,j\in\{1,2,3\},\,
	   i\ne j.
	\end{equation}
\end{enumerate}
\label{d:metrictree}
\end{definition}

Our main goal is to forget the metric while keeping the tree structure
encoded by the branch point map. To axiomatize the latter, notice that for metric trees the branch point map satisfies the following obvious properties:
\begin{enumerate}[\axiom(BPM1)]
\item\label{BPM:1} The map $c\colon T^3\to T$ is symmetric.
\item\label{BPM:2} The map $c\colon T^3\to T$ satisfies the \emph{$2$-point condition} that
for all $x,y\in T$
\begin{equation}
   c(x,y,y)=y.
\end{equation}
\item\label{BPM:3} The map $c\colon T^3\to T$ satisfies the \emph{$3$-point condition} that
for all $x,y,z\in T$
\begin{equation}
\label{e:2pc}
   c\big(x,y,\,c(x,y,z)\big)=c(x,y,z).
\end{equation}
\item\label{BPM:4} The map  $c\colon T^3\to T$ satisfies the \emph{$4$-point condition} that
for all $x_1,x_2,x_3,x_4\in T$,
\begin{equation}\label{e:BPM4}
  c(x_1,x_2,x_3)\in\big\{c(x_1,x_2,x_4),\,c(x_1,x_3,x_4),\,c(x_2,x_3,x_4)\big\}.
\end{equation}
\end{enumerate}
\begin{definition}[algebraic tree]
An \define{algebraic tree} $(T,c)$ consists of a set $T\not=\emptyset$ and a branch point map $c\colon T^3\to T$ satisfying (BPM1)--(BPM4).
\label{d:algebraic}
\end{definition}

We define a natural topology on an algebraic tree $(T,c)$ as follows.
For each $x\in T$, we define an equivalence relation $\sim_x$ on $T\setminus\{x\}$ such that for all
$y,z\in T\setminus\{x\}$,
  $y\sim_x z$ 
  iff $c(x,y,z)\not =x$.
For $y\in T\setminus\{x\}$, we denote by
\begin{equation}
\label{e:equiv}
  \Sub_x(y):=\set{z\in T}{z\sim_x y}
\end{equation}
the equivalence class w.r.t.\ $\sim_x$ which contains $y$.
$\Sub_x(y)$ should be thought of as a subtree rooted at (but not containing) $x$.
We consider the topology generated by sets of the form \eqref{e:equiv} with $x\ne y$
and denote by $\B(T,c)$ the corresponding Borel $\sigma$\nbd algebra.

Our first main result (Theorem~\ref{t:algtreechar}) relates metric trees with algebraic trees. On the one
hand, if $(T,r)$ is a metric tree, then it is clear that $T$ together with the map $c$ from (MT\ref{MT:cex}) yields an algebraic tree.
On the other hand, we show that every order separable algebraic tree (Definition~\ref{d:ordersep})
is induced by a metric tree in this way.
More concretely, if $\nu$ is a measure on $\B(T,c)$ which is finite
and non-zero on non-degenerate intervals, i.e., on sets of the form
\begin{equation}
\label{e:001}
   [x,y]:=\big\{z\in T:\,c(x,y,z)=z\big\}
\end{equation}
   for $x,y\in T$, $x\ne y$, then a metric representation of $(T,c)$ is given by
\begin{equation}\label{e:rnu}
    r_\nu(x,y)
  :=
    \nu\big([x,y]\big)-\tfrac{1}{2}\nu\big(\{x\}\big)-\tfrac{1}{2}\nu\big(\{y\}\big).
\end{equation}

Next, we equip an algebraic tree $(T,c)$ with a sampling probability measure $\mu$ on $\B(T,c)$, and call the
resulting triple $(T,c,\mu)$ \emph{algebraic measure tree}.
Two algebraic measure trees $(T,c,\mu)$ and $(T',c',\mu')$ are equivalent (compare with
Definition~\ref{d:amtequiv}) if there are $A\subseteq T$, $A'\subseteq T'$ and a bijection $\phi\colon A \to A'$
such that the following holds.
\begin{itemize}
\item $\mu(A)=\mu'(A')=1$, $c(A^3) \subseteq A$ and $c'((A')^3) \subseteq A'$.
\item $\phi$ is measure preserving, and $c'(\phi(x),\phi(y),\phi(z))=\phi(c(x,y,z))$ for all $x,y,z\in T$.
\end{itemize}

Denote by $\T$ the space of all equivalence classes of order separable algebraic measure trees. We equip
$\T$ with a topology based on the Gromov-weak topology
(introduced in \cite{GrevenPfaffelhuberWinter2009} and shown in \cite{Loehr2013} to be equivalent to Gromov's
$\underline\Box_1$-topology from \cite{Gromov2000}).
For that purpose, we introduce a particular
metric representation of an algebraic measure tree. As metric representations are far from being unique, we will
consider the intrinsic metric $r_\nu$ which comes from the branch point distribution, i.e.,
the image measure $\nu:=c_\ast\mu^{\otimes 3}$ of $\mu^{\otimes 3}$ under the branch point map $c$.
We declare that
\begin{equation}
\label{e:convergence}
\begin{aligned}
   (T_n,c_n,\mu_n)\tno(T,c,\mu)\hspace{.2cm}\mbox{ iff }\hspace{.2cm}(T,r_{(c_n)_\ast\mu_n^{\otimes 3}},\mu_n)\to(T,r_{c_\ast\mu^{\otimes 3}},\mu)\mbox{ Gromov-weakly},
\end{aligned}
\end{equation}
or equivalently, $\Phi((T_n,c_n,\mu_n)) \tno \Phi((T,c,\mu))$
for all test functions of the form
\begin{equation}
\label{e:PhiGw}
   \Phi(T,c,\mu)=\Phi^{n,\phi}(T,c,\mu):=\int_{T^n}\phi\big((r_{c_\ast\mu^{\otimes 3}}(x_i,x_j))_{1\le i,j\le n}\big)\,\mu^{\otimes n}(\mathrm{d}\underline{x}),
\end{equation}
where $n\in\N$ and $\phi\in\C_b(\R^{n\times n})$.
We refer to this convergence as {branch point distribution distance Gromov-weak convergence}, or shortly,
\define{bpdd-Gromov-weak convergence}. It is important to keep in mind that---even though bpdd-Gromov-weak
convergence is defined via Gromov-weak convergence of particular metric representations---Gromov-weak
convergence of a sequence $(T_n,r_n,\mu_n)_{n\in\N}$ of metric measure trees does not imply bpdd-Gromov-weak
convergence of the corresponding sequence of algebraic measure trees. For instance, if the diameters
$\sup_{x,y\in\T_n} r_n(x,y)$ converge to zero, the sequence of metric measure trees converges to the trivial
(one-point) tree, while the corresponding sequence of algebraic measure trees might or might not converge, to the
same or a different limit. The same reasoning also applies to the stronger Gromov-Hausdorff-weak topology.

A particular subclass of interest is the space of binary algebraic measure trees.
Similar to encoding compact $\R$-trees by a continuous excursion on the unit interval,
binary algebraic trees can be encoded by \emph{sub-triangulations of the circle} (see Figure~\ref{f:triangtree}), where a sub-triangulation of the circle $\S$ is a closed, non-empty subset $C$ of $\disc$ satisfying the following two conditions:
\begin{enumerate}[\axiom(Tr{i}1)]
	\item The complement of the convex hull of $C$ consists of open interiors of triangles.
	\item $C$ is the union of non-crossing (non-intersecting except at endpoints), possibly
		degenerate closed straight line segments with endpoints in $\S$.
\end{enumerate}
\cfigure{f:triangtree}{
\ifpdf
	\includegraphics{figure-triangtree}
\else
	\psset{unit=0.1667\textwidth, linewidth=0.05pt, dotsize=3.6pt}
		\providecommand{\pstrrootedge}{\psline[linestyle=dashed, linewidth=0.3pt, arrows=*-o]}
\providecommand{\pstrintedge}{\psline[linestyle=dashed, linewidth=0.3pt, arrows=*-]}
\providecommand{\pstrextedge}{\psline[linestyle=dashed, linewidth=0.3pt, arrows=o-]}
\providecommand{\addTriangCommand}{}
\begin{pspicture}(-1.01,-1.01)(1.01,1.01)
\addTriangCommand
\psarc[linestyle=dotted,linewidth=0.75pt]{-}(0,0){1}{0}{360}
\pstrintedge(-0.122008,-0.455342)(0.044658,0.166667)
\pstrintedge(0.455342,-0.455342)(-0.122008,-0.455342)
\pstrextedge(-0.965926,-0.258819)(-0.910684,0.000000)
\pstrintedge(-0.577350,0.333333)(0.044658,0.166667)
\pstrintedge(0.455342,0.788675)(0.622008,0.500000)
\pstrextedge(0.258819,0.965926)(0.455342,0.788675)
\pstrrootedge(0.622008,0.500000)(0.965926,0.258819)
\pstrextedge(0.707107,-0.707107)(0.288675,-0.744017)
\pstrintedge(0.288675,-0.744017)(0.455342,-0.455342)
\pstrintedge(-0.000000,-0.910684)(0.288675,-0.744017)
\pstrextedge(0.258819,-0.965926)(-0.000000,-0.910684)
\pstrextedge(0.707107,0.707107)(0.455342,0.788675)
\pstrextedge(-0.707107,0.707107)(-0.455342,0.788675)
\pstrextedge(-0.258819,-0.965926)(-0.000000,-0.910684)
\pstrintedge(-0.910684,0.000000)(-0.577350,0.333333)
\pstrintedge(0.044658,0.166667)(0.622008,0.500000)
\pstrextedge(-0.965926,0.258819)(-0.910684,0.000000)
\pstrintedge(-0.455342,0.788675)(-0.577350,0.333333)
\pstrextedge(-0.258819,0.965926)(-0.455342,0.788675)
\pstrextedge(-0.707107,-0.707107)(-0.122008,-0.455342)
\pstrextedge(0.965926,-0.258819)(0.455342,-0.455342)
\SpecialCoor
\pspolygon(1;210.000000)(1;240.000000)(1;360.000000)
\pspolygon(1;240.000000)(1;330.000000)(1;360.000000)
\pspolygon(1;90.000000)(1;150.000000)(1;210.000000)
\pspolygon(1;30.000000)(1;60.000000)(1;90.000000)
\pspolygon(1;30.000000)(1;90.000000)(1;360.000000)
\pspolygon(1;240.000000)(1;300.000000)(1;330.000000)
\pspolygon(1;240.000000)(1;270.000000)(1;300.000000)
\pspolygon(1;150.000000)(1;180.000000)(1;210.000000)
\pspolygon(1;90.000000)(1;210.000000)(1;360.000000)
\pspolygon(1;90.000000)(1;120.000000)(1;150.000000)
\end{pspicture}
 \fi
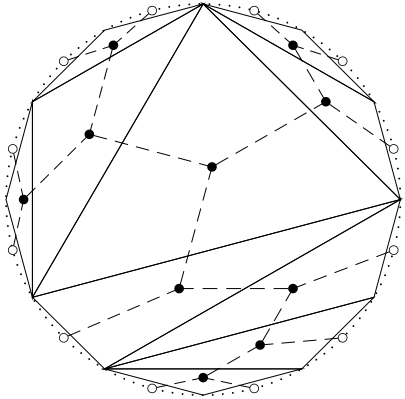
}{A triangulation of the $12$-gon and the tree coded by it.}
Such an encoding was introduced by David Aldous in \cite{Aldous94,Aldous94b},
and there has since then been an increasing amount of research in the random tree community using this approach
(e.g.\ \cite{CurienLeGall11,BroutinSulzbach15,CurienKortchemski15}). Also more general ${}$-angulations and
dissections have been considered which allow for encoding not necessarily binary trees
(\cite{Curien,CurienHaasKortchemski15}). Note, however, that the relation between triangulations and trees has
never been made explicit, except for the finite case, where the tree is the dual graph. 

Aldous originally defines a triangulation of the circle as a
closed subset of the disc the complement of which is a disjoint union of open triangles with vertices on the
circle (\cite[Definition~1]{Aldous94b}). We modify his definition in two respects.
First, we add Condition (Tri2) which enforces existence of branch points and under which triangulations of the
circle are precisely the Hausdorff-metric limits of triangulations of \ngon s as $n\to\infty$.
Second, we extend the definitions to sub-triangulation of the circle (triangulations of a subset of
the circle) which allow for encoding algebraic measure trees with point masses on leaves. In fact,
triangulations of the whole circle encode binary trees with non-atomic measures,
which is relevant in the case of Aldous's CRT.
We formally construct the coding map that associates to a sub-triangulation of the circle the corresponding
binary algebraic measure tree with point-masses restricted to the leaves.
Furthermore, we show that---similar to the case of coding compact \Rtrees\ by continuous excursions---the
coding map is \emph{surjective} and \emph{continuous} when the set of sub-triangulations
is equipped with the Hausdorff metric topology and the set of binary algebraic measure trees with our
bpdd-Gromov-weak topology (Theorem~\ref{t:tree}).

We also analyze the subspace of binary algebraic measure trees with point-masses restricted to the leaves in more
detail. Our third main result (Theorem~\ref{t:topeq}) states that this space in the bpdd-Gromov-weak topology is
topologically as nice as it gets, namely a compact, metrizable space.
We also give two more notions of convergence which turn out to be equivalent to bpdd-Gromov-weak convergence on
this subspace.
One is of combinatorial nature and based on the weak convergence of test functions of the form
\begin{equation}
\label{e:Phicomb}
   \Phi(T,c,\mu)=\Phi^{n,\testtree}(T,c,\mu):=\mu\(\bset{(u_1,...,u_n)\in T^n}{\shape[(T,c)](u_1,...,u_n)=\testtree}\),
\end{equation}
where $\testtree$ is an $n$-cladogram (a binary graph-theoretic tree with $n$ leaves) and $\mathfrak{s}_{(T,c)}$
denotes the shape spanned by a finite sample in $(T,c)$ (Definitions~\ref{def:cladogram} and~\ref{def:treeshape}).
The other one is more in the spirit of stochastic analysis and based on weak convergence of the \emph{tensor of
subtree-masses} read off the algebraic measure subtree spanned by a finite sample (see
Definition~\ref{d:massdist}).
This equivalence allows to switch between different perspectives and turns out to be very useful for the following reasons:
\begin{itemize}
\item Using convergence of sample bpd-distance matrices allows to exploit well-known results about Gromov-weak convergence.
\item Showing convergence of graph theoretic tree-valued Markov chains as the number of vertices tends to
	infinity is, due to the combinatorial nature of the Markov chains, often easiest by showing
	convergence of the sample shape distributions. This has recently been successfully applied in the
	construction of the conjectured continuum limit of the Aldous chain (\cite{Aldous2000}) in
	\cite{LoehrMytnikWinter}, and of the continuum limit of the $\alpha=1$-Ford chain (\cite{Ford2005}) in
	\cite{Nussbaumer}.
\item The convergence of sample subtree-mass tensor distributions allows to analyze the limit process with
	stochastic analysis methods and gives more insight into the global structure of the evolving random
	trees.
\end{itemize}

\medbreak
\noindent{\bf Related work. }As an alternative with better compactness properties to Gromov-Hausdorff
convergence of discrete trees, Curien suggested in \cite{Curien} to look at convergence of coding
triangulations (in Hausdorff metric topology). He also proposed to read off a measured, ordered,
\emph{topological tree} from the limit triangulation, and sketched the construction as quotient w.r.t.\ some
equivalence relation in the special case of the Brownian triangulation. Note, however, that the topological
information cannot be completely encoded by the triangulation, because the latter only encodes the algebraic
measure tree by Theorem~\ref{t:tree}, and the algebraic structure does not determine the topological structure
uniquely (see Example~\ref{ex:homhomeom}). Therefore, Curien did not obtain a general map from the space of
triangulations to a space of trees.

In order to turn the set of valuations on the ring $\mathbb{C}\gentree{x,y}$ into the so-called
\emph{valuative tree}, Favre and Jonsson use in \cite{FavreJonsson04} partial orders to define the tree structure.
Using partial orders is essentailly equivalent to using branch point maps, and under some additional
assumptions (separability, order completeness and edge-freeness), their \emph{nonmetric trees} are equivalent
to our algebraic trees. We want to stress, however, that for our theory the branch point map plays a much more
crucial role than the partial order.
The relation between partial orders and algebraic trees is further discussed in Section~\ref{S:trees}.

The random exchangeable \emph{didendritic systems} introduced recently by Evans, Gr\"ubel, and Wakolbinger in
\cite{EvansGruebelWakolbinger17} can be considered as rooted, ordered versions of binary algebraic measure trees
with diffuse measure on the set of leaves. A didendritic system is an equivalence relation on $\N\times \N$
together with two partial orders on the set of equivalence classes.
An exchangeable didendritic system is similar to our sequence of sample-shape distributions. The authors also
introduce a particular metric representation as an \Rtree. Even though it is implicit in their work
that they think of the set of exchangeable didendritic systems as equipped with a kind of sample shape
convergence, they do not define it explicitly and do not analyze the resulting topological space.

Close relatives of algebraic measure trees have recently been studied independently by Forman in
\cite{Forman2018}. He uses ideas from \cite{FormanHaulkPitman2018} to represent rooted trees by so-called
\emph{hierarchies} (certain sets of subsets) on $\N$, which are similar to the didendritic systems in
\cite{EvansGruebelWakolbinger17}, but unordered. Thus, exchangeable random hierarchies can be thought of as
rooted versions of algebraic measure trees. Forman shows that the resulting equivalence classes of rooted
measure \Rtree s coincide with the so-called \emph{mass-structural} equivalence classes, which he defines by
bijections preserving intervals as well as masses of points, intervals and certain sub-trees.
He also singles out a particular representative, which he calls \emph{interval partition tree}, with the
essentially same metric as in \cite{EvansGruebelWakolbinger17} (not restricted to the binary case). This metric
follows a similar idea to but is different from our $r_\nu$.
Note that \cite{Forman2018} does not talk about convergence of trees or introduce a notion of ``continuum tree''
without a measure.

\medbreak
\noindent{\bf Outline. }The rest of the paper is organized as follows.
In Section~\ref{S:trees}, we introduce our concept of \emph{algebraic trees} by formalising the
branch point map as a tertiary operation on the tree. We also introduce an intrinsic Hausdorff topology and
characterize compactness (Proposition~\ref{p:compact}) and second countability (Proposition~\ref{p:separable}).
We show that under a separability constraint, algebraic trees can be seen as metric trees (subtrees of \Rtree
s), where the metric structure has been ``forgotten'' (Theorem~\ref{t:algtreechar}), and give an example that
the separability condition cannot be dropped without replacement.

In Section~\ref{S:amt}, we introduce the space of (equivalence
classes of) order separable \emph{algebraic measure trees}, and equip it with the Gromov-weak topology with respect to
the metric associated with the branch point distribution. We show that the resulting space is separable and
metrizable (Corollary~\ref{c:bpdd-metrizable}).
Furthermore, we prove a Carath\'eodory-type extension theorem, which is helpful for constructing algebraic
measure trees (Propositions~\ref{p:caratheodory} and \ref{p:construction}).

In Section~\ref{S:triangulation}, we give a definition of \emph{triangulations of the circle}, and show that
they are precisely the limits of triangulations of $n$-gons (Proposition~\ref{p:fintriapp}).
We also formalize the notion of the algebraic measure tree associated with a given triangulation of the circle.
This correspondence has been allured to in the literature, but it has never been made
precise (except for finite trees), and it has never been shown a tree in what sense is coded by a
triangulation of the circle.
We show that the resulting \emph{coding map} (mapping triangulations to trees) is well-defined
and surjective onto the space of binary algebraic measure trees with non-atomic measure.
Furthermore, the coding map is \emph{continuous} if the space of triangulations is equipped with the Hausdorff metric
topology, and the space of trees with the bpdd-Gromov-weak topology (Theorem~\ref{t:tree}).

In Section~\ref{S:topo}, we consider the subspace of \emph{binary} algebraic measure trees, and introduce
two other, natural notions of convergence.
We use the construction of the coding map from Section~\ref{S:triangulation}
to show that on this subspace all three notions of convergence are actually equivalent and define the same
topology (Theorem~\ref{t:topeq}).
This topology turns the subspace of binary algebraic measure trees into a \emph{compact, metrizable} space,
which in particular implies that it is a closed subset of the space of algebraic measure trees.
In this section, we also finish the proof of Theorem~\ref{t:tree} by showing continuity of the coding map.

In Section~\ref{s:examples}, we consider the example of the continuum limits of sampling consistent families of
random trees and illustrate it with the example of so-called $\beta$-splitting trees introduced in
\cite{Aldous1996}. This family includes the uniform binary tree (converging to the Brownian CRT) and the Yule
tree (aka Kingman tree or random binary search tree).

\section{Algebraic trees}
\label{S:trees}
In this section we introduce algebraic trees. In Subsection~\ref{sub:trees}
we formalize the ``tree structure'' common to both graph-theoretic trees and metric trees by a function that
maps every triplet of points in the tree to the corresponding branch point.  We show that the set of defining properties is rich enough to obtain known concepts such as leaves, branch points, degree, edges,
intervals, subtrees spanned by a set, discrete and continuum trees, etc. In Subsection~\ref{s:morphisms} we introduce the notion of structure preserving morphisms.
In Subsection~\ref{sub:astopological}
we equip algebraic trees with a canonical Hausdorff topology. We also characterize compactness and a concept we
call order separability, which is closely related to second countability of the topology.
Finally, in Subsection~\ref{sub:asRtree}, we show that any order
separable algebraic tree is induced by a metric tree (which is not true without order separability), and
establish the condition under which this metric tree can be chosen to be a compact $\R$-tree.

\subsection{The branch point map}
\label{sub:trees}
In this subsection we introduce algebraic trees. Recall from Definition~\ref{d:metrictree} the definition of a metric tree, and
the properties (BPM1)--(BPM4) of the map which sends a triplet of $3$ points in a metric tree to its branch point.

\begin{definition}[algebraic trees] \label{d:algebraic2}
An \define{algebraic tree} $(T,c)$ consists of a set $T\not=\emptyset$ and
a branch point map $c\colon T^3\to T$ satisfying (BPM1)--(BPM4).
\end{definition}

The following useful property reflects the fact that any four points in an algebraic tree can be associated with
a shape as illustrated in Figure~\ref{f:4pt} above.

\begin{lemma}[a consequence of (BPM4)]
Let\/ $(T,c)$ be an algebraic tree. Then for all\/ $x_1,x_2,x_3,x_4\in T$
the following hold:
\begin{enumerate}
\item\label{i:4ptsplit} If\/ $c(x_1,x_2,x_3)=c(x_1,x_2,x_4)$, then\/ $c(x_1,x_3,x_4)=c(x_2,x_3,x_4)$.
\item If\/ $c(x_1,x_2,x_3)=c(x_1,x_2,x_4)$, then\/ $c(x_1,x_2,x_3)=c(x_1,x_2,c(x_1,x_3,x_4))$.
\end{enumerate}
\label{l:4pt}
\end{lemma}

\begin{proof}
Let $x_1,x_2,x_3,x_4\in T$ with $c_1:=c(x_1,x_2,x_3)=c(x_1,x_2,x_4)$, and $c_2:=c(x_1,x_3,x_4)$.

\smallskip
\emph{(i) }
Condition (BPM4) implies that
\begin{equation}
  c_2\in \big\{c_1=c(x_1,x_3,x_2),\, c(x_2,x_3,x_4),\, c_1=c(x_1,x_2,x_4)\big\}.
\end{equation}
Thus $c_1=c_2$, or $c_2=c(x_2,x_3,x_4)$. The second case is the claim.
In the first case, we apply Condition (BPM4) once more to find that
\begin{equation}\label{e:006}
  c(x_2,x_3,x_4)
  \in
  \big\{c_1=c(x_1,x_2,x_3),\, c_2=c(x_1,x_3,x_4),\, c_1=c(x_1,x_2,x_4)\big\}
  =\{c_1,c_2\}=\{c_2\},
\end{equation}
so that the claim also holds in this case.

\smallbreak
\emph{(ii) } Condition~(BPM3) implies that
\begin{equation}
   c(x_1,x_3,c_2) = c\big(x_1,\, x_3,\, c(x_1,x_3,x_4)\big)=c(x_1,x_3,x_4)=c_2,
\end{equation}
and similarly also $c(x_2,x_3,c_2)=c(x_2,x_3,x_4)=c_2$.
Now part \emph{\ref{i:4ptsplit}} with $x_4$ replaced by $c_2$ yields $c(x_1,x_2,x_3)=c(x_1,x_2,c_2)$ as claimed.
\end{proof}

We have seen that the four axiomatizing properties of the branch point map are
necessary. In many respects they are also sufficient to capture the tree structure.
For example, in analogy to \eqref{e:arc} we can
define for each $x,y\in T$ the \emph{interval} $[x,y]$ by
\begin{equation}\label{e:path}
	[x,y]:=\big\{w\in T:\,c(x,y,w)=w\big\}.
\end{equation}
We also use the notation $\openint xy := [x,y] \setminus\{x,y\}$, and similarly $\lopenint xy$, $\ropenint xy$.
The following properties of intervals are known to hold in \Rtrees\ (compare, e.g., to
\cite[Chapter~2]{Chiswell2001} or \cite[Chapter~3]{Evans2008}):

\begin{lemma}[properties of intervals]
Let\/ $(T,c)$ be an algebraic tree. Then the following hold:
\begin{enumerate}
\item If\/ $x,v,w,z\in T$ are such that\/ $w\in[x,z]$ and\/ $v\in[x,w]$, then\/ $v\in[x,z]$.
\item If\/ $x,y,z\in T$, then
\begin{equation} \label{e:010}
   [x,y] \cap [y,z] = \bigl[c(x,y,z),\, y\bigr].
\end{equation}

In particular,
\begin{equation}
\label{e:011}
   \bigl[x,\,c(x,y,z)\bigr]\cap \bigl[c(x,y,z),\,z\bigr] = \bigl\{c(x,y,z)\bigr\}.
\end{equation}
\item If\/ $x,y,z\in T$, then
\begin{equation}
\label{e:012}
   [x,y] \cup [y,z] = [x,z] \uplus \blopenint{c(x,y,z)}{y}.
\end{equation}

In particular,
\begin{equation}
\label{e:013}
   [x,y] \cup [y,z] = [x,z]  \quad\mbox{iff}\quad  y\in[x,z].
\end{equation}
\item For all $x,y,z\in T$,
\begin{equation} \label{e:014}
	[x,y]\cap [y,z]\cap [z,x]  =  \bigl\{c(x,y,z)\bigr\}.
\end{equation}
\end{enumerate}
\label{l:intprop}
\end{lemma}

\begin{proof} \emph{(i) } Let $x,v,w,z\in T$ with $w=c(x,w,z)$ and $v=c(x,v,w)$. Then by Condition~(BPM4),
\begin{equation}
\label{e:009}
  c(x,v,z) \in \bigl\{c(x,v,w),\,w=c(x,w,z),\,c(v,w,z)\bigr\}.
\end{equation}
We discuss the three cases separately.
If $c(x,v,z)=c(x,v,w)$, then $c(v,w,z)=c(x,w,z)=w$ by Lemma~\ref{l:4pt}(i).
It then follows that $c(x,v,z)=c(x,v,c(x,w,z))=c(x,v,w)=v$ by Lemma~\ref{l:4pt}(ii), which gives the claim in this case.

If $c(x,v,z)=w$ then $v=c(v,w,x)=c(v,w,z)$  by  Lemma~\ref{l:4pt}(i). It then follows that $c(x,v,z)=c(x,z,c(z,w,v))=c(x,v,v)=v$ by Lemma~\ref{l:4pt}(ii), which gives the claim in this case.

If $c(x,v,z)=c(v,w,z)$ then $v=c(x,w,v)=c(x,w,z)=w$ by  Lemma~\ref{l:4pt}(i).
Thus $v=w\in[x,z]$, and the claim holds also in this case.

\smallbreak
\emph{(ii) }Let $x,y,z\in T$, and $v\in[x,y]\cap[y,z]$. That is,
$v=c(x,v,y)=c(y,v,z)$. It follows from Lemma~\ref{l:4pt}(i) that
$c(x,z,v)=c(x,z,y)$, and then from Lemma~\ref{l:4pt}(ii) together with Condition~(BPM2) that
\begin{equation}
\label{e:008}
   v=c(x,v,y)=c\big(v,y,c(y,x,z)\big).
\end{equation}
Equivalently, $v\in[c(x,y,z),y]$. This proves the inclusion $[x,y]\cap[y,z]\subseteq[c(x,y,z),y]$. The other inclusion follows from (i).

Notice that \eqref{e:011} follows from \eqref{e:010} with the special choice $y=c(x,y,z)$.

\smallbreak
\emph{(iii) } Notice first that it follows immediately from (i) that the union on the right hand side is disjoint.
We claim that
	\begin{equation}
    \label{e:sclaim}
		[x,z] \subseteq [x,y] \cup [y,z].
	\end{equation}
	Indeed, let $v\in [x,z]$, i.e.\ $c(x,z,v)=v$. Then by (BPM4) applied to $\{v,x,y,z\}$,
\begin{equation}
\label{e:004}
   v=c(x,z,v) \in \big\{c(x,y,v),\, c(x,y,z),\, c(y,z,v)\big\},
\end{equation}
which implies that $v\in[x,y]$ (if $v=c(x,y,v)$) or $v\in[x,z]\cap[x,y]$ (if $v=c(x,y,z)$) or $v\in[y,z]$ (if $v=c(y,z,v)$).
Second, we claim that for all $x,y,z\in T$,
\begin{equation}
	[x,z] \cup \bigl[c(x,y,z),\,y\bigr] \subseteq [x,y] \cup [y,z].
\end{equation}
To see this, recall from (ii) that $[c(x,y,z),y]= [x,y]\cap[z,y]\subseteq [x,y]\cap[z,y]$.
As $[x,c(x,y,z)]\subseteq[x,y]$ by (i), we have $[x,y]\subseteq [x,c(x,y,z)]\cup
[c(x,y,z),y]\subseteq[x,y]\uplus\lopenint{c(x,y,z)}{y}$. The corresponding inclusion for $[y,z]$
is shown in the same way, and we have proven Equation \eqref{e:012}.

\smallbreak
\emph{(iv) } This follows immediately from (ii).
\end{proof}

We say that $\{x,y\}\subseteq T$ with $x\not=y$ is an \emph{edge} of $(T,c)$ if and only if there is ``nothing
in between'', i.e.\ $[x,y]=\{x,y\}$, and denote by
\begin{equation}\label{e:ed}
	\edge(T,c) := \bset{\{x,y\}\subseteq T}{x\ne y,\;[x,y] = \{x,y\}}
\end{equation}
the \emph{set of edges}. The following example explains that there is no need to distinguish between finite
algebraic trees and graph-theoretical trees, and the definitions of edges are consistent.

\begin{example}[finite algebraic trees correspond to graph-theoretic trees]
	Finite algebraic trees are in one to one correspondence with finite (undirected) graph-theoretic trees.
	Let $(T,E)$ be a graph-theoretic tree with vertex set $T$ and edge set $E$. Then $(T,E)$ corresponds to
	the algebraic tree $(T,c_E)$ with $c_E(u,v,w)$ defined as the unique vertex that is on the
	(graph-theoretic) path between any two of $u,v,w$. Conversely, if $(T,c)$ is an algebraic tree with $T$
	finite, then $(T,c)$  corresponds to the graph-theoretic tree $(T,E_c)$ with $E_c:=\edge(T,c)$.
	Obviously, $c_{E_c}=c$.
\label{ex:finite}
\end{example}

For a graph-theoretic tree $(T,E)$, we can allow the vertex set $T$ to be countably infinite, and still obtain a
corresponding algebraic tree as in the previous example.
Note, however, that countable algebraic trees do not necessarily correspond to graph-theoretic trees. Indeed, it
is possible that $T$ is countably infinite and $\edge(T,c)=\emptyset$. This can be seen by taking $T=\Q$ in the
following example, which shows that every totally ordered space naturally corresponds to an algebraic tree.

\begin{example}[totally ordered spaces as algebraic trees]\label{ex:totord}
	For a  totally ordered space $(T,\le)$, define $c_\le(x,y,z) := y$ whenever $x\le y \le z$, ($x,y,z\in
	T$). Then it is trivial to check that $(T,c_\le)$ is an algebraic tree and the interval $[x,y]$
	coincides with the order interval $\set{z\in T}{x\le z \le y}$.
\end{example}

Conversely, given an algebraic tree $(T,c)$ and a distinguished point $\rho$ (often referred to as \emph{root}), we can
define a \emph{partial order} $\le_\rho$ by letting for $x,y\in T$,
\begin{equation}
\label{e:partialrho}
  x\le_\rho y\quad\mbox{ iff }x\in[\rho,y].
\end{equation}
Partial orders provide an equivalent way of defining algebraic trees.

\begin{proposition}[algebraic trees and semi-lattices]\label{p:semilat}
\begin{enumerate}
\item Let\/ $(T,c)$ be an algebraic tree, and\/ $\rho\in T$. Then\/ $(T,\le_\rho)$ is a partially ordered
	set, and a meet semi-lattice with infimum
\begin{equation}\label{e:inf}
	x\land y = c(\rho,x,y) \qquad \forall x,y \in T.
\end{equation}
Furthermore, $\le_\rho$ is a total order on\/ $[\rho,x]$ for all\/ $x\in T$.
\item Let\/ $(T,\le)$ be a partially ordered set, such that all initial segments\/ $\set{y\in T}{y\le x}$, $x\in
	T$, are totally ordered (in particular a meet semi-lattice). Then for\/ $x,y,z\in T$
\begin{equation}\label{e:cfrominf}
	c(x,y,z) := \max\{x\land y,\, y\land z,\, z\land x\}
\end{equation}
	exists, and\/ $(T,c)$ is an algebraic tree.
\end{enumerate}
\end{proposition}

\begin{proof}\proofcase{(i)}
Let $x,y\in T$ with $x\le_\rho y$ and $y\le_\rho x$. That is, $x=c(\rho,x,y)$ and $y=c(\rho,y,x)$ which implies that $x=y$, and proves that $\le_\rho$ is \emph{antisymmetric}.
As $x=c(\rho,x,x)$, $x\le_\rho x$ which proves that $\le_\rho$ is \emph{reflexive}. Finally, to show \emph{transitivity}, let $x,y,z\in T$ with $x\le_\rho y$ and $y\le_\rho z$. That is $x\in[\rho,y]$ and $y\in[\rho,z]$, which implies that $x\in[\rho,z]$ by Lemma~\ref{l:intprop}(i). Equivalently, $x\le_\rho z$
which proves the transience, and thus that $\le_\rho$ is a partial order.

For the \emph{infimum}, notice that $v\le_\rho x$ and $v\le_\rho y$ if and only if $v\in[\rho,x]\cap[\rho,y]$, or equivalently by Lemma~\ref{l:intprop}(ii), $v\in[\rho,c(\rho,x,y)]$.
As for all $v\in[\rho,c(\rho,x,y)]$ we have $v\le c(\rho,x,y)$, the claim \eqref{e:inf} follows.

Fix $x\in T$. For \emph{totality} on $[\rho,x]$, let $v,w\in[\rho,x]$, i.e., $v=c(\rho,v,x)$ and $w=c(\rho,w,x)$. Applying Condition (BPM4) to $\{\rho,v,w,x\}$ we find that one of the following three cases must occur: $c(\rho,v,w)=c(\rho,v,x)$ (which implies that $v=c(\rho,v,w)$, or equivalently, $v\le_\rho w$), $c(\rho,w,v)=c(\rho,w,x)$ (which implies that $w=c(\rho,w,v)$, or equivalently, $w\le_\rho v$), or $c(\rho,x,v)=c(\rho,x,w)$ (which implies that $w=v$).

\proofcase{(ii)} The maximum in \eqref{e:cfrominf} is over a totally ordered set (because initial segments are
totally ordered), thus exists. Furthermore, if $x\land y\le x\land z \le y\land z$, say, we also obtain $x\land y =
x\land y \land z \ge x\land z$. This means that (at least) two of $x\land y$, $y\land z$, $z\land x$ are
identical, and the maximum $c(x,y,z)$ is the third one. That $v$ satisfies (BPM1)--(BPM3) is obvious. To see the
4-point condition (BPM4), let $x_1, \ldots, x_4\in T$ and assume w.l.o.g.\ that $x_2\land x_3=c(x_1, x_2, x_3)=:v$, and hence $x_1\land x_2 = x_1\land x_3 \le v$.
We distinguish cases: If $x_2\land x_4 <v$, then $c(x_2, x_3, x_4)=\max\{x_2\land x_4, v, x_3\land x_4\} = v$,
and \eqref{e:BPM4} is satisfied.
Otherwise, $x_2\land x_4, x_3\land x_4 \ge v$ and at most one of them can be strictly larger.
If $x_2 \land x_4 > v \ge x_1 \land x_2$, then $x_1\land x_2 = x_1 \land x_4 = x_1\land x_3$,
and $c(x_1, x_3, x_4)=x_3\land x_4 = v$. The case $x_3 \land x_4 >v$ is analogous.
In the last case $x_2\land x_4 = x_3 \land x_4 = v$, and $c(x_2, x_3, x_4)=v$.
\end{proof}

\begin{remark}[Favre and Jonsson's nonmetric trees]
	In \cite{FavreJonsson04}, \emph{rooted nonmetric trees} are introduced as partially ordered sets with
	global minimum, totally ordered initial segments, and the additional property that all full, totally ordered
	subsets are order isomorphic to a real interval. Proposition~\ref{p:semilat} shows that they naturally
	induce algebraic trees.
\end{remark}

\begin{cor} Let\/ $(T,c)$ be an algebraic tree, and\/ $\rho,x,y\in T$. If\/ $v\in[x,y]$, then\/ $v\ge_\rho c(x,y,\rho)$.
\label{l:vgec}
\end{cor}

\begin{proof} Let $\rho,x,y\in T$ and $v\in[x,y]$. That is, $v=c(x,v,y)$.
We need to show that $c(\rho,v,c(\rho,x,y))=c(\rho,x,y)$.

By Condition~(BPM4) applied to $\{x,y,\rho,v\}$ we have one of the following three cases:
$c(x,y,\rho)=c(x,y,v)$ (in which case $c(\rho,x,y)=v$) or $c(\rho,y,x)=c(\rho,y,v)$ (in which case $c(x,v,\rho)=c(x,v,y)=v$ by Lemma~\ref{l:4pt}(i) and thus $v\in[\rho,x]$; the claim then follows since this implies that $v\in[\rho,x]\cap[x,y]=[c(\rho,x,y),y]$ by Lemma~\ref{l:intprop}(ii)), or $c(\rho,x,y)=c(\rho,x,v)$ (in which we conclude similar to the second case that $v\in[c(\rho,x,y),x]$).
\end{proof}

The partial orders $\le_\rho$ allow us to define a notion of completeness of algebraic trees.
\begin{definition}[directed order completeness]
	Let $(T,c)$ be an algebraic tree. We call $(T,r)$ \define{(directed) order complete} if for all $\rho\in
	T$ the supremum of every totally ordered, non-empty subset exists in the partially ordered set
	$(T,\le_\rho)$.
\label{d:ordercomplete}
\end{definition}

Obviously, in an order complete algebraic tree, infima of totally ordered sets exists, because they are
either $\rho$ if the set is empty or a non-empty supremum w.r.t.\ a different root.
This notion of completeness allows us to define the analogs of complete $\R$-trees.

\begin{definition}[algebraic continuum tree] \label{d:aCT}
We call an algebraic tree $(T,c)$ \define{algebraic continuum tree} if the following two conditions hold:
\begin{enumerate}[\axiom({A}CT1)]
\item $(T,c)$ is order complete.
\item $\edge(T,c)=\emptyset$.
\end{enumerate}
\end{definition}

\subsection{Morphisms of algebraic trees} \label{s:morphisms}

Like any decent algebraic structure (or in fact mathematical structure), algebraic trees come with a notion of
structure-preserving morphisms.

\begin{definition}[morphisms]
Let $(T,c)$ and $(\Th,\ch)$ be algebraic trees. A map $f\colon T\to \Th$
is called a \define{tree homomorphism} (from $T$ into $\Th$) if for all $x,y,z\in T$,
	\begin{equation}
\label{e:treehom}
		f\bigl(c(x,y,z)\bigr) = \ch\big(f(x),f(y),f(z)\big).
	\end{equation}
We refer to a bijective tree homomorphism as \define{tree isomorphism}.
\label{d:treehom}
\end{definition}

As we have seen, the tree structure can be expressed also in terms of
intervals or partial orders rather than the branch point map. This also works for the morphisms.
\begin{lemma}[equivalent definitions of morphisms]
	Let\/ $(T,c)$ and\/ $(\Th,\ch)$ be algebraic trees, and\/ $f\colon T\to\Th$. Then
	the following are equivalent:
	\begin{enumerate}[1.]
		\item\label{i:homom} $f$ is a tree homomorphism.
		\item\label{i:order} For all\/ $\rho\in T$, $f$ is an order preserving map from\/ $(T,\le_\rho)$ to\/ $(\Th,\le_{f(\rho)})$.
		\item\label{i:interval} For all\/ $x,y\in T$, $f([x,y])\subseteq[f(x),f(y)]$.
	\end{enumerate}
\label{l:treehom}
\end{lemma}

\begin{proof}
\istep{\ref{i:homom}}{\ref{i:order}} Let $x,y,\rho\in T$ with
$x\le_\rho y$. Then $x=c(\rho,x,y)$ and thus
	$f(x)=\ch(f(\rho),f(x),f(y))$. Therefore $f(x)\le_{f(\rho)} f(y)$.

\istep{\ref{i:order}}{\ref{i:interval}} Let $x,y,z\in T$ with $z\in [x,y]$. Then $z\le_x y$ and thus $f(z) \le_{f(x)} f(y)$, i.e.\ $f(z) \in [f(x),f(y)]$.
\istep{\ref{i:interval}}{\ref{i:homom}} Let $x,y,z\in T$. Then $\{c(x,y,z)\}=[x,y]\cap[x,z]\cap[y,z]$.  Hence
\begin{equation}
   \big\{f(c(x,y,z))\big\}\subseteq \big[f(x),f(y)\big]\cap \big[f(y),f(z)\big]\cap \big[f(x),f(z)\big]
   	= \bigl\{ \ch\(f(x),f(y),f(z)\)\bigr\}.
\end{equation}
Therefore, $f(c(x,y,z))=\ch(f(x),f(y),f(z))$.
\end{proof}

Lemma~\ref{l:treehom} shows that our notion of morphisms of algebraic trees is weaker than the morphisms of
nonmetric trees used in \cite{FavreJonsson04}, but the notion of isomorphism is the same.
The image of an algebraic tree under a tree homomorphism is a subtree in the following sense.
\begin{definition}[subtree]
Let $(T,c)$ be an algebraic tree, and $\emptyset\ne A\subseteq T$.
$A$ is called a \define{subtree (of\/ $(T,c)$)} if
\begin{equation}
\label{e:subtree}
    c(A^3)\subseteq A.
\end{equation}
We refer to $c(A^3)$ as the \define{algebraic subtree generated by\/ $A$}.
\label{d:subtree}
\end{definition}
Obviously, a subtree $A$ of $(T,c)$, implicitly equipped with the restriction of $c$ to $A^3$, is an algebraic
tree in its own right. Furthermore, the following lemma is easy to check.

\begin{lemma}[tree homomorphisms]
Let\/ $(T,c)$ and\/ $(\Th,\ch)$ be two algebraic trees, and\/ $f\colon T \to \Th$ a homomorphism.
Then the image\/ $f(T)$ is a subtree of\/ $\Th$.
If\/ $f$ is injective, $f^{-1}$ is a tree homomorphism from\/ $f(A)$ into\/ $T$.

In particular, if\/ $(\tilde{T},\tilde{c})$ is another algebraic tree, and\/ $g$ is a homomorphism form\/
$(\Th,\ch)$ to\/ $(\tilde{T},\tilde{c})$, then\/ $g\circ f$ is a homomorphism from\/ $(T,c)$ to\/
$(\tilde{T},c_{\tilde{T}})$.
\label{l:aa}	
\end{lemma}

\subsection{Algebraic trees as topological spaces}
\label{sub:astopological}
In contrast to metric trees, there is a priori no topology defined on a given algebraic tree.
In this section, we therefore equip algebraic trees with a canonical topology.

For each $x\in T$, we introduce a (component) relation $\sim_x$ by letting $y\sim_x z$
if and only if $x\not\in[y,z]$, where $y,z\in T$. Let for each $y\in T\setminus\{x\}$
\begin{equation} \label{e:005}
  \Sub_x(y) = \Sub^{(T,c)}_x(y) := \bset{z\in T\setminus \{x\}}{z\sim_x y}
\end{equation}
be the equivalence class of $T\setminus\{x\}$ containing $y$, and
note that $\Sub_x(y)$ is a subtree for all $x,y\in T$, and $\Sub_x(y)=\Sub_x(z)$ whenever $z\in
\Sub_x(y)$.
We refer to $\Sub_x(y)$ as the \emph{component} of $T\setminus x$ containing $y$.
Now and in the following, we equip $(T,c)$ with the topology
\begin{equation}
\label{e:tau}
   \tau := \tau\(\bset{\Sub_x(y)}{x,y\in T,\; x\ne y}\)
\end{equation}
generated by the set of components, i.e.\ with the coarsest topology such that all components are open sets. We
call $\tau$ the \define{component topology} of $(T,c)$.

\begin{example}[on totally ordered trees, $\tau$ is the order topology]
	If $(T,\le)$ is a totally ordered space, and $(T,c_\le)$ the corresponding algebraic tree as in
	Example~\ref{ex:totord}, then $\tau$ coincides with the \emph{order topology}
	(i.e.\ the one generated by sets of the form $\set{y\in T}{y>x}$ and $\set{y\in T}{y<x}$ for $x\in T$).
\end{example}

\begin{example}[intervals are closed sets] \label{exp:002}
Let $(T,c)$ be an algebraic tree, and $x,y\in T$. Then
\begin{equation}
\label{e:030}
   T\setminus [x,y] = \bigcup\bset{\Sub_u(v)}{u\in[x,y],\,v\in T,\, \Sub_u(v)\cap[x,y]=\emptyset}\in\tau.
\end{equation}
This means that $[x,y]$ is closed in the component topology $\tau$.
\end{example}

\begin{lemma}\label{l:ccont}
	Let\/ $(T,c)$ be an algebraic tree. Then\/ $c$ is continuous w.r.t.\ the component topology\/ $\tau$.
\end{lemma}
\begin{proof}
	By definition of $\tau$, it is sufficient to show that for any $x,y\in T$, $x\ne y$, the set
	$c^{-1}(\Sub_x(y))$ is open in $T^3$. By definition of $\Sub_x(y)$ and the property $c(u,v,w)\in
	[u,v]\cap [v,w]\cap [w,u]$ shown in Lemma~\ref{l:intprop}, $c(u,v,w) \in \Sub_x(y)$ if and only if (at
	least) two of $u,v,w$ are in $\Sub_x(y)$. Because $\Sub_x(y)$ is open, the same is true for
	$\set{(u,v,w)\in T^3}{u,v\in \Sub_x(y)}$ in the product topology. Hence $c^{-1}(\Sub_x(y))$ is a union
	of open set and thus open.
\end{proof}

Next, we show that $\tau$ is a Hausdorff topology and characterize compactness of algebraic trees in this
topology.

\begin{lemma}[$\tau$ is Hausdorff]\label{l:Hausdorff}
	Let\/ $(T,c)$ be an algebraic tree. Then the component topology\/ $\tau$ defined in \eqref{e:tau} is a
	Hausdorff topology on\/ $T$.
\end{lemma}
\begin{proof}
	To show that $(T,\tau)$ is Hausdorff, let $x,y\in T$ be distinct. If $\Sub_y(x) \cap \Sub_x(y) =
	\emptyset$, then $\Sub_y(x)$ and $\Sub_x(y)$ are clearly disjoint neighbourhoods of $x$ and $y$,
	respectively. Assume that this is not the case,  and choose $z\in \Sub_x(y)\cap \Sub_y(x)$.  Then
	$\rho:=c(x,y,z)\not\in\{x,y\}$. Furthermore, $c(x,\rho,y)=c(x,y,z)=\rho$, and hence $x\not\sim_\rho y$.
	Thus $\Sub_\rho(x)$ and $\Sub_\rho(y)$ are disjoint neighbourhoods of $x$ and $y$, respectively.
	Hence $\tau$ is Hausdorff.
\end{proof}

\begin{proposition}[characterizing compactness]\label{p:compact}
	Let\/ $(T,c)$ be an algebraic tree with component topology $\tau$.
	Then\/ $(T,\tau)$ is compact if and only if\/ $(T,c)$ is directed order complete.
\end{proposition}

\begin{proof}
\pstep{``only if''} Assume first that $(T,c)$ is not order complete. Then we can choose $\rho\in T$ and
	$\emptyset\ne A\subseteq T$ such that $A$ is totally ordered w.r.t.\ $\le_\rho$ but does not have a
	supremum in $(T,\le_\rho)$.
	For $x,y\in T$, let $U_x:= \set{z\in T}{z\not\ge_\rho x}$ and $V_y:=\set{z\in T}{z>_\rho y}$. Then $U_x$
	and $V_y$ are open sets.
	We claim that $\U:=\set{U_x}{x\in A} \cup \set{V_y}{y\ge A}$ is an open cover of $T$. Indeed, if
	$z\ge_\rho A$, then, because $A$ has no supremum, there is $y\in T$ with $A\le_\rho y \le_\rho z$, hence
	$z\in V_y \in \U$. Otherwise, if $z\not\ge_\rho A$, there is $x\in A$ with $z\in U_x \in \U$. Thus $\U$
	is a cover of $T$.

	$\U$ has no finite sub-cover, because if $\U' = \{U_{x_1},...,U_{x_n},V_{y_1},...,V_{y_m}\}$
	were such a finite sub-cover, then $\{U_{x_1},...,U_{x_n}\}$ would cover $A$. This, however, would imply that
	$\max\{x_1,...,x_n\}$ would be a supremum of $A$, contradicting our assumption. Hence $(T,\tau)$ is not
	compact.

\pstep{``if''} Assume that $(T,c)$ is order complete. Consider a cover $\U$ of $T$ with components,
	i.e.\ $\U\subseteq \set{\Sub_y(x)}{x,y\in T,\, x\ne y}$.  By the Alexander subbase theorem, for
	compactness of $\tau$, it is sufficient to show that $\U$ has a finite sub-cover.

	To this end, fix an element $\rho\in T$ and consider the set $\U_\rho:=\set{U\in \U}{\rho \in
	U}\not=\emptyset$. By Hausdorff's maximal chain theorem (or Zorn's lemma), there is a maximal chain $I$
	in the partially ordered set $(\U_\rho, \subseteq)$.
	For every $U\in I$, we have $\rho\in U$, and thus there is $x_U\in T$ such that $U=\Sub_{x_U}({\rho})$.
	We claim that $U\subseteq V$ implies $x_U \le_\rho x_V$. Indeed, $x_V\not\in V$
	and hence $x_V\not\in U$ which is equivalent to $x_V \ge_\rho x_U$. Therefore, $z:= \sup \set{x_U}{U\in I}$
	exists in $(T,\le_\rho)$ by directed order completeness of $T$.
Because $\U$ is a cover, there is
	$V\in \U$ with $z\in V$, hence $V=\Sub_y(z)$ for some $y\in T$.
Because $V\not\in I$ and $I$ is a maximal chain, $y\not\ge_\rho z$.
Hence there is $U\in I$ with
	$y\not\ge_\rho x_U=:x$. We claim that $T=\Sub_y(z)\cup \Sub_x(\rho)$.
	Indeed, let $w\in T\setminus \Sub_x(\rho)$. Then $w\ge_\rho x$. Using $z\ge_\rho x$ and $c(w,z,y)\in [w,z]$,
	we obtain $c(w,z,y) \ge_\rho x$, and hence $c(w,z,y) \ne y$. Thus $w\in \Sub_y(z)$ as claimed, and
	$\{\Sub_y(z),\,\Sub_x(\rho)\}$ is the desired sub-cover.
\end{proof}

It turns out that the following separability condition, which we call order separability,
is crucial for us.

\begin{proposition}[order separability]\label{p:separable}
	Let\/ $(T,c)$ be an algebraic tree with component topology $\tau$.
	Then the following are equivalent:
\begin{enumerate}[1.]
\item\label{i:dense} There exists a countable set\/ $D$ such that for all\/ $x,y\in T$ with\/ $x\not=y$,
	\begin{equation} \label{e:sep}
		D\cap \ropenint{x}{y} \ne \emptyset.
	\end{equation}
\item\label{i:A2} The topological space\/ $(T,\tau)$ is second countable (i.e.\ $\tau$ has a countable base),
	and\/ $\edge(T,c)$ is countable.
\item\label{i:sepa} The topological space\/ $(T,\tau)$ is separable, and\/  $\edge(T,c)$ is countable.
\end{enumerate}
\end{proposition}

\begin{proof}
\istep{\ref{i:sepa}}{\ref{i:dense}} Assume that $\edge(T,c)$ is countable, and
that $(T,\tau)$ is separable. Then there exists a countable, dense subset $\tilde{D}\subseteq T$. We claim that
\begin{equation}\label{e:orderdense}
   D:=c\big(\tilde{D}^3\big)\cup\big\{z\in T:\,\exists\,x\in T\mbox{ such that }\{x,z\}\in\edge(T,c)\big\}
\end{equation}
satisfies   \eqref{e:sep}. Indeed, $D$ is countable by assumption. Moreover,
let $x,y\in T$. Then two cases are possible:
either $\Sub_x(y)\cap\Sub_y(x)=\emptyset$. In this case, $\{x,y\}\in\edge(T,c)$, which implies that
$\ropenint{x}{y} \cap D\not=\emptyset$.
Or $\Sub_x(y)\cap\Sub_y(x)\not=\emptyset$.
In this case, as $\Sub_x(y)\cap\Sub_y(x)$ is open by definition of $\tau$, there is $z\in
\tilde{D}\cap \Sub_x(y)\cap\Sub_y(x)$. Let $v:=c(x,y,z)$. Then $v\in\openint{x}{y}$, and either
$v=z\in D$, or the three components $\Sub_v(x)$, $\Sub_v(y)$, $\Sub_v(z)$ are distinct. In the second case, we
can choose $x'\in \tilde{D} \cap \Sub_v(x)$ and $y'\in \tilde{D}\cap \Sub_v(y)$ to see that $v=c(x',y',z) \in
D$. In any case, $v\in \ropenint{x}{y} \cap D$.
\istep{\ref{i:dense}}{\ref{i:A2}}
Let $D$ be a countable set satisfying \eqref{e:sep}. Then for all $\{x,y\}\in\edge(T,c)$, $D\cap
\ropenint{x}{y}=\{x\}$. This implies that  $\edge(T,c)$ is countable.
We consider the countable set $\U=\bset{\Sub_v(u)}{u,v\in D} \subseteq \tau$ and claim that it is a subbase for
$\tau$ (i.e.\ generates $\tau$). To this end, let $x,y\in T$. We show that $U:=\Sub_x(y)$ is a union of sets
from $\U$, i.e.\ for every $z\in U$ we construct $V\in\U$ with $z\in V \subseteq U$. By assumption on $D$, there
is $v\in D \cap \ropenint{x}{z}$ and $u\in D \cap \lopenint{v}{z}$. Let $V:=\Sub_v(u) \in \U$. Because
$c(u,v,z)=u\ne v$, we have $z\in V$. Let $w\in T\setminus U$. Because $u\in U$, we have $U=\Sub_x(u)$ and
therefore $x\in[u,w]$. Similarly, $x\in [v,w]$. In particular, by Lemma~\ref{l:4pt},
$c(u,v,w)=c(u,v,c(u,x,w))=c(u,v,x)=v$, and thus $w\not\in V$.
Because $w\in T\setminus U$ is arbitrary, we obtain $V\subseteq U$.
\istep{\ref{i:A2}}{\ref{i:sepa}} Trivial, because every second countable topological space is separable.
\end{proof}

\begin{definition}[order separability]\label{d:ordersep}
	We call an algebraic tree $(T,c)$ \define{order separable} if the equivalent conditions of
	Proposition~\ref{p:separable} are satisfied. We call a set $D\subseteq T$ \define{order dense} if it
	satisfies \eqref{e:sep}.
\end{definition}

\begin{example}[uncountable star tree] This example shows that in \eqref{e:sep}
we can not replace $\ropenint{x}{y}$  by $[x,y]$. Let $T:=\{0\}\cup [1,2]$ with $c(x,y,z):=0$ whenever $x,y,z\in
T$ are distinct. Then if $D\subseteq T$ is such that $D\cap\ropenint{x}{0}\not =\emptyset$ for all $x\in[1,2]$
then $[1,2]\subseteq D$, and thus $D$ is uncountable and $(T,c)$ not order separable.
On the other hand, $D:=\{0\}$ has the property that $D\cap[x,y]\not=\emptyset$ for all $x,y\in T$ with $x\not
=y$.
\label{exp:star}
\end{example}

An order complete, order separable algebraic tree is, in its component topology $\tau$, a compact, second
countable Hausdorff space by Propositions~\ref{p:compact} and \ref{p:separable}. In particular, it is metrizable.
In fact, order separability already implies metrizability, as we will see in Subsection~\ref{sub:asRtree}.
The following example shows that (topological) separability of $(T,\tau)$ alone, without requiring the number of
edges to be countable, is neither sufficient for order separability nor for metrizability of $(T,\tau)$.

\begin{example}[a continuum ladder]\label{ex:topnotalgsep}
	Let $T=[0,1]\times\{0,1\}$ with the lexicographic order $\le$ on $T$, and define the canonical branch
	point map $c_\le$ as in Example~\ref{ex:totord}. Then $\edge(T,c_\le) =
	\bset{\{(x,0),\,(x,1)\}}{x\in [0,1]}$ is uncountable, and hence $(T,c_\le)$ is not order separable.
	Because $(\Q\cap[0,1]) \times \{0,1\}$ is a countable dense set, $(T,\tau)$ is (topologically) separable.
	The topological subspace $[0,1] \times \{1\}$ is the \emph{Sorgenfrey line},
	which is known to be non-metrizable (see \cite[Counterexample~51]{SteenSeebach78}). Thus also $(T,\tau)$
	cannot be metrizable.
\end{example}

\begin{definition}[Borel $\sigma$-algebra $\B(T,c)$]
	Let $(T,c)$ be an algebraic tree. We denote the Borel $\sigma$-algebra of the component topology $\tau$
	by $\B(T,c)$ and call it \define{Borel\/ $\sigma$-algebra of\/ $(T,c)$}.
\end{definition}

In general, $\B(T,c)$ is not generated by the set of components. Order separability, however, is sufficient to
ensure this property because it implies second countability of the component topology.

\begin{cor}[Borel $\sigma$-algebra generated by components]\label{c:borel}
	Let\/ $(T,c)$ be an order separable algebraic tree, and\/ $D\subseteq T$ an order dense set.
	Then its Borel\/ $\sigma$-algebra is generated by the set of components indexed by\/ $D$, i.e.
	\begin{equation}
		\B(T,c) = \sigma\(\bset{\Sub_x(y)}{x,y\in D,\; x\ne y}\).
	\end{equation}
\end{cor}
\begin{proof}
	Define $\U:=\bset{\Sub_x(y)}{x,y\in D,\; x\ne y}$.
	By Proposition~\ref{p:separable}, $(T,\tau)$ is second countable. Hence $\B(T,c)$ is generated by any
	subbase of $\tau$. If $D$ is order dense, $\U$ is such a subbase as shown in the proof of
	Proposition~\ref{p:separable}.
\end{proof}

\subsection{Metric tree representations of algebraic trees}
\label{sub:asRtree}
In this subsection, we discuss the connection of metric trees with algebraic trees.
Let $(T,r)$ be a metric tree (recall from Definition~\ref{d:metrictree}). Then by (MT\ref{MT:cex}),
there exists to any three points $x_1,x_2,x_3\in T$ a unique branch point $c_{(T,r)}(x_1,x_2,x_3)$
satisfying \eqref{e:brapoi}. We refer to $(T,c_{(T,r)})$ as the algebraic tree \define{induced by} $(T,r)$, and
to $(T,r)$ as a \define{metric representation} of $(T,c_{(T,r)})$.

\begin{lemma}[the algebraic tree induced by a metric tree]
	Let\/ $(T,r)$ be a metric tree, and\/ $c_{(T,r)}$ the map which sends any three distinct points to their branch point.
	Then the following hold:
\begin{enumerate}
\item $(T,c_{(T,r)})$ is an algebraic tree.
\item  $(T,c_{(T,r)})$ is order separable if and only if\/ $(T,r)$ is separable.
\item $(T,c_{(T,r)})$ is directed order complete if and only if\/ $(T,r)$ is bounded and complete. In
	particular, it is an algebraic continuum tree if and only if\/ $(T,r)$ is a bounded, complete\/ $\R$-tree.
\end{enumerate}
\label{l:cinduce}
\end{lemma}
\begin{proof} \emph{(i) } It can be easily checked that $(T,c_{(T,r)})$ is an algebraic tree.

\smallbreak
\emph{(ii) } Let $(T,r)$ be separable. Then $\edge(T,c_{(T,r)})$ is countable. The topology induced by $r$ is obviously
stronger than the topology $\tau$ introduced in \eqref{e:tau}, hence $\tau$ is separable and therefore the algebraic tree
$(T,c_{(T,r)})$ is order separable. Conversely, if $(T,c_{(T,r)})$ is order separable, then any countable set $D$ satisfying
\eqref{e:sep} is also dense in $(T,r)$.

\smallbreak
\emph{(iii) }
Clearly, $(T,c_{(T,r)})$ admits suprema along any linearly ordered set with respect to some root if and only if $(T,r)$
is bounded and complete. The ``in particular'' follows because a complete metric tree $(T,r)$
is an $\R$-tree if and only if $\edge(T,r):=\edge(T,c_{(T,r)})=\emptyset$
(\cite[Remark~1.2]{AthreyaLohrWinter17}).
\end{proof}

Our first main result states that under the assumption of order separability any algebraic tree can be embedded
by an injective homomorphism into a compact \Rtree\ and hence is isomorphic to (the algebraic tree induced by) a
totally bounded metric tree.

\begin{theorem}[characterisation of order separable algebraic trees]
Let\/ $T$ be a set, $c\colon T^3 \to T$.
\begin{enumerate}
\item\label{i:converse} $(T,c)$ is an order separable algebraic continuum tree if and only if there
	exists a metric\/ $r$ on\/ $T$ such that\/ $(T,r)$ is a compact\/ \Rtree\ with
	\begin{equation}
	\label{e:cTrTc}
	  c=c_{(T,r)}.
	\end{equation}
\item $(T,c)$ is an order separable algebraic tree if and only if there is an order separable algebraic continuum
	tree\/ $(\Tb,\cb)$ such that\/ $(T,c)$ is a subtree of\/ $(\Tb,\cb)$. In particular, every order
	separable algebraic tree is induced by a totally bounded metric tree.
\end{enumerate}\label{t:algtreechar}
\end{theorem}

The separability hypothesis in Theorem~\ref{t:algtreechar} is crucial and cannot be dropped.
In Example~\ref{ex:topnotalgsep}, we have already seen an algebraic tree where the component topology $\tau$ is
not metrizable. Moreover, in this example, $\tau$ coincides with the order topology which is also the case for the
metric topology of any metric tree without branch points. Thus the algebraic tree cannot be induced by a metric tree.
The following example shows that also algebraic continuum trees need not be induced by metric trees.

\begin{example}[algebraic continuum tree that is not induced by a metric tree]
	Let $T=[0,1]\times[0,1]$ with lexicographic order, and $(T,c)$ the corresponding algebraic tree as in
	Example~\ref{ex:totord}. It is easy to check that $(T,c)$ is an algebraic continuum tree.
	It cannot be induced by a metric tree because in its order topology $\tau$, it is connected but not
	path-wise connected. These two properties are equivalent for metric trees
	(see \cite[Theorem~2.20]{Evans2008}).
\label{exp:001}
\end{example}

In order to prove Theorem~\ref{t:algtreechar}, given an algebraic tree $(T,c)$, we need to provide a metric
$r$ such that \eqref{e:cTrTc} holds.
For that purpose, we consider for any measure $\nu$ on $(T,\B(T,c))$ such that $\nu$ is finite on
every interval, the following pseudometric,
\begin{equation}\label{e:rnu2}
	r_\nu(x,y) := \nu([x,y]) - \tfrac12\nu(\{x\}) - \tfrac12\nu(\{y\}), \quad x,y\in T.
\end{equation}
\begin{lemma}[$r_\nu$ is a pseudometric]
	Let\/ $(T,c)$ be an algebraic tree, and\/ $\nu$ a measure on\/ $(T,c)$ with\/
	$\nu([x,y])<\infty$ for all\/ $x,y\in T$. Then\/ $r_\nu$ is a pseudometric on\/ $T$.
\label{l:pseudo}
\end{lemma}

\begin{proof} By Lemma~\ref{l:intprop} for all $x,y,z\in T$,
\begin{equation}
\label{e:017}
\begin{aligned}
   \nu\big([x,y]\big) + \nu\big([y,z]\big)
   &=
   \nu\big([x,y] \cup [y,z]\big) + \nu\big([x,y]\cap [y,z]\big)
   \\
	&=
   \nu\big([x,z]\big) + \nu\big(\blopenint{c(x,y,z)}{y}\big)+\nu\big(\bigl[c(x,y,z),\,y\bigr]\big).
\end{aligned}
\end{equation}
Hence
\begin{equation}
\label{e:018}
\begin{aligned}
	\MoveEqLeft r_\nu(x,y)+\tfrac{1}{2}\nu\{x\}+\tfrac{1}{2}\nu\{y\}+r_\nu(y,z)
   +\tfrac{1}{2}\nu\{y\}+\tfrac{1}{2}\nu\{z\}
   \\
   &=
   r_\nu(x,z)+\tfrac{1}{2}\nu\{x\}+\tfrac{1}{2}\nu\{z\}
   +2r_\nu\big(c(x,y,z),\,y\big)+\nu\{c(x,y,z)\}+\nu\{y\}-\nu\{c(x,y,z)\},
\end{aligned}
\end{equation}
or equivalently,
\begin{equation}
\label{e:019}
\begin{aligned}
	r_\nu(x,y)+r_\nu(y,z)
   &=
   r_\nu(x,z)+2r_\nu\big(c(x,y,z),\,y\big).
\end{aligned}
\end{equation}
This implies that $r_\nu$ satisfies the triangle inequality.
\end{proof}

We denote the quotient metric space by $(T_\nu,r_\nu)$, i.e.\ $T_\nu$ is the set of equivalence classes of
points in $T$ with $r_\nu$-distance zero, and the quotient metric on $T_\nu$ is again denoted by $r_\nu$.
Furthermore, let $\pi_\nu \colon T \to T_\nu$ be the canonical projection.
\begin{lemma}[$(T_\nu,r_\nu)$ is a metric tree]
	Let\/ $(T,c)$ be an algebraic tree, and\/ $\nu$ a measure on\/ $(T,c)$ with\/
	$\nu([x,y])<\infty$ for all\/ $x,y\in T$.
	Then the quotient space\/ $(T_\nu, r_\nu)$ is a metric tree, and the canonical projection\/ $\pi_\nu$ is a tree
	homomorphism.
\label{l:o0hyper}
\end{lemma}

\begin{proof}
Let $x_1,\ldots,x_4\in T$. By Condition~(BPM4), we can assume w.l.o.g.\ that $c(x_1,x_2,x_3)=c(x_1,x_2,x_4)$.
Then by Lemma~\ref{l:4pt}(ii),
$c(x_1,x_2,x_3)\in [x_1,x_2]\cap[x_1,x_3]\cap[x_2,x_3]\cap[x_1,x_4]\cap [x_2,x_4]$, and \eqref{e:019} yields that
for $\{i,j\} \in \{\{1,2\},\,\{1,3\},\,\{1,4\},\,\{2,3\},\,\{2,4\}\}$,
\begin{equation}\label{e:cinint}
	r_\nu(x_i,x_j) = r_\nu\(x_i,c(x_1,x_2,x_3)\)  + r_\nu\(c(x_1,x_2,x_3),x_j\).
\end{equation}
Therefore,
\begin{equation}
\begin{aligned}
\label{e:015}
    &r_\nu(x_1,x_3)+r_\nu(x_2,x_4)
    \\
 &=
   r_\nu\big(x_1,\,c(x_1,x_2,x_3)\big)+r_\nu\big(c(x_1,x_2,x_3),\,x_3\big) +
   r_\nu\big(x_2,\,c(x_1,x_2,x_3)\big)+r_\nu\big(c(x_1,x_2,x_3),\,x_4\big)
   \\
   &=r_\nu(x_1,x_2)+r_\nu\big(c(x_1,x_2,x_3),\,x_3\big)+r_\nu\big(c(x_1,x_2,x_3),\,x_4)
  \\
  &\ge
    r_\nu(x_1,x_2)+r_\nu(x_3,x_4),
 \end{aligned}
\end{equation}
	and analogously,
\begin{equation}
\begin{aligned}
\label{e:016}
    r_\nu\big(x_1,x_4\big) + r_\nu\big(x_2,x_3)
  &=
    r_\nu\big(x_3,\,c(x_1,x_2,x_3)\big)+r_\nu\big(c(x_1,x_2,x_3),\,x_4\big)+r_\nu(x_1,x_2)
  \\
  &\ge
  r_\nu(x_1,x_2)+r_\nu(x_3,x_4).
\end{aligned}
\end{equation}
This means that the four point Condition~(MT\ref{MT:4pt}) is satisfied. Moreover, \eqref{e:cinint} implies
Condition~(MT\ref{MT:cex}) with branch point $\pi_\nu(c(x_1,x_2,x_3))$. In particular, $\pi_\nu$ is a tree
homomorphism.
\end{proof}

\begin{remark} Lemma~\ref{l:o0hyper} also explains why we had defined
$r_\nu$ as in \eqref{e:rnu2} and not just as $r_\nu':=\nu([x,y])$ for $x\ne y$.
Namely, in the latter case we would still have (MT\ref{MT:4pt}), but (MT\ref{MT:cex}) might fail.
Take, for example, $T:=\{1,2,3\}$, $c(1,2,3)=2$, and $\nu=\delta_2$. In this case, $r_\nu'$ is the
	discrete metric on $T$, thus $2$ does not lie on the interval $[1,3]$ anymore.
\label{r:020}
\end{remark}

Let $(T,c)$ be an algebraic tree.
For all $v\in T$, define the \emph{degree} of $v$ in $(T,c)$ by
\begin{equation}
\label{e:degree}
   \deg(v):=\deg_{(T,c)}(v):=\#\bset{\Sub_v(y)}{y\in T}.
\end{equation}
We say that $v\in T$ is
a \emph{leaf} if $\deg_{(T,c)}(v)=1$, and
a \emph{branch point} if $\deg_{(T,c)}(v)\ge 3$.
Notice that
\begin{equation}
\label{e:leaf}
   \lf(T,c):=\big\{u\in T:\,c(u,v,w)\ne u \;\forall v,w\in T\setminus\{u\}\big\}
\end{equation}
equals the set of leaves of $T$, and
\begin{equation}
\label{e:br}
   \br(T,c):=\big\{u\in T:\,c(x,v,w)= u \;\mbox{ for some } x,v,w\in T\setminus\{u\}\big\}
\end{equation}
the set of branch points.
Moreover, note that any $\nu$-mass on $\lf(T,c)$ that is not atomic does not contribute to $r_\nu$.

\begin{proposition}[metric representations of algebraic trees]
\label{p:properI}
	Let\/ $(T,c)$ be an algebraic tree, $\nu$ a measure on\/ $(T,\B(T,c))$ with\/ $\nu([x,y])<\infty$ for
	all\/ $x,y\in T$, and\/ $r_\nu$ defined by \eqref{e:rnu}. Then the following hold:
\begin{enumerate}
\item If\/ $(T,c)$ is order separable and\/ $\nu$ has at most countably many atoms, then\/ $(T_\nu,r_\nu)$ is separable.
\item If\/ $\# T>1$, $(T,c)$ is order complete, and\/ $[x,y]$ is order separable for every $x,y\in T$, then
	$(T_\nu,r_\nu)$ is connected if and only if\/ $\nu$ is non-atomic.
	In this case, $(T_\nu,r_\nu)$ is a complete\/ \Rtree.
\end{enumerate}
\end{proposition}

\begin{proof} Throughout the proof denote by $\pi_\nu\colon T\to T_\nu$ the canonical projection.

\smallbreak
\emph{(i) }
It is easy to see that if a set $A\subseteq T$ satisfies \eqref{e:sep} and contains all atoms of $\nu$, then
$\pi_\nu(A)$ is dense in $(T_\nu,r_\nu)$. Therefore, by Proposition~\ref{p:separable} order separability of
$(T,c)$ implies separability of $(T_\nu,r_\nu)$.

\smallbreak
\emph{(ii) }
For all $x,y\in T$ with $x\ne y$, $r_\nu(x,y)\ge \frac12 \nu\{x\}$. Hence $(T_\nu,r_\nu)$ cannot be connected if $\nu$ has atoms.
	Conversely, assume that $\nu$ is non-atomic. For $x,z\in T$, consider $([x,z], \le_x)$, which is a
	totally ordered space according to Proposition~\ref{p:semilat}, and define $y:=\sup\set{v\in
	[x,z]}{2\nu([x,v]) \le \nu([x,z])}$. The supremum exists due to order completeness of $(T,c)$.
	Because of the order separability of $[x,z]$ and the non-atomicity of $\nu$, we obtain
	$2\nu([x,y])=\nu([x,z])=2\nu([y,z])$ and therefore $2r_\nu(x,y)=r_\nu(x,z)=2r_\nu(y,z)$.
	From this equality, connectedness follows once we have shown completeness, and every connected metric
	tree is an \Rtree.
	
	Recall from Lemma~\ref{l:o0hyper} that $(T_\nu,r_\nu)$ is a metric tree. The same holds for its metric
	completion $\Tb_\nu$. Assume for a contradiction that there is a sequence $\folge{x}$ in $T_\nu$
	converging to some $x\in \Tb_\nu\setminus T_\nu$.
	Then $x$ cannot be a branch point and one of the at most two components of $\Tb_\nu \setminus \{x\}$
	contains infinitely many $x_n$. Thus we may assume w.l.o.g.\ that $x\in \lf(\Tb_\nu)$.
	Define $y_n := c_{\Tb_\nu}(x_1, x_n, x)$. Then $y_n\to x$ and, for large enough $m$,
	we have $y_n=c_{\Tb_\nu}(x_1, x_n, x_m)$.
	Hence $y_n\in T_\nu$ for all $n\in\N$ and we may choose representatives
	$x_n'\in \pi_\nu^{-1}(y_n)$ such that $x_n'=c(\rho, x_n', x_m')$ for $\rho:=x_1'$ and all sufficiently
	large $m$.
	By Proposition~\ref{p:semilat}, $\set{x_n'}{n\in\N}$ is totally ordered w.r.t.\ $\le_\rho$, and hence
	$x':=\sup\set{x_n'}{n\in\N} \in T$ exists by order completeness. Obviously, $\pi_\nu(x')=x$ and $x\in
	T_\nu$.
\end{proof}

In order to prove Theorem~\ref{t:algtreechar}\ref{i:converse} using Proposition~\ref{p:properI}, we need a
non-atomic probability measure $\nu$ (to ensure connectedness of $(T_\nu, r_\nu)$) charging all intervals (so
that $\pi_\nu$ is injective). Such a measure always exists in the case of \emph{order separable} algebraic
continuum trees.

\begin{lemma}\label{l:133}
	Let\/ $(T,c)$ be an order separable algebraic continuum tree with\/ $\# T>1$. Then there exists a non-atomic
	probability measure\/ $\nu$ on\/ $(T,\B(T,c))$ with\/ $\nu(\lf(T,c))=0$ and
	\begin{equation}\label{e:fullsupp}
		\nu\big([x,y]\big)>0\quad\forall\, x,y\in T,\; x\ne y.
	\end{equation}
\end{lemma}

\begin{proof} Fix $\rho\in T$. Then, for every $x\in T\setminus\{\rho\}$, the interval $([\rho,x], \le_\rho)$ is
	a separable \emph{linear continuum} in the sense of order theory, i.e.\ a totally ordered space (proven in
	Proposition~\ref{p:semilat}) without \emph{jumps} (what we call here edges) or \emph{gaps} (which follows from
	directed order completeness).
	Due to Cantor's order characterisation of $\R$ (e.g.\ \cite[Theorem~560]{Dasgupta14}),
	this means that $[\rho,x]$ is order isomorphic to the unit interval. Obviously, every order isomorphism
	is measurable and bijective, and the image of Lebesgue measure on the unit interval is a non-atomic
	probability measure $\nu_x$ on $[\rho,x]$. Then $\sum_{n\in\N} 2^{-n} \nu_{x_n}$, where
	$\set{x_n}{n\in\N}$ satisfies \eqref{e:sep}, is a non-atomic
	probability satisfying \eqref{e:fullsupp} and $\nu(\lf(T,c))=0$.
\end{proof}

Any separable $\R$-tree $(T,r)$ comes with an intrinsic measure, called length measure, that generalizes the
Lebesgue-measure on $\R$. More generally, if $(T,r)$ is a complete, separable metric tree and $\rho\in T$ a
fixed root, the \emph{length measure} $\lambda=\lambda^{(T,r,\rho)}$ is uniquely defined by the two properties
$\lambda([\rho,x])=r(\rho, x)$ for all $x\in T$, and $\lambda(\lf_0(T,r))=0$, where $\lf_0$ is the set
of non-isolated leaves (see \cite[Section~2.1]{AthreyaLohrWinter17}).
Note that the total mass $\lambda(T)$ (the ``total length'' of the metric tree) does not depend on the choice of
$\rho$.

\begin{proposition}[total length of $(T_\nu,r_\nu)$]
	Let\/ $(T,c)$ be an order separable, order complete algebraic tree, $\nu$ a measure on\/
	$(T,\B(T,c))$ with\/ $\nu([x,y])<\infty$ for all\/ $x,y\in T$ and such that\/
	$\nu\restricted{\lf(T,c)}$ is purely atomic, and\/ $r_\nu$ be defined by \eqref{e:rnu}. Then the following hold:
\begin{enumerate}
\item The total length of the metric tree\ $(T_\nu,r_\nu)$  is given by
	\begin{equation}
    \label{e:Tnulen}
		\lambda(T_\nu) = \tfrac12 \int_T \deg_{(T,c)} \,\dx\nu.
	\end{equation}
\item $\int_T\deg_{(T,c)} \,\dx\nu = \int_{T_\nu} \deg_{(T_\nu,r_\nu)}\circ \pi_\nu \,\dx \nu.$
\end{enumerate}
\label{p:Tnuprop}
\end{proposition}

\begin{proof} \emph{(i) }
	Let $D:=\set{v_n}{n\in\N}$ be a subset of $(T,c)$ which contains the atoms of $\nu$ and satisfies
	\eqref{e:sep}, and $\pi_\nu\colon T\to T_\nu$ be the canonical projection. We use $\rho:=\pi_\nu(v_1)$
	as the root of $(T_\nu,r_\nu)$. Then
	\begin{equation} \label{e:130}
	   T\setminus \lf(T,c) \subseteq \gentree{D}=\bigcup_{n\in\N}\gentree{v_1,...,v_n},
	\end{equation}
	where $\gentree{A}:=\bigcup_{x,y\in A}[x,y]$.
	Hence $\nu(T\setminus \gentree{D})=0$, and
	\begin{equation} \label{e:Tnapprox}
		\lambda^{(T_\nu,r_\nu,\rho)}(T_\nu)
	   = \lim_{n\to\infty}\lambda^{(T_\nu,r_\nu,\rho)}\big(\pi_\nu(\gentree{v_1,...,v_n})\big).
	\end{equation}

	Abbreviate $T_n:=\gentree{v_1,...,v_n}$ and $\ell_n:=\lambda^{(T_\nu,r_\nu,\rho)}\big(\pi_\nu(\gentree{v_1,...,v_n})\big)$.
	If $v_{n+1}\in T_n$, then $T_{n+1}=T_n$ and
	$\lambda^{(T_\nu,r_\nu,\rho)}\big(\pi_\nu(T_{n+1})\big) =\lambda^{(T_\nu,r_\nu,\rho)}\big(\pi_\nu(T_n)\big)$.
	Otherwise, there exists a unique $u_n\in T$ with $T_{n+1}=T_n\uplus\lopenint{u_n}{v_{n+1}}$, and thus
	\begin{equation}
		\ell_{n+1} = \ell_n + r_\nu(u_n,v_{n+1})
		  = \ell_n + \nu\(\lopenint{u_n}{v_{n+1}}\) - \tfrac12 \nu\{v_{n+1}\} + \tfrac12 \nu\{u_n\}.
	\end{equation}
	For $v\in T_n$, let $\deg_n(v)$ be the degree of $v$ in the tree $(T_n, c\restricted{T_n})$.
	In the case $v_{n+1}\not\in T_n$, we have $\deg_{n+1}(v)=\deg_n(v)$ for $v\in T_n\setminus\{u_n\}$,
	and $\deg_{n+1}(u_n)=\deg_n(u_n)+1$.
	By induction over $n$, we obtain
	\begin{equation}\label{e:Tnlen}
		\ell_n = \tfrac12 \int_{T_n} \deg_n \,\dx\nu
	\end{equation}
	Note that $\deg_n(v)$ is monotonically increasing in $n$, and $\deg(v)=\lim_{n\to\infty} \deg_n(v)$
	holds for all $v\in \gentree{D}$. Thus using the monotone convergence theorem,
	combining \eqref{e:Tnapprox} and \eqref{e:Tnlen} yields \eqref{e:Tnulen}.

\smallbreak
\emph{(ii) }If $\deg_{(T,c)}(v) \ne \deg_{(T_\nu,r_\nu)}(\pi_\nu(v))$,
then either $\pi(\Sub_v(y))=\{\pi(v)\}$ for some $y\in T$
(and thus $\deg_{(T,c)}(v) > \deg_{(T_\nu,r_\nu)}(\pi_\nu(v))$), or
$\pi(v)=\pi(v')$ for some $v'\in \mathrm{Br}(T,c)$ (and thus
$\deg_{(T,c)}(v) <\deg_{(T_\nu,r_\nu)}(\pi_\nu(v))$). In both cases we have that
$\nu\{v\}=\nu\{\pi_\nu(v)\}=0$,
and thus the claim follows.
\end{proof}

\begin{cor}[compactness for bounded degree trees]
	Let\/ $(T,c)$ be an order separable algebraic tree, and\/ $\nu$ a finite measure on\/ $(T,\B(T,c))$
	with\/ $\nu\{v\in T:\;\deg(v)>d\}=0$ for some\/ $d\in\N$. Then the completion of\/ $(T_\nu,r_\nu)$ is compact.
\label{c:Tcompact}
\end{cor}

\begin{proof}
	W.l.o.g.\ assume that $\nu\restricted{\lf(T,c)}$ is non-atomic (if $\nu\restricted{\lf(T,c)}$ has a
	non-atomic part, we can remove it without changing $r_\nu$).
	Then by Proposition~\ref{p:Tnuprop}(i),  $(T_\nu,r_\nu)$ has finite total length.
	As complete metric trees with finite total length are necessarily compact, the statement follows.
\end{proof}

We are now in a position to prove Theorem~\ref{t:algtreechar}.

\begin{proof}[Proof of Theorem~\ref{t:algtreechar}.]
\emph{(i) ``$\Longleftarrow$'' }
	Since every compact metric space is bounded, complete and separable, this step follows from Lemma~\ref{l:cinduce}.

\emph{``$\Longrightarrow$'' } Let $(T,c)$ be an order separable algebraic continuum tree. To avoid
trivialities, assume that $T$	contains more than two points. By Lemma~\ref{l:133}
we can choose a non-atomic probability measure $\nu$ on $(T,\B(T,c))$ satisfying \eqref{e:fullsupp}.
Define $r_\nu$ by \eqref{e:rnu}. Then the equivalence classes in $T_\nu$ are
	singletons by \eqref{e:fullsupp}, and we may identify $T_\nu$ with $T$.

   By Proposition~\ref{p:properI}, $(T,r_\nu)$ is a complete \Rtree\ and the identity is a tree homomorphism by Lemma~\ref{l:o0hyper}.
	Thus $c$ is induced by $r_\nu$. Moreover, $\nu(\br(T,c))=0$ because $\br(T,c)$ is countable and $\nu$
is non-atomic. We can therefore conclude with Corollary~\ref{c:Tcompact} that
	$(T,r_\nu)$ is also compact. \smallskip
	
\smallbreak
\emph{(ii) ``$\Longleftarrow$'' } This is obvious because every order separable algebraic continuum tree is
	induced by a separable \Rtree\ according to part(i), and subspaces of separable metric spaces are
	separable.

\emph{``$\Longrightarrow$''  }
	Let $(T,c)$ be an order separable algebraic tree and $D\subseteq T$ a countable set satisfying
	\eqref{e:sep}.
	Let $\nu$ be any probability measure on $D$ with $\nu\{x\}>0$ for all $x\in D$, and $r_\nu$ defined by
	\eqref{e:rnu}. The equivalence classes in $T_\nu$ are singletons, and we may again identify $T_\nu$ with
	$T$. By Proposition~\ref{p:Tnuprop}, $(T, r_\nu)$ is a metric tree with \eqref{e:cTrTc}. As $(T,c)$ is
	order separable, $(T,r_\nu)$ is separable by Proposition~\ref{p:properI}(i). Moreover, the diameter of
	$(T,r_\nu)$ is bounded by $1$.
	Hence, by \cite[Theorem~3.38]{Evans2008} (known since \cite{Dress84}), there is a bounded, separable
	\Rtree\ $(\Tb,\bar{r})$ such that $T\subseteq \Tb$ and $r_\nu$ is the restriction of $\bar{r}$ to $T$.
	By Lemma~\ref{l:cinduce}, this \Rtree\ induces an algebraic continuum tree $(\Tb,\cb)$, and $T$ is a
	subtree of $\Tb$.

\emph{``in particular''.} According to part \emph{(i)}, there is a metric $\tilde{r}$ on $\Tb$ such that
	$(\Tb,\tilde{r})$ is a compact \Rtree\ inducing $(\Tb,\cb)$. Let $r$ be the restriction of $\tilde{r}$
	to $T$. Then $(T,r)$ is a totally bounded metric tree inducing $(T,c)$.
\end{proof}

\subsection{Tree homomorphisms versus homeomorphisms}
\label{s:hom(e)o}
Since order separable algebraic continuum trees are \Rtrees\ where we have ``forgotten'' the metric, the question
arises how homeomorphisms of \Rtrees\ relate to tree homomorphisms of the corresponding algebraic trees. A
first observation is that homeomorphisms are necessarily tree homomorphisms. This statement relies on
connectedness of the \Rtrees\ and we cannot replace ``\Rtree'' by ``metric tree'': every bijection between finite
metric trees is obviously a homeomorphism because the topologies are discrete, but not necessarily a tree
homomorphism.

\begin{lemma}[homeomorphisms are tree isomorphisms]\label{l:homeomhom}
	Let\/ $(T,r),\, (\Th,\rh)$ be\/ \Rtrees, and\/ $f\colon T \to \Th$ a homeomorphism. Then\/ $f$ is a tree homomorphism.
\end{lemma}
\begin{proof}
	The branch point map can be expressed in terms of intervals by \eqref{e:bp}. In an \Rtree\ $(T,r)$, the interval
	$[x,y]$, $x,y\in T$, is the unique simple path from $x$ to $y$, which is a purely topological notion, and
	hence preserved by homeomorphisms.
\end{proof}

\begin{example}[tree isomorphisms need not be homeomorphisms] In Lemma~\ref{l:homeomhom},
	the converse is not true: bijective tree homomorphisms need not be
	homeomorphisms, even if the trees are order separable.
	To see this, let $r,\rh$ the metrics on $\N$ defined by $r(n,m)=\frac1n+\frac1m$,
	$\rh(n,m)=2$ for distinct $n,m\in \N$. Let $T$ and $\Th$ be the \Rtrees\ generated by $(\N,r)$ and
	$(\N,\rh)$, respectively. Then both $\Th$ and $T$ are the countable star with set $\N$ of leaves. In $T$,
	the distance from the branch point to leaf $n$ is $\frac1n$, while it is $1$ in $\Th$. Hence $T$ is
	compact while $\Th$ is not.
	The identity on $\N$ can be extended to a bijective tree homomorphism $f\colon T \to \Th$ which cannot be
	continuous.
\label{ex:homhomeom}
\end{example}

Example~\ref{ex:homhomeom} shows that it is possible for non-homeomorphic (topologically non-equivalent) \Rtree s
to induce isomorphic (equivalent) algebraic continuum trees. This can only happen if at least one of the
trees is non-compact.

\begin{proposition}[tree isomorphisms of compact \Rtrees\ are homeomorphisms]
Let\/ $T,\Th$ be \Rtrees, and\/ $f \colon T \to \Th$.
\begin{enumerate}
\item\label{i:homcont} If\/ $\Th$ is compact, $f(T)$ is connected, and\/ $f$ a tree homomorphism, then\/ $f$ is
	continuous.
\item\label{i:isomhomeom} If both\/ $T$ and\/ $\Th$ are compact and\/ $f$ is bijective, then\/ $f$ is a
	homeomorphism if and only if it is a tree homomorphism.
\end{enumerate}
\label{p:isomhomeom}
\end{proposition}

\begin{proof}
	\ref{i:isomhomeom} is obvious from \ref{i:homcont} and Lemma~\ref{l:homeomhom}.

	Assume $f$ is a tree homomorphism, $f(T)$ is connected, and $\Th$ is compact. Choose a root $\rho \in T$.
	Let $v_n \to v$ be a convergent sequence in $T$, and $w\in \Th$ an accumulation point of $f(v_n)$.
	Then there is a subsequence $(n_k)_{k\in\N}$ with $f(v_{n_k}) \to w$. We have
	\begin{equation}\label{e:supinf}
		v = \sup_{k\in\N} \inf_{i>k} v_{n_i} \qquad\text{and}\qquad
		w = \sup_{k\in\N} \inf_{i>k} f(v_{n_i}),
	\end{equation}
	where $\sup$ and $\inf$ are w.r.t.\ the partial orders $\le_\rho$ and $\le_{f(\rho)}$ in the first and
	second equality, respectively. In the following, we show $w=f(v)$.
	Because $f$ is order preserving for these partial orders due to Lemma~\ref{l:treehom}, we obtain $w
	\le_{f(\rho)} f(v)$. Assume for a contradiction $w\ne f(v)$.  Because $f(T)$ is connected, there is
	$y\in \Th$ with $w<_{f(\rho)} y<_{f(\rho)} f(v)$ and $x \in T$ with $y=f(x)$. For $u:=c(\rho,x,v)$, we
	have $f(u) = \ch(f(\rho),y,f(v)) = y$, $u \le_\rho v$, and $u\ne v$.
	Therefore, $u \le_\rho v_{n_i}$ for all sufficiently large $i$, and thus $y=f(u) \le_{f(\rho)}
	f(v_{n_i})$ for those $i$. Now \eqref{e:supinf} implies $y\le_{f(\rho)} w$ in contradiction to the
	choice of $y$, finishing the proof of $w=f(v)$.
	Compactness of $\Th$ and uniqueness of accumulation points implies $f(v_n) \to f(v)$, and $f$ is
	continuous.
\end{proof}

In view of Theorem~\ref{t:algtreechar}, Proposition~\ref{p:isomhomeom} implies that order separable algebraic
continuum trees are in one-to-one correspondence with homeomorphism classes of compact \Rtrees. Furthermore,
the unique metric topology induced by the compact \Rtree\ coincides with the component topology $\tau$
introduced in Subsection~\ref{sub:astopological}.
But be aware that there may be other, non-homeomorphic, non-compact \Rtrees\ inducing the same order separable
algebraic continuum tree, as shown in Example~\ref{ex:homhomeom}.

\begin{cor}[uniqueness of inducing \Rtree]\label{c:unique}
	Every order separable algebraic continuum tree is induced by a compact\/ \Rtree\ that is unique up to
	homeomorphism, and the unique induced topology coincides with the component topology\/ $\tau$ defined in
	\eqref{e:tau}.
\end{cor}

\begin{proof}
	That an order separable algebraic continuum tree is induced by a compact \Rtree\ is
	Theorem~\ref{t:algtreechar}\ref{i:converse}.
	Any two such compact \Rtrees\ are isomorphic as algebraic trees, hence homeomorphic by
	Proposition~\ref{p:isomhomeom}. The component topology is a Hausdorff topology and clearly weaker than the
	topology induced by the \Rtree, because components are open sets of \Rtree s.
	Hence, by compactness of the \Rtree, the two topologies coincide.
\end{proof}

\section{The space of algebraic measure trees}
\label{S:amt}
In this section, we define algebraic measure trees, and equip the space of (equivalence classes of) algebraic
measure trees with a topology.
In what follows, the order separability of the underlying algebraic tree is crucial. Therefore, we include it
already in the following definition of algebraic measure trees.
\begin{definition}[algebraic measure trees]
An \define{algebraic measure tree} $(T,c,\mu)$ is an order separable algebraic tree $(T,c)$ together with a
probability measure $\mu$ on $\B(T,c)$.
\label{d:amt}
\end{definition}

\begin{definition}[equivalence of algebraic measure trees]
\begin{enumerate}
\item	We call two algebraic measure trees $(T_i,c_i,\mu_i)$, $i=1,2$, \define{equivalent} if there exist
	subtrees $A_i$ of $T_i$ with $\mu_i(A_i)=1$, and a measure preserving tree isomorphism $f$ from
	$A_1$ onto $A_2$. In this case, we call $f$ \define{isomorphism} of the algebraic measure trees.
\item	A metric measure tree $(T,r,\mu)$ is called a \define{metric representation} of the algebraic measure
	tree $(T',c',\mu')$ if its induced algebraic measure tree $(T,c_{(T,r)},\mu)$ is equivalent to
	$(T',c',\mu')$.
\end{enumerate}
\label{d:amtequiv}
\end{definition}

In the following, we denote for an algebraic measure tree $\smallx:=(T,c,\mu)$ by $\supp(\smallx)$
the algebraic subtree generated by the support of $\mu$, i.e.\
\begin{equation}\label{e:suppsmallx}
	\supp(\smallx):=c\(\supp(\mu)^3\),
\end{equation}
and by
\begin{equation}
\label{e:brsmallx}
   \br(\smallx):=\br(T,c)\cap \supp(\smallx)
\end{equation}
the set of \emph{branch points} of $\smallx$.
It is easy to check that an isomorphism $f$ from $\smallx=(T,c,\mu)$ to $\smallx'=(T',c',\mu')$ induces a
bijection between $\br(\smallx)$ and $\br(\smallx')$ (although it need neither be defined nor injective on all
of $\supp(\smallx)$). Also note that $\smallx$ is equivalent to $\supp(\smallx)$ equipped with the appropriate
restrictions of $c$ and $\mu$.

\begin{remark}[a note on our definition of equivalence]
	Every algebraic measure tree is equivalent to an algebraic continuum measure tree, and has a metric
	representation with a compact \Rtree\ by Theorem~\ref{t:algtreechar}. For the definition of equivalence
	of algebraic measure trees it is important that we do not require the whole trees to be isomorphic (see
	Example~\ref{e:toptree} below). On the other hand, it is also important that the isomorphism is injective
	on a subtree (as opposed to only a subset) of full measure, because otherwise it would not be an
	equivalence relation and every tree with $n$ leaves and uniform distribution on them would be equivalent
	to the $n$-star.
\label{r:algcont}
\end{remark}

\begin{example}[the linear non-atomic measure tree]\label{ex:lintree}
	There is only one equivalence class of linearly ordered algebraic measure trees with non-atomic measure.
	Indeed, let $(T,c,\mu)$ be an algebraic measure tree with $\br(T,c)=\emptyset=\at(\mu)$. Then, by
	Theorem~\ref{t:algtreechar}, there is a tree isomorphism from $T$ into $[0,1]$ and we may assume
	$T\subseteq [0,1]$ to begin with. Let  $F_\mu\colon [0,1]\to[0,1]$ be the
	distribution function of $\mu$. Then $F_\mu$ is continuous and maps $\mu$ to Lebesgue-measure
	$\lambda_{[0,1]}$.
	Let $A:=\set{x\in\supp(\mu)}{\text{there is no }y_n\in[0,1]\setminus \supp(\mu): y_n<x,\; y_n\to x}$ be
	the support of $\mu$ with left boundary points removed. Then $F_\mu$ restricted to $A$ is bijective and
	hence a measure preserving tree isomorphism onto $[0,1]$ (with Lebesgue measure and canonical branch
	point map).
	Thus $(T,c,\mu)$ is equivalent to $[0,1]$.
\end{example}

Let
\begin{equation}
\label{e:bbT}
  \T:=\{\text{equivalence classes of algebraic measure trees}\}.
\end{equation}
Next, we equip $\T$ with a topology. We shall base this notion of convergence on the fact that algebraic
measure trees allow for metric representations (see Theorem~\ref{t:algtreechar}), and require convergence in
Gromov-weak topology of particular representations.
To this end, let
\begin{equation}
\label{e:bbH}
  \H := \{\text{equivalence classes of (separable) metric measure trees}\},
\end{equation}
where we consider two metric measure trees $(T,r,\mu)$ and $(T',r',\mu')$ as \emph{equivalent} if there exists a
measure preserving isometry between the metric completions of $\supp(\mu)$ and $\supp(\mu')$.

In order to get a useful topology on $\T$, we cannot take arbitrary (optimal) metric representations. Instead,
given an algebraic measure tree $(T,c,\mu)$, we use the metric $r_\nu$ defined in \eqref{e:rnu2} for the
\emph{branch point distribution} $\nu$, namely the distribution of the random branch point obtained by sampling
three points with the sampling measure $\mu$.

\begin{definition}[branch point distribution]
	The \define{branch point distribution} of an algebraic measure tree $(T,c,\mu)$ is the push-forward of
	$\mu^{\otimes 3}$ under the branch point map,\label{d:bpd}
	\begin{equation}
\label{e:bpd}
		\nu:=c_\ast\mu^{\otimes 3}.
	\end{equation}
\end{definition}

Note that the branch point distribution is not necessarily supported by $\br(T,c)$. For instance, every atom of
$\mu$ is also an atom of $\nu$.
If $(T,c,\mu)$ and $(T',c',\mu')$ are equivalent algebraic measure trees with branch point distributions $\nu$
and $\nu'$, respectively, then the isomorphism is also an isometry w.r.t.\ $r_\nu$ and
$r_{\nu'}$.
Therefore, the following selection map, which associates a particular metric representation to every algebraic
measure tree, is well-defined.

\begin{definition}[selection map $\iota$]
Define the map $\iota\colon \T \to \H$ by
\begin{equation}\label{e:iota}
	 \iota(T,c,\mu) := (T_\nu,r_\nu,\mu_\nu),
\end{equation}
where $\nu=c_*\mu^{\otimes 3}$ is the branch point distribution of $(T,c,\mu)$, $(T_\nu,r_\nu)$ is the quotient
metric space, and $\mu_\nu$ is the image of $\mu$ under the canonical projection $\pi_\nu$.
\label{d:iota}
\end{definition}

The topology we use on $\T$ is the Gromov-weak topology w.r.t.\ the branch point distribution distance. That is,
it is the topology induced by the selection map $\iota$, i.e., the weakest (coarsest) topology on $\T$ such that
$\iota$ is continuous.

\begin{definition}[bpdd-Gromov-weak topology]
	Let $\H$ be equipped with the Gromov-weak topology.
	We call the topology induced on $\T$ by the selection map $\iota$
	\define{branch point distribution distance Gromov-weak topology} (\define{bpdd-Gromov-weak
	topology}).
\label{d:bpddGw}
\end{definition}

The following reconstruction theorem is crucial for the usefulness of bpdd-Gromov-weak convergence. It shows that
the selection map $\iota$ is an embedding and indeed selects metric representations.
\begin{proposition}[$\iota$ is injective]
	The selection map\/ $\iota\colon \T \to \H$ is injective, and\/
	$\iota(\smallx)$ is a metric representation of\/ $\smallx \in \T$.
\label{p:injective}
\end{proposition}
\begin{proof}
	If we show that $\iota(\smallx)$ is a metric representation of $\smallx=(T,c,\mu)\in \T$, it is
	obvious that $\iota$ is injective, because equivalence of metric measure spaces implies equivalence of
	the corresponding algebraic measure trees by Lemma~\ref{l:homeomhom}.

	Choosing an appropriate representative, we can assume that $\nu\{v\}>0$ for all $v\in\br(T,c)$.
	The canonical projection $\pi_\nu\colon T \to T_\nu$ is a tree homomorphism by Lemma~\ref{l:o0hyper}.
	To show equivalence of $(T,c,\mu)$ and $(T_\nu, c_{(T_\nu,r_\nu)},\mu_\nu)$, we have to show that
	$\pi_\nu$ is injective on a subtree $A\subseteq T$ with $\mu(A)=1$.
	Let $N:=\set{v\in T}{\pi_\nu(v) \ne \{v\}}$. Then $\mu(\pi_\nu(v)) = 0$ for all $v\in N$, and
	$w\in\pi_\nu(v)$ implies $[v,w]\subseteq \pi_\nu(v)$ because $\pi_\nu$ is a tree homomorphism.
	Because there are at most countably many non-degenerate, disjoint closed intervals in $T$ due to
	order separability, this implies that $\pi_\nu(N)$ is countable, and thus $\mu(N)=0$.
	Define $A=T\setminus N$. Then $\mu(A)=1$, and $\pi_\nu$ is injective on $T\setminus N$.
	To see that $A$ is a subtree, pick $x,y,z\in A$. If $v:=c(x,y,z)\in\{x,y,z\}$, then $v\in A$.
	Otherwise, $v\in \br(T,c)$, and hence $\nu\{v\}>0$. This implies $\pi_\nu(v)=\{v\}$, i.e.\ $v\in A$.
\end{proof}

\begin{cor}[metrizability]\label{c:bpdd-metrizable}
	$\T$ equipped with bpdd-Gromov-weak topology is a separable, metrizable space.
\end{cor}
\begin{proof}
	The Gromov-weak topology on $\H$ is separable, and metrizable, e.g.\ by the Gromov-Prohorov metric
	$\dGP$ (see \cite{GrevenPfaffelhuberWinter2009}).
	Because $\iota$ is injective by Proposition~\ref{p:injective},
	$\diota(\smallx,\smally) := \dGP(\iota(\smallx),\iota(\smally))$, $\smallx,\smally\in \T$, is a metric
	on $\T$ inducing bpdd-Gromov-weak topology.
\end{proof}

\begin{remark}[distance polynomials]\label{rem:bpddGw}
By definition, a sequence $(\smallx_n)_{n\in\N}$ in $\T$ converges to $\smallx\in \T$ bpdd-Gromov-weakly
if and only if $\iota(\smallx_n) \to \iota(\smallx)$ Gromov-weakly.
It has been shown that the Gromov-weak convergence is equivalent to the convergence
of the distribution of the distance matrix (\cite[Theorem~5]{GrevenPfaffelhuberWinter2009}). Therefore, the
bpdd-Gromov-weak convergence is equivalent to
\begin{equation}
	\Phi(\smallx_n) \tno \Phi(\smallx)
\end{equation}
for all so-called \emph{polynomials} $\Phi\colon \T\to\R$, which are test functions of the form \eqref{e:PhiGw}.
Note that the set $\ePol$ of all polynomials is an
algebra, and therefore also convergence determining for $\T$-valued random variables (see
\cite{Loehr2013,BlountKouritzin2010}).
\end{remark}

As pointed out in Remark~\ref{r:algcont}, the equivalence class of every algebraic measure tree contains an
algebraic \emph{continuum} measure tree.
The following example shows that $\iota$ would not be injective if we had defined it on the set of algebraic
continuum measure trees with the stricter notion of equivalence where the whole algebraic continuum trees
have to be measure preserving isomorphic.

\begin{example}\label{e:toptree}
	For $x\ge 0$, let $T_x$ be the \Rtree\ generated by the interval $I_x=[-x,1]$ together with additional
	leaves $\{v_n\}$, $n\in\N$, where $c(0,1,v_n)=\frac1n$ and $r(\frac1n,v_n)=\frac1n$, i.e.\ at each point
	$\frac1n \in I_x$ there is a branch of length $\frac1n$ attached.
	Then $T_x$ is a compact \Rtree\ for every $x\ge 0$, hence induces an algebraic continuum tree by
	Theorem~\ref{t:algtreechar}. Let $\mu_x\{-x\}=\frac12$, and $\mu_x\{v_n\}=2^{-n-1}$ for $n\in\N$. Then
	$\smallx_x := (T_x,\mu_x) \in \Tbin$. Now $\iota(\smallx_x)=\iota(\smallx_y)$ for every $x,y\ge 0$, but
	$T_x$ and $T_0$ are not homeomorphic, hence not isomorphic by Proposition~\ref{p:isomhomeom}.

	Note that $A_x:=\{x\} \cup\set{v_n}{n\in\N} \cup \set{\frac1n}{n\in\N}$ is a subtree of $T_x$ with
	$\mu_x(A_x)=1$, and $A_x$ is isomorphic (although not homeomorphic) to $A_0$.
\end{example}

In order to construct algebraic measure trees, it is of course not necessary to specify the mass of every Borel
subset. On the contrary, we can use the following Carath\'eodory-type extension result.
To this end, recall for $x,y\in T$ with $x\not =y$ from \eqref{e:005} the component
$\Sub_x(y)=\Sub_x^{(T,c)}(y)$ of $T\setminus\{x\}$ which contains $y$. In this section, it is convenient to define
\begin{equation}\label{e:Subxx}
	\Sub_x(x) := \{x\}.
\end{equation}
Then $T$ is the disjoint union of the $\deg(x)+1$ sets in
\begin{equation}
	\C_x := \bset{\Sub_x(y)}{y\in T}.
\end{equation}
Note that $\C_x=\set{\Sub_x(y)}{y\in V}$ for order dense $V\subseteq T$ with $x\in V$. In particular,
$\C_x$ is countable if $(T,c)$ is order separable.
For $y\in T$, $V\subseteq T$, we call a function $f\colon V \to \R$ \define{order left-continuous} on $V$ w.r.t.\
$\le_y$ if the following holds.
For all $x,x_n\in V$ with $x_1\le_y x_2 \le_y\cdots$ and $x=\sup_{n\in\N} x_n$ w.r.t.\ $\le_y$ (in short $x_n
\uparrow x$), we have $\lim_{n\to\infty} f(x_n) = f(x)$.
Recall the notion of algebraic continuum tree from Definition~\ref{d:aCT}.

\begin{proposition}[extension to a measure]\label{p:caratheodory}
	Let\/ $(T,c)$ be an order separable algebraic continuum tree, and\/ $V\subseteq T$ order
	dense. Then a set-function\/ $\mu_0 \colon \C_V:= \bigcup_{x \in V} \C_x \to [0,1]$ has a unique extension to a
	probability measure on\/ $\B(T,c)$ if it satisfies
	\begin{enumerate}[1.]
		\item For all\/ $x\in V$, $\sum_{A \in \C_x} \mu_0(A) = 1$
		\item For all\/ $x,y\in V$ with\/ $x\ne y$,
		\begin{equation}\label{e:posint}
			\mu_0\(\Sub_x(y)\) + \mu_0\(\Sub_y(x)\) \ge 1
		\end{equation}
		\item For every\/ $y\in V$, the function\/ $\psi_y\colon x\mapsto \mu_0(\Sub_x(y))$
			is order left-continuous on\/ $V$ w.r.t.\ $\le_y$.
	\end{enumerate}
\end{proposition}

\begin{proof}
	Note that $\psi_y(x) = \psi_z(x)$ for $z\in \Sub_x(y)$. We therefore may write $\psi_A(x) := \psi_y(x)$
	for any $A\subseteq \Sub_x(y)$.
	Define the $\cap$-stable set system
\begin{equation}\label{e:A}
	\A := \Bset{\bigcap_{k=1}^n A_k}{n\in\N,\, A_k \in \C_V}.
\end{equation}
	By Corollary~\ref{c:borel}, $\A$ generates the Borel $\sigma$-algebra $\B(T,c)$.
	Let $\emptyset \ne A \in \A$ and $y \in A$. Because $(T,c)$ has no edges and is order complete, we have
	$A = \bigcap_{x\in \partial A} \Sub_x(y)$, where $\partial$ denotes the boundary w.r.t.\ the component
	topology $\tau$, which is a finite set in the case of $A$.
	Using \eqref{e:posint}, we obtain for $v\in V$, $x_0,\ldots,x_n \in V\setminus\{v\}$ such that
	$\Sub_v(x_0),\ldots,\Sub_v(x_n)$ are distinct, that
\begin{equation}
	\psi_{x_0}(v) \le 1 - \sum_{k=1}^n \psi_{x_k}(v) \le 1 - \sum_{k=1}^n \(1 - \psi_v(x_k)\).
\end{equation}
	This implies for $\emptyset \ne A\in \A$, by induction over $\#\partial A$, that
\begin{equation}\label{e:muA}
	\mu(A) := 1 - \sum_{x\in \partial A} \(1-\psi_A(x)\) \ge 0,
\end{equation}
	hence $\mu$ is a non-negative extension of $\mu_0$ to $\A$.
	We claim that $\mu$ is super-additive, additive and inner regular for compact sets. From this it follows
	by standard arguments that it has a unique extension to a measure on the generated $\sigma$-algebra
	$\sigma(\A)=\B(T,c)$.

\pstep{Additivity}
	Let $n\in \N\setminus\{1\}$, and $A_1,\ldots, A_n \in \A\setminus \{\emptyset\}$ disjoint with
	$A:=\biguplus_{k=1}^n A_k \in \A$. Define $D:= \bigcup_{k=1}^n \partial A_k$.
	Then $\partial A \subseteq D$ and there is $x\in D \setminus \partial A \subseteq A$.
	Let $I_x:=\set{k\in \{1,\ldots,n\}}{x\in \partial A_k}$ and choose $y_k\in A_k$.
	Then, because the $A_k$ are disjoint, the $\Sub_x(y_k)$, $k\in I$, are distinct, and because the $A_k$
	cover $A$, we have $\set{\Sub_x(y_k)}{k\in I_x} = \C_x$.
	In particular, $\sum_{k\in I_x} \psi_{y_k}(x) = 1$, and $B_x:=\bigcup_{k\in I_x} A_k \in \A$ with
	$\partial B_x = \biguplus_{k\in I_x} \partial A_k\setminus \{x\}$.
	We obtain
\begin{equation}\label{e:lokad}
\begin{aligned}
	\sum_{k\in I_x} \mu(A_k)
	  &= \sum_{k\in I_x} \Bigl( 1 - \(1-\psi_{y_k}(x)\) - \sum_{z\in \partial A_k \setminus\{x\}} \(1-\psi_{y_k}(z)\) \Bigr) \\
	  &= \sum_{k\in I_x} \psi_{y_k}(x) - \sum_{z\in \partial B_x} \(1-\psi_x(z)\)
	  = \mu(B_x).
\end{aligned}
\end{equation}
	By induction over $n$, this implies additivity of $\mu$.

\pstep{Super-additivity}
	Let $A_1,\ldots, A_n \in \A \setminus \{\emptyset\}$ be disjoint and $\biguplus_{k=1}^n A_k \subseteq A
	\in \A$. The case $n=1$
	is trivial, and we proceed by induction over $n$.
	Choose $y\in A_1$ and let $D:=\partial A_1 \setminus \partial A$.
	For $x\in D$, $C\in \C_x':=\C_x \setminus \Sub_x(y)$
	and $k\in \{2,\ldots,n\}$, either $A_k \subseteq C$, or $A_k \cap C = \emptyset$.
	Therefore, we have the decomposition $\{2,\ldots,n\}=\biguplus_{x\in D}\biguplus_{C\in \C_x'} I_C$ with
	$I_C:=\set{k}{A_k \subseteq C}$. Because $C\cap A \in \A$, and $A_k \subseteq C\cap A$ for $k\in I_C$,
	we can use the induction hypothesis to obtain
\begin{equation}
	\sum_{k\in I_C} \mu(A_k) \le \mu(C\cap A) = \psi_C(x) - \sum_{z\in \partial A \cap C} \(1-\psi_A(x)\).
\end{equation}
	Therefore,
\begin{equation}\begin{aligned}
	\mu(A_1) &= 1 - \sum_{x\in \partial A_1 \cap \partial A} \(1- \psi_y(x)\)
			- \sum_{x\in D} \(1-\psi_y(x)\) \\
	  &= \mu(A) + \sum_{x\in \partial A \setminus \partial A_1} \(1-\psi_y(x)\)
	  	- \sum_{x\in D} \sum_{C\in \C_x'} \psi_C(x) \\
	  &\le \mu(A) - \sum_{x\in D} \sum_{C\in \C_x'} \sum_{k\in I_C} \mu(A_k) \\
	  &= \mu(A) - \sum_{k=2}^n \mu(A_k).
\end{aligned}\end{equation}

\pstep{Compact regularity}
	According to Proposition~\ref{p:compact}, all closed subsets of $T$ are compact. Let $y\in A\in \A$.
	Because $(T,c)$ is an order separable algebraic continuum tree, and $V$ is order dense, we find for $z\in
	\partial A$ a sequence $(x_n(z))_{n\in\N}$ in $A \cap V$ with $x_n(z) \uparrow z$ w.r.t.\ $\le_y$ as
	$n\to\infty$. Define $A_n:= \bigcap_{z\in \partial A} \Sub_{x_n(z)}(y) \in \A$ and $K_n := A_n \cup
	\partial A_n$. Then $K_n$ is compact, $A_n \subseteq K_n \subseteq A$, and because $\partial A$ is
	finite, we have $\partial A_n = \set{x_n(z)}{z\in \partial A}$ for sufficiently large $n$. Thus, by
	order left-continuity of $\psi_y$,
	\begin{equation}
		\lim_{n\to\infty} \mu(A_n) = 1 - \lim_{n\to\infty} \sum_{z\in \partial A} \(1-\psi_y(x_n(z))\)
			= 1 - \sum_{z\in \partial A} \(1-\psi_y(z)\) = \mu(A),
	\end{equation}
	and $\mu$ is inner compact regular as claimed.
\end{proof}

We conclude this section with an extension result, which will be very useful for reading off algebraic measure
trees from (sub\nobreakdash-)triangulations of the circle in Section~\ref{S:triangulation} below.
In Proposition~\ref{p:caratheodory}, we assumed the whole tree to be known, and considered the question of
constructing a probability measure on it. Now, we assume that not the whole tree is given a priori, but only the
(countably many) branch points. The question is, whether there is an extension of the tree which is rich
enough to carry a measure with the specified masses of components.

\begin{proposition}[construction of algebraic measure trees]\label{p:construction}
	Let\/ $(V,c_V)$ be a countable algebraic tree, and for each\/ $x\in V$, let\/ $A \mapsto \psi_A(x)$ be a
	probability measure on\/ $\C_x$. Define\/ $\psi_y(x) := \psi_{\Sub_x(y)}(x)$.
	Assume that for $x,y\in V$ with $x\ne y$,
	\begin{equation}\label{e:psisum}
		\psi_x(y)+\psi_y(x) \ge 1.
	\end{equation}
	Then there is a unique (up to equivalence) algebraic measure tree\/ $\smallx=(T,c,\mu)$ such that
	\begin{enumerate}
		\item\label{i:tree} $V\subseteq T$, $\br(T,c) = \br(V,c_V)$,
		\item\label{i:mupsi} $\mu\(\Sub^{(T,c)}_{x}(y)\) = \psi_y(x)$ for all\/ $x,y\in V$,
		\item\label{i:at} $\at(\mu) \subseteq V$, 
			where\/ $\at(\mu)$ denotes the set of atoms of\/ $\mu$.
	\end{enumerate}
\end{proposition}

Note that in general we cannot obtain $\lf(T,c) \subseteq \lf(V,c_V)$. To the contrary, $\lf(T,c)$ can be
uncountable (for every representative of $\smallx$).

\begin{proof} \pstep{Existence}
	First note that for $y\in V$, $\psi_y$ is monotonic w.r.t. $\le_y$. Indeed,
	$z\le_y x$ implies $\psi_y(z) \le 1 - \psi_x(z) \le \psi_z(x) = \psi_y(x)$.

	We need to enlarge the tree to make $\psi_y$ order left-continuous.
	Because $V$ is countable, we may consider one $y$ and one point $x$ at a time.
	If $x,y \in V$ are such that there exists $x_n \in V$ with $x_n \uparrow x$, then by
	monotonicity $\phi_y(x):=\lim_{n\to\infty} \psi_y(x_n) \le \psi_y(x)$ exists and is independent of the
	choice of $x_n$. If $\phi_y(x) \ne \psi_y(x)$, we extend the tree by adding one extra point $z\not\in
	V$, i.e.\ we consider $\Vt := V \uplus \{z\}$ with the unique extension $\ct$ of $c_V$
	such that $(\Vt,\ct)$ is an algebraic tree with $x_n \le_y z \le_y x$ for all $n$.
	Furthermore, we extend $\psi$ to $\tilde{\psi}$ on $\Vt$ by defining $\tilde\psi_y(z):=\phi_y(x)$,
	$\tilde\psi_z(z)=0$ and $\tilde\psi_x(z)=1-\phi_y(x)$.
	It is easy to check that $(\Vt,\ct)$ together with
	$\tilde\psi$ satisfies the prerequisites of the Proposition, $\br(\Vt,\ct)=\br(V,c)$, and
	$\set{x\in \Vt}{\tilde\psi_x(x) > 0} = \set{x\in V}{\psi_x(x)>0} \subseteq V$.

	Now assume w.l.o.g.\ that $\psi_y$ is already order left-continuous for all $y\in V$. Because $V$ is
	countable, it is in particular order separable and according to Theorem~\ref{t:algtreechar}, there is an
	order separable algebraic continuum tree $(T,c)$ such that $(V,c_V)$ is a subtree. We can choose
	$(T,c)$ such that $\br(T,c) = \br(V,c_V)$. Consider the closure $\Vc$ of $V$ w.r.t.\ the component
	topology $\tau$. For $x\in \Vc\setminus V$, we define
	\begin{equation}
		\psi_y(x) := \sup\set{\psi_y(z)}{z\in V\cap \Sub_x(y)}.
	\end{equation}
	Then \eqref{e:psisum} holds for $x,y\in\Vc$, $x\ne y$, and $\psi_y$ is order left-continuous.
	For every $\{x,y\} \in \edge(\Vc,\cb)$, where $\cb$ is the restriction of $c$ to $\Vc^3$, we fix an
	order isomorphism $\varphi_{x,y}\colon [x,y] \to [0,1]$, which exists by Cantor's order characterization of $\R$
	because $[x,y]$ is a linearly ordered, separable algebraic continuum tree.
	For every $z\in T \setminus \Vc$, there exists $\{x,y\} \in \edge(\Vc,\cb)$, with $z\in[x,y]$. We define
	\begin{equation}
		\psi_y(z) := (1-\varphi_{x,y}(z)) \psi_y(x) + \varphi_{x,y}(z) \( 1 - \psi_x(y)\),
	\end{equation}
	$\psi_x(z):=1-\psi_y(z)$ and $\psi_z(z):=0$.
	Now we can use Proposition~\ref{p:caratheodory} to see that
	\begin{equation}
		\mu_0\(\Sub_x(y)\) := \psi_y(x)
	\end{equation}
	has a unique extension to a probability measure $\mu$ on $\B(T,c)$.

	The last step in the construction is to remove point-masses outside $V$ by expanding them to intervals.
	To this end, let $P:=\at(\mu) \setminus V$, and $\Tb:=(T \setminus P) \uplus (P \times [0,1])$. Because
	$P\subseteq T\setminus V$ contains no branch points, we can extend the restriction
	of $c$ to $T\setminus P$ to a branch point map $\ct$ on $\Tb$ in a canonical way such that
	$[(x,0),(x,1)]=\{x\}\times[0,1]$ for $x\in P$. Define the Markov kernel $\kappa$ from $T$ to $\Tb$ by
	\begin{equation}
		\kappa(x) := \begin{cases} \delta_x, & x\in T\setminus P,\\
			\delta_x \otimes \lambda_{[0,1]}, & x \in P, \end{cases}
	\end{equation}
	where $\delta_x$ is the Dirac measure in $x$ and $\lambda_{[0,1]}$ is Lebesgue measure.
	Let $\mub:=\kappa_*(\mu)$ be the push-forward of $\mu$ under $\kappa$.
	Then $(\Tb,\ct,\mub)$ is a separable algebraic
	measure tree, and by construction $\br(\Tb,\ct) = \br(V,c_V)$ as well as $\at(\mub) = \at(\mu)\cap V
	\subseteq V$.
	Furthermore, for $x,y\in V$, we have $\mub(\Sub_x^{(\Tb,\ct)}(y)) = \mu(\Sub_x^{(T,c)}(y)) = \psi_y(x)$
	as claimed.
\pstep{Uniqueness} Follows similarly, where we note that it does not matter how we distribute the mass on an
	edge of $(\Vc,\cb)$ in a non-atomic way, because all algebraic measure trees without branch points and
	non-atomic measure are equivalent by Example~\ref{ex:lintree}.
\end{proof}

\section{Triangulations of the circle}
\label{S:triangulation}

In this section, we encode binary algebraic measure trees by triangulations of subsets of the circle.
This is comparable with the encoding of compact (ordered, rooted) metric (probability) measure trees by
excursions over the unit interval, where the height profile encodes the branch point map as well as the metric
distances. Moreover, also the measure can be encoded by the excursion by identifying the lengths of
sub-excursions with the mass of the corresponding subtrees. Similarly, it turns out that we can encode binary
algebraic measure trees by what we call sub-triangulations of the circle. As in the case of coding metric
measure trees with excursions, the resulting \emph{coding map} associating to a sub-triangulation the algebraic
measure tree is continuous.

In Subsection~\ref{s:spacetriang}, we introduce the space of sub-triangulations of the circle. In
Subsection~\ref{s:coding}, we construct the coding map.

\subsection{The space of sub-triangulations of the circle}
\label{s:spacetriang}

Let $\disc$ be a (fixed) closed disc of circumference $1$, and $\S:=\partial\disc$ the circle.
As usual, for a subset $A\subseteq\disc$, we denote by $\bar{A}$, $\interior{A}$, $\partial A$ and $\conv(A)$ the
closure, the interior, the boundary and the convex hull of $A$, respectively.
Furthermore, let
\begin{equation}
\label{e:tri}
   \tri(A):=\bigl\{\text{connected components of $\conv(A)\setminus A$}\bigr\},
\end{equation}
and
\begin{equation}
\label{e:nabla}
   \circseg(A):=\bigl\{\text{connected components of $\disc\setminus\conv(A)$}\bigr\}.
 \end{equation}
Then we have the disjoint decomposition $\disc = A \uplus \bigcup\tri(A) \uplus \bigcup\circseg(A)$.

\begin{definition}[(sub-)triangulations of the circle]
	A \define{sub-triangulation of the circle} is a closed, non-empty subset $C$ of $\disc$
	satisfying the following two conditions:
\begin{enumerate}[\axiom(Tr{i}1)]
	\item $\tri(C)$ consists of open interiors of triangles.
	\item $C$ is the union of non-crossing (non-intersecting except at endpoints), possibly
		degenerate closed straight line segments with endpoints in $\S$.
\end{enumerate}
We denote the set of sub-triangulations of the circle by $\triang$, i.e.\
\begin{equation}
\label{e:tatT}
   \triang := \bigl\{\text{sub-triangulations of the circle}\},
\end{equation}
and call $C\in\triang$ \emph{triangulation of the circle} if and only if $\S\subseteq C$.
\label{d:triangfinite}
\end{definition}

In particular, (Tri1) implies that $\partial\conv(C) \subseteq C$, and we may call $C$ triangulation of
$\partial\conv(C)$. Given (Tri1), (Tri2) implies that $\circseg(A)$ consists of circular segments with the
bounding straight line excluded and the rest of the bounding arc included.
We want to point out that our definition of triangulation of the circle differs from the one given by Aldous in
\cite[Definition~1]{Aldous94b}. Namely, Aldous required only Condition~(Tri1).
For the characterization of triangulations of the circle as limits of triangulations of $n$-gons given in
Proposition~\ref{p:fintriapp} below, however, Condition~(Tri2) is necessary.
See Figure~\ref{f:nontriang} for
an example of a triangulation in the sense of Aldous that is excluded by Condition~(Tri2), a sub-triangulation
of the circle that is no triangulation of the circle, and a triangulation of the circle.

\begin{figure}[t]
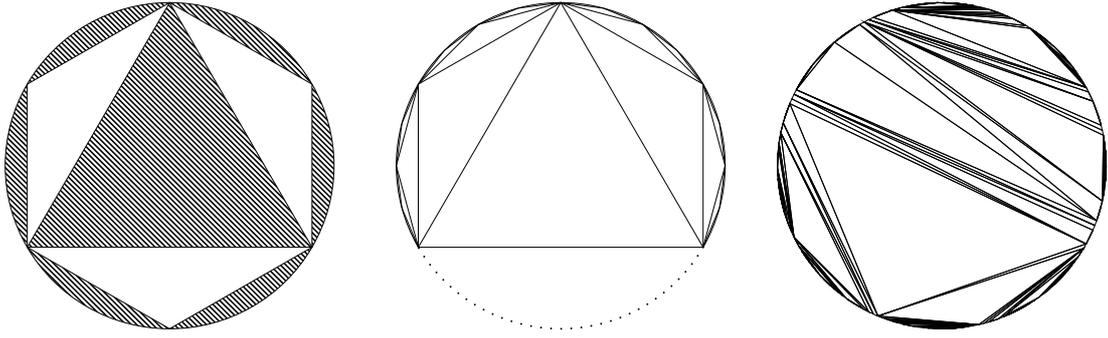

\begin{center}
\ifpdf
	\includegraphics{pretriang} \hfil \includegraphics{subtriang} \hfil \includegraphics{browntriang}
\else
	\psset{unit=0.139\textwidth}
	\pictpretri \hfil \pictsubtri  \hfil \begin{pspicture}(-1.01,-1.01)(1.01,1.01)
\SpecialCoor
\pspolygon(1;51.940299)(1;53.731343)(1;55.522388)
\pspolygon(1;152.238806)(1;338.507463)(1;340.298507)
\pspolygon(1;318.805970)(1;320.597015)(1;322.388060)
\pspolygon(1;197.014925)(1;198.805970)(1;200.597015)
\pspolygon(1;297.313433)(1;299.104478)(1;300.895522)
\pspolygon(1;266.865672)(1;274.029851)(1;279.402985)
\pspolygon(1;78.805970)(1;82.388060)(1;84.179104)
\pspolygon(1;154.029851)(1;336.716418)(1;338.507463)
\pspolygon(1;37.611940)(1;46.567164)(1;57.313433)
\pspolygon(1;5.373134)(1;7.164179)(1;354.626866)
\pspolygon(1;211.343284)(1;213.134328)(1;214.925373)
\pspolygon(1;26.865672)(1;28.656716)(1;116.417910)
\pspolygon(1;130.746269)(1;141.492537)(1;152.238806)
\pspolygon(1;1.791045)(1;354.626866)(1;356.417910)
\pspolygon(1;132.537313)(1;134.328358)(1;136.119403)
\pspolygon(1;340.298507)(1;345.671642)(1;347.462687)
\pspolygon(1;68.059701)(1;69.850746)(1;107.462687)
\pspolygon(1;146.865672)(1;150.447761)(1;152.238806)
\pspolygon(1;200.597015)(1;202.388060)(1;204.179104)
\pspolygon(1;266.865672)(1;270.447761)(1;274.029851)
\pspolygon(1;91.343284)(1;94.925373)(1;98.507463)
\pspolygon(1;234.626866)(1;238.208955)(1;240.000000)
\pspolygon(1;261.492537)(1;279.402985)(1;281.194030)
\pspolygon(1;308.059701)(1;311.641791)(1;313.432836)
\pspolygon(1;191.641791)(1;193.432836)(1;197.014925)
\pspolygon(1;188.059701)(1;189.850746)(1;197.014925)
\pspolygon(1;59.104478)(1;60.895522)(1;62.686567)
\pspolygon(1;256.119403)(1;257.910448)(1;259.701493)
\pspolygon(1;159.402985)(1;161.194030)(1;162.985075)
\pspolygon(1;107.462687)(1;109.253731)(1;112.835821)
\pspolygon(1;12.537313)(1;14.328358)(1;118.208955)
\pspolygon(1;94.925373)(1;96.716418)(1;98.507463)
\pspolygon(1;286.567164)(1;290.149254)(1;329.552239)
\pspolygon(1;168.358209)(1;171.940299)(1;205.970149)
\pspolygon(1;48.358209)(1;50.149254)(1;57.313433)
\pspolygon(1;295.522388)(1;300.895522)(1;302.686567)
\pspolygon(1;180.895522)(1;182.686567)(1;197.014925)
\pspolygon(1;356.417910)(1;358.208955)(1;360.000000)
\pspolygon(1;164.776119)(1;245.373134)(1;247.164179)
\pspolygon(1;205.970149)(1;234.626866)(1;241.791045)
\pspolygon(1;78.805970)(1;80.597015)(1;82.388060)
\pspolygon(1;50.149254)(1;55.522388)(1;57.313433)
\pspolygon(1;171.940299)(1;204.179104)(1;205.970149)
\pspolygon(1;136.119403)(1;137.910448)(1;139.701493)
\pspolygon(1;325.970149)(1;327.761194)(1;329.552239)
\pspolygon(1;30.447761)(1;32.238806)(1;57.313433)
\pspolygon(1;189.850746)(1;191.641791)(1;197.014925)
\pspolygon(1;23.283582)(1;116.417910)(1;118.208955)
\pspolygon(1;290.149254)(1;291.940299)(1;329.552239)
\pspolygon(1;7.164179)(1;351.044776)(1;352.835821)
\pspolygon(1;293.731343)(1;295.522388)(1;302.686567)
\pspolygon(1;247.164179)(1;248.955224)(1;282.985075)
\pspolygon(1;91.343284)(1;100.298507)(1;103.880597)
\pspolygon(1;291.940299)(1;304.477612)(1;306.268657)
\pspolygon(1;125.373134)(1;127.164179)(1;128.955224)
\pspolygon(1;69.850746)(1;73.432836)(1;107.462687)
\pspolygon(1;291.940299)(1;325.970149)(1;329.552239)
\pspolygon(1;211.343284)(1;214.925373)(1;216.716418)
\pspolygon(1;313.432836)(1;318.805970)(1;322.388060)
\pspolygon(1;205.970149)(1;209.552239)(1;232.835821)
\pspolygon(1;28.656716)(1;114.626866)(1;116.417910)
\pspolygon(1;291.940299)(1;293.731343)(1;304.477612)
\pspolygon(1;1.791045)(1;5.373134)(1;354.626866)
\pspolygon(1;10.746269)(1;347.462687)(1;349.253731)
\pspolygon(1;130.746269)(1;152.238806)(1;340.298507)
\pspolygon(1;10.746269)(1;118.208955)(1;347.462687)
\pspolygon(1;247.164179)(1;282.985075)(1;329.552239)
\pspolygon(1;7.164179)(1;10.746269)(1;349.253731)
\pspolygon(1;234.626866)(1;236.417910)(1;238.208955)
\pspolygon(1;125.373134)(1;128.955224)(1;130.746269)
\pspolygon(1;315.223881)(1;317.014925)(1;318.805970)
\pspolygon(1;46.567164)(1;48.358209)(1;57.313433)
\pspolygon(1;146.865672)(1;148.656716)(1;150.447761)
\pspolygon(1;14.328358)(1;23.283582)(1;118.208955)
\pspolygon(1;306.268657)(1;324.179104)(1;325.970149)
\pspolygon(1;286.567164)(1;288.358209)(1;290.149254)
\pspolygon(1;168.358209)(1;205.970149)(1;243.582090)
\pspolygon(1;342.089552)(1;343.880597)(1;345.671642)
\pspolygon(1;250.746269)(1;281.194030)(1;282.985075)
\pspolygon(1;265.074627)(1;266.865672)(1;279.402985)
\pspolygon(1;184.477612)(1;186.268657)(1;188.059701)
\pspolygon(1;180.895522)(1;197.014925)(1;200.597015)
\pspolygon(1;193.432836)(1;195.223881)(1;197.014925)
\pspolygon(1;209.552239)(1;231.044776)(1;232.835821)
\pspolygon(1;75.223881)(1;103.880597)(1;107.462687)
\pspolygon(1;130.746269)(1;139.701493)(1;141.492537)
\pspolygon(1;14.328358)(1;16.119403)(1;17.910448)
\pspolygon(1;340.298507)(1;342.089552)(1;345.671642)
\pspolygon(1;7.164179)(1;349.253731)(1;351.044776)
\pspolygon(1;205.970149)(1;241.791045)(1;243.582090)
\pspolygon(1;266.865672)(1;268.656716)(1;270.447761)
\pspolygon(1;274.029851)(1;275.820896)(1;277.611940)
\pspolygon(1;57.313433)(1;59.104478)(1;107.462687)
\pspolygon(1;28.656716)(1;57.313433)(1;107.462687)
\pspolygon(1;247.164179)(1;329.552239)(1;331.343284)
\pspolygon(1;7.164179)(1;352.835821)(1;354.626866)
\pspolygon(1;295.522388)(1;297.313433)(1;300.895522)
\pspolygon(1;75.223881)(1;77.014925)(1;103.880597)
\pspolygon(1;120.000000)(1;340.298507)(1;347.462687)
\pspolygon(1;32.238806)(1;34.029851)(1;57.313433)
\pspolygon(1;252.537313)(1;254.328358)(1;281.194030)
\pspolygon(1;209.552239)(1;222.089552)(1;231.044776)
\pspolygon(1;120.000000)(1;130.746269)(1;340.298507)
\pspolygon(1;155.820896)(1;157.611940)(1;331.343284)
\pspolygon(1;39.402985)(1;41.194030)(1;42.985075)
\pspolygon(1;14.328358)(1;19.701493)(1;23.283582)
\pspolygon(1;164.776119)(1;166.567164)(1;168.358209)
\pspolygon(1;234.626866)(1;240.000000)(1;241.791045)
\pspolygon(1;168.358209)(1;243.582090)(1;245.373134)
\pspolygon(1;211.343284)(1;218.507463)(1;220.298507)
\pspolygon(1;34.029851)(1;37.611940)(1;57.313433)
\pspolygon(1;10.746269)(1;12.537313)(1;118.208955)
\pspolygon(1;59.104478)(1;62.686567)(1;64.477612)
\pspolygon(1;261.492537)(1;263.283582)(1;279.402985)
\pspolygon(1;84.179104)(1;85.970149)(1;87.761194)
\pspolygon(1;254.328358)(1;256.119403)(1;281.194030)
\pspolygon(1;308.059701)(1;313.432836)(1;322.388060)
\pspolygon(1;157.611940)(1;164.776119)(1;247.164179)
\pspolygon(1;209.552239)(1;211.343284)(1;222.089552)
\pspolygon(1;284.776119)(1;286.567164)(1;329.552239)
\pspolygon(1;152.238806)(1;154.029851)(1;338.507463)
\pspolygon(1;73.432836)(1;75.223881)(1;107.462687)
\pspolygon(1;211.343284)(1;220.298507)(1;222.089552)
\pspolygon(1;333.134328)(1;334.925373)(1;336.716418)
\pspolygon(1;173.731343)(1;180.895522)(1;204.179104)
\pspolygon(1;184.477612)(1;188.059701)(1;197.014925)
\pspolygon(1;84.179104)(1;87.761194)(1;89.552239)
\pspolygon(1;28.656716)(1;107.462687)(1;114.626866)
\pspolygon(1;50.149254)(1;51.940299)(1;55.522388)
\pspolygon(1;19.701493)(1;21.492537)(1;23.283582)
\pspolygon(1;14.328358)(1;17.910448)(1;19.701493)
\pspolygon(1;23.283582)(1;26.865672)(1;116.417910)
\pspolygon(1;157.611940)(1;162.985075)(1;164.776119)
\pspolygon(1;291.940299)(1;306.268657)(1;325.970149)
\pspolygon(1;222.089552)(1;223.880597)(1;227.462687)
\pspolygon(1;37.611940)(1;39.402985)(1;42.985075)
\pspolygon(1;78.805970)(1;89.552239)(1;91.343284)
\pspolygon(1;109.253731)(1;111.044776)(1;112.835821)
\pspolygon(1;175.522388)(1;179.104478)(1;180.895522)
\pspolygon(1;270.447761)(1;272.238806)(1;274.029851)
\pspolygon(1;78.805970)(1;91.343284)(1;103.880597)
\pspolygon(1;263.283582)(1;265.074627)(1;279.402985)
\pspolygon(1;293.731343)(1;302.686567)(1;304.477612)
\pspolygon(1;222.089552)(1;227.462687)(1;229.253731)
\pspolygon(1;91.343284)(1;93.134328)(1;94.925373)
\pspolygon(1;205.970149)(1;207.761194)(1;209.552239)
\pspolygon(1;7.164179)(1;8.955224)(1;10.746269)
\pspolygon(1;180.895522)(1;200.597015)(1;204.179104)
\pspolygon(1;175.522388)(1;177.313433)(1;179.104478)
\pspolygon(1;130.746269)(1;136.119403)(1;139.701493)
\pspolygon(1;259.701493)(1;261.492537)(1;281.194030)
\pspolygon(1;154.029851)(1;331.343284)(1;336.716418)
\pspolygon(1;59.104478)(1;64.477612)(1;107.462687)
\pspolygon(1;211.343284)(1;216.716418)(1;218.507463)
\pspolygon(1;66.268657)(1;68.059701)(1;107.462687)
\pspolygon(1;164.776119)(1;168.358209)(1;245.373134)
\pspolygon(1;130.746269)(1;132.537313)(1;136.119403)
\pspolygon(1;168.358209)(1;170.149254)(1;171.940299)
\pspolygon(1;1.791045)(1;3.582090)(1;5.373134)
\pspolygon(1;42.985075)(1;44.776119)(1;46.567164)
\pspolygon(1;78.805970)(1;84.179104)(1;89.552239)
\pspolygon(1;331.343284)(1;333.134328)(1;336.716418)
\pspolygon(1;107.462687)(1;112.835821)(1;114.626866)
\pspolygon(1;154.029851)(1;155.820896)(1;331.343284)
\pspolygon(1;103.880597)(1;105.671642)(1;107.462687)
\pspolygon(1;250.746269)(1;252.537313)(1;281.194030)
\pspolygon(1;145.074627)(1;146.865672)(1;152.238806)
\pspolygon(1;306.268657)(1;308.059701)(1;322.388060)
\pspolygon(1;100.298507)(1;102.089552)(1;103.880597)
\pspolygon(1;91.343284)(1;98.507463)(1;100.298507)
\pspolygon(1;120.000000)(1;121.791045)(1;130.746269)
\pspolygon(1;34.029851)(1;35.820896)(1;37.611940)
\pspolygon(1;173.731343)(1;175.522388)(1;180.895522)
\pspolygon(1;223.880597)(1;225.671642)(1;227.462687)
\pspolygon(1;121.791045)(1;125.373134)(1;130.746269)
\pspolygon(1;64.477612)(1;66.268657)(1;107.462687)
\pspolygon(1;118.208955)(1;120.000000)(1;347.462687)
\pspolygon(1;37.611940)(1;42.985075)(1;46.567164)
\pspolygon(1;313.432836)(1;315.223881)(1;318.805970)
\pspolygon(1;306.268657)(1;322.388060)(1;324.179104)
\pspolygon(1;28.656716)(1;30.447761)(1;57.313433)
\pspolygon(1;171.940299)(1;173.731343)(1;204.179104)
\pspolygon(1;121.791045)(1;123.582090)(1;125.373134)
\pspolygon(1;274.029851)(1;277.611940)(1;279.402985)
\pspolygon(1;157.611940)(1;247.164179)(1;331.343284)
\pspolygon(1;23.283582)(1;25.074627)(1;26.865672)
\pspolygon(1;248.955224)(1;250.746269)(1;282.985075)
\pspolygon(1;205.970149)(1;232.835821)(1;234.626866)
\pspolygon(1;182.686567)(1;184.477612)(1;197.014925)
\pspolygon(1;222.089552)(1;229.253731)(1;231.044776)
\pspolygon(1;141.492537)(1;143.283582)(1;145.074627)
\pspolygon(1;256.119403)(1;259.701493)(1;281.194030)
\pspolygon(1;1.791045)(1;356.417910)(1;360.000000)
\pspolygon(1;77.014925)(1;78.805970)(1;103.880597)
\pspolygon(1;308.059701)(1;309.850746)(1;311.641791)
\pspolygon(1;282.985075)(1;284.776119)(1;329.552239)
\pspolygon(1;157.611940)(1;159.402985)(1;162.985075)
\pspolygon(1;69.850746)(1;71.641791)(1;73.432836)
\pspolygon(1;141.492537)(1;145.074627)(1;152.238806)
\end{pspicture}
 \fi
\end{center}
\caption{\emph{left:} An Aldous-triangulation of the circle that is not a triangulation of the circle (Condition~(Tri2) does not hold as the
	black triangle in the middle is not the union of non-crossing straight lines with endpoints on the circle).
\emph{middle:} A sub-triangulation of the circle (compare with Example~\ref{ex:compbin}).
\emph{right:} A triangulation of the circle. It is a realisation of the Brownian triangulation (compare with
Example~\ref{ex:CRT}).
}
\label{f:nontriang}
\end{figure}

For a metric space $(X,d)$, let
\begin{equation}
	\F(X):= \set{F\subseteq X}{F\ne \emptyset,\, F\text{ closed}},
\end{equation}
and equip $\F(X)$ with the \emph{Hausdorff metric topology}. That is, we say that a
sequence $(F_n)_{n\in\N}$ converges to $F$ in $\F(X)$ if and only if for all $\varepsilon>0$ and all large
enough $n\in\N$,
\begin{equation}\label{e:Hausdorff}
   F^\eps_n\supseteq F  \qquad\mbox{and}\qquad  F^\eps\supseteq F_n,
\end{equation}
where  for all $A\in\F(X)$, as usual, $A^\eps:=\set{x\in X}{d(x,A)<\eps}$. It is well-known that if $(X,d)$ is
compact, then $\F(X)$ is a compact metrizable space as well.
As sub-triangulations of the circle are elements of $\F(\disc)$, we naturally equip $\triang$ with the Hausdorff
metric topology. A first observation is that $\triang$ is actually a closed, and therefore compact subspace of
$\F(\disc)$.

\begin{lemma}[compactness of $\triang$] \label{l:triangcomp}
	Both the space of triangulations of the circle, and the space\/ $\triang$ of sub-triangulations of the
	circle, are compact metrizable spaces in the Hausdorff metric topology.
\end{lemma}
\begin{proof}
	Because $\disc$ is compact, $\F(\disc)$ is compact as well, and it is sufficient to show that $\triang$
	and the set of triangulations of the circle are closed subsets of $\F(\disc)$.

	Let $C_n\in\triang$ with $C_n \tno C \in \F(\disc)$ in the Hausdorff metric topology.
	(Tri1) is easily seen to be a closed property, thus $C$ satisfies (Tri1).
	Let $L_n$ be a set of non-crossing line segments with endpoints in $\S$ such that $C_n=\bigcup L_n$.
	The closure of $L_n$ in $\F(\disc)$ has the same property (it possibly differs from $L_n$ by a set of
	degenerated one-point segments contained in non-degenerate segments of $L_n$), so we may assume $L_n$ is
	closed to begin with, so that $L_n \in \F(\F(\disc))$. Because $\F(\F(\disc))$ is compact, we may
	assume, taking a subsequence if necessary, that $L_n \to L$ for some $L\in\F(\F(\disc))$.
	Obviously, $L_n$ consists of non-crossing line segments with endpoints in $\S$.
	Because the union operator $\bigcup \colon \F(\F(\disc)) \to \F(\disc)$ is continuous, we have
	$\bigcup L = C$. In particular, (Tri2) holds for $C$, and $C\in \triang$.
	Obviously, also the property that $\S\subseteq C$ is preserved by Hausdorff metric limits, thus the
	set of triangulations of the circle is closed as well.
\end{proof}

We now show two characterizations of sub-triangulations of the circle. Namely, condition (Tri2) can be replaced
by existence of ``triangles in the middle'' which is the major technical ingredient for the construction of the
branch point map in the next subsection.
Furthermore, they are precisely the limits of finite sub-triangulations, where we consider a sub-triangulation
$C$ as \emph{finite} if $C\cap \S$ is a finite set, or equivalently, $C$ consists of finitely many line
segments.

\begin{proposition}[characterization of (sub-)triangulations]\label{p:fintriapp}
Let\/ $\emptyset\ne C\subseteq\disc$ be closed. Then the following are equivalent.
\begin{enumerate}[1.]
	\item\label{it:sub} $C$ is a sub-triangulation of the circle.
	\item\label{it:br} Condition~(Tri1) holds, all extreme points of\/ $\conv(C)$ are contained in\/ $\S$, and
	   \begin{enumerate}[\bf(Tr{i}1)']\setcounter{enumii}{1}
   	   \item For\/ $x,y,z\in\tri(C)\cup \circseg(C)$ pairwise distinct, there exists a unique\/
		$c_{xyz}\in\tri(C)$ such that\/ $x,y,z$ are subsets of pairwise different connected
		components of\/ $\disc\setminus\partial c_{xyz}$.
	   \end{enumerate}
	\item\label{it:approx} There exists a sequence\/ $(C_n)_{n\in\N}$ of finite sub-triangulations of the
		circle with\/ $C_n\tno C$ in the Hausdorff metric topology.
\end{enumerate}
Furthermore, $C$ is a triangulation of the circle if and only if\/ $C_n$ in \ref{it:approx}.\ can be chosen as a
triangulation of a regular\/ $n$-gon inscribed in\/ $\S$.
\end{proposition}

\begin{remark}[condition (Tri2)']
	That $x,y,z$ are subsets of different connected components of $\disc\setminus \partial c_{xyz}$ means
	that either\/ $c_{xyz}\in\{x,y,z\}$ and the two elements of\/ $\{x,y,z\}\setminus\{c_{xyz}\}$ are
	subsets of different connected components of\/ $\disc\setminus\closure{c_{xyz}}$,
	or\/ $c_{xyz}\not\in\{x,y,z\}$ and\/ $x,y,z$ are subsets of pairwise different connected components of\/
	$\disc\setminus\closure{c_{xyz}}$.
\end{remark}

\begin{proof}[Proof of Proposition~\ref{p:fintriapp}]
\istep{\ref{it:sub}}{\ref{it:br}}
	Because $C$ is the union of line segments with endpoints on $\S$, it is obvious that the extreme points
	of $\conv(C)$ are contained in $\S$.
	We have to show (Tri2)', so let $x,y,z\in\tri(C)\cup\circseg(C)$ be pairwise distinct and note that
	uniqueness is obvious. If one of the elements of $\{x,y,z\}$, say $x$, is such that the other two are
	subsets of two different connected components of $\disc\setminus\bar{x}$, then necessarily $x\in
	\tri(C)$, and $c_{xyz}:=x$ has the desired properties. So assume this is not the case.
	
	Fix a set $L$ of non-crossing, closed lines with endpoints in $\S$ such that $C=\bigcup L$. Define
	\begin{equation}
		L_x:=\set{\ell \in L}{\ell\text{ separates $x$ from $y \cup z$ in $\disc$}},
	\end{equation}
	note that $L_x \ne \emptyset$ because $y$ and $z$ are in the same connected component of
	$\disc\setminus\bar{x}$ by assumption, and order $L_x$ by distance from $x$.
	Similarly, define $L_y$ as set of lines separating $y$ from $x\cup z$ ordered by distance from $y$, and
	$L_z$ as set of lines separating $z$ from $x\cup y$, ordered by distance from $z$.
	Define $\ell_x := \sup L_x$, $\ell_y:=\sup L_y$, and $\ell_z := \sup L_z$, which exist because $C$ is
	closed. In particular, they are non-crossing, and because $\conv(C)\setminus C$ may only consist of
	triangles, they have to be the sides of some $c_{xyz}\in\tri(C)$ which has the desired properties.

\istep{\ref{it:br}}{\ref{it:approx}}
	Because the extreme points of $\conv(C)$ are on the circle, for every $x\in \circseg(C)$,
	the boundary $\partial_\disc x$ in $\disc$ is a single straight line with endpoints in $\S$.
	Let $(V_n)_{n\in\N}$ be an increasing sequence of finite subsets of $\tri(C)\cup\circseg(C)$ such that
	$c_{xyz} \in V_n$ for pairwise distinct $x,y,z \in V_n$, and $V_n \uparrow \tri(C)\cup \circseg(C)$.
	Let $A_n:=\disc\setminus \bigcup V_n$. Then $A_n \to C$ in the Hausdorff metric topology.
	Because $c_{xyz} \in V_n$ for distinct $x,y,z \in V_n$, the boundary of each of the finitely many
	connected components of $A_n \setminus \S$ consists of one or two line segments and one or two
	connected sub-arcs of $\S$. Therefore, there is a finite sub-triangulation $C_n\subseteq A_n$ of the
	circle with Hausdorff distance from $A_n$ less than $e^{-n}$. Thus $C_n \to C$.
\istep{\ref{it:approx}}{\ref{it:sub}}
	Obvious, because $\triang$ is a closed subset of $\F(\disc)$ by Lemma~\ref{l:triangcomp}.
\pstep{``Furthermore''}
	If $C_n$ is a triangulation of the $n$-gon, it contains the $n$-gon, and hence any Hausdorff metric
	limit as $n\to \infty$ contains the circle, and hence is a triangulation of the circle.
	That triangulations of the circle can be approximated by triangulations of regular
	$n$-gons is a slight modification of the arguments above. Details are left to the reader.
\end{proof}

The most prominent random tree is Aldous's Brownian CRT, which is the limit of uniform random trees. Similarly,
one can define the Brownian triangulation of the circle.
\begin{example}[Brownian triangulation]
	The uniform random triangulation of the $n$-gon converges in law with respect to the Hausdorff metric
	topology to the so-called \emph{Brownian triangulation} $\CCRT$,
	see \cite{Aldous94,Aldous94b,CurienKortchemski14}.  A realisation is shown in the right of
	Figure~\ref{f:nontriang}.  It has a.s.\ Hausdorff dimension $\frac32$ (see \cite{Aldous94}).
\label{ex:CRT}
\end{example}

\subsection{Coding binary measure trees with (sub-)triangulations of the circle}
\label{s:coding}

Given an algebraic tree $(T,c)$, recall the set of leaves $\lf(T,c)$, and the degree $\deg_{(T,c)}(v)$ of $v\in T$
from \eqref{e:leaf} and \eqref{e:degree}, respectively.
In this section, we are interested in the following subspace of the space of all binary algebraic measure trees.

\begin{definition}[our space $\Tbin$]\label{d:Tbin}
	Let $\Tbin \subseteq \T$ be the set of (equivalence classes of) algebraic measure trees
	$(T,c,\mu)$ with $(T,c)$ binary (i.e.\ $\deg_{(T,c)}(v)\le 3$ for all $v\in T$) and
	$\at(\mu)\subseteq \lf(T,c)$.
\end{definition}

The space $\Tbin$ is of particular interest to us, as it is invariant under the dynamics of the Aldous diffusion
on cladograms, the construction of which was one of the motivations for studying algebraic measure trees, and
because, as we will see, it is precisely the space of algebraic measure trees that can be coded by
sub-triangulations of the circle.

To illustrate the construction of the tree coded by a sub-triangulation, we first consider a triangulation $C$
of the regular $n$-gon into necessarily $n-2$ triangles (see Figure~\ref{f:triangtree}).
Here, the coded tree is the dual graph. That is, every triangle corresponds to a branch point of the tree,
and two branch points are connected by an edge if and only if the triangles share a common edge.
We then add a leaf for every edge of the $n$-gon and obtain a graph-theoretic binary tree with $n$ leaves
and $n-2$ internal vertices. Recall from Example~\ref{ex:finite} that the finite graph-theoretic tree
corresponds to a unique algebraic tree. We finally assign to each leaf mass $n^{-1}$ (which corresponds to the
length of the arcs of the circle connecting two endpoints of edges of the $n$-gon if we inscribe it in a circle
of unit length), and obtain an algebraic measure tree.

The main result of this section is that there is a natural, surjective coding map from $\triang$ onto $\T_2$,
which is also continuous.
To state that formally, we need further notation.
Given a sub-triangulation $C\subseteq\disc$, recall $\tri(C)$ and $\circseg(C)$ from \eqref{e:nabla} and
\eqref{e:tri}, respectively.
For $x\in\tri(C)\cup \circseg(C)$, and $y\subseteq\disc$ connected and disjoint from $\partial_\disc x$, where
$\partial_\disc$ denotes the boundary in the space $\disc$, let
\begin{equation}
\label{s:comp}
   \comp{x}{y} := \text{the connected component of $\disc\setminus\partial_\disc x$ which contains $y$}.
\end{equation}
For $x\in \tri(C)$, let $p_i(x)$, $i=1,2,3$, be the mid-points of the three arcs of $\S\setminus \partial x$,
and define
\begin{equation}
\label{e:Box}
	\Box(C) := \bset{\{p_i(x)\}}{x\in \tri(C),\, i\in \{1,2,3\},\, \comp{x}{\{p_i(x)\}} \subseteq C},
\end{equation}
as well as $\comp{p}{p}:=p$ for $p\in \Box(C)$ (see Figure~\ref{f:Box}).
Recall the definition of components $\Sub_v(w)$ in an algebraic tree from \eqref{e:005}.
\cfigure{f:Box}{
\ifpdf
	\includegraphics{pictBox}
\else
	\psset{unit=0.09\textwidth}
	\pictBox
\fi
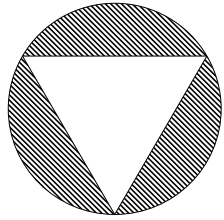
}{Triangulation $C$ with $\#\tri(C)=1$, $\circseg(C)=\emptyset$, and $\#\Box(C)=3$. The coded tree consists of
three line segments with non-atomic measure of $\frac13$ each, glued together at one branch point.}

\begin{lemma}[induced branch point map]\label{l:branchtriang}
	For\/ $C\in \triang$, let\/ $V_C:=\tri(C) \cup \circseg(C) \cup \Box(C)$.
	If\/ $V_C\ne \emptyset$, then there is a unique branch point map\/ $c_V\colon V_C^3\to V_C$ such that\/
	$(V_C, c_V)$ is an algebraic tree with\/ $\Sub_x^{(V_c,c_V)}(y) = \set{v\in V_C}{\comp{x}{y} =
	\comp{x}{v}}$ for\/ $x,y\in V_C$.
	Furthermore, $\deg(x) = 3$ for all\/ $x\in \tri(C)$, and\/ $\deg(x)=1$ for\/ $x\in \circseg(C)\cup \Box(C)$.
\end{lemma}
\begin{proof}
Recall from Proposition~\ref{p:fintriapp} that for a sub-triangulation $C$ of the circle and pairwise distinct
$x,y,z\in \tri(C)\cup \circseg(C)$, there is a triangle $c_{xyz}\in \tri(C)$ ``in the middle''.
It is straight-forward to see that this defines a branch point map and can naturally be extended to $V_C^3$.
\end{proof}

The following theorem states that all sub-triangulations $C$ of the circle can be associated with an element in $\T_2$ for which
$\tri(C)$ corresponds to the set of branch points, $\circseg(C)$ corresponds to the set
\begin{equation}
\label{e:033}
   \lfatom(\smallx):=\bset{x\in\lf(T,c)}{\mu(\{x\})>0}
\end{equation}
of leaves which carry an atom,
and $\comp{v}{w}$ corresponds to the component $\Sub_v(w)$.

\begin{theorem}[algebraic measure tree associated to a sub-triangulation] \label{t:tree}
\begin{enumerate}[(i)]
\item For every sub-triangulation\/ $C\subseteq\disc$ of the circle, there is a unique (up to equivalence)
	algebraic measure tree\/ $\smallx_C=(T_C,c_C,\mu_C)\in \Tbin$ with the following properties:
	\begin{enumerate}[\bf(CM1)]
	\item  $V_C \subseteq T_C$, $\br(\smallx_C)=\tri(C)$, and\/ $c_C$ is an extension
		of\/ $c_V$, where\/ $(V_C,c_V)$ is defined in Lemma~\ref{l:branchtriang}.
	\item $\mu_C\big(\Sub^{(T_C,c_C)}_x(y)\big) = \lambda_\S\big(\S\cap\comp{x}{y}\big)$
		for all\/ $x,y\in V_C$, where\/ $\lambda_\S$ denotes the Lebesgue measure on\/ $\S$.
	\item $\at(\mu_C) = \circseg(C)$.
	\end{enumerate}
\item The \emph{coding map} $\tree\colon \triang\to\Tbin$, $C \mapsto \smallx_C$ is surjective and continuous,
	where\/ $\triang$ is equipped with the Hausdorff metric topology and\/ $\Tbin$ with the bpdd-Gromov-weak
	topology.
\end{enumerate}
\end{theorem}

\begin{proof}
\proofcase{(i)} Let $C$ be a sub-triangulation of the circle.
If $C=\disc$, then $\tri(C)=\circseg(C)=\emptyset$, which requires by (CM1) that $\br(\smallx_C)=\emptyset$, and
by (CM3) that $\at(\mu)=\emptyset$.
There is a unique algebraic measure tree without branch points and atoms, namely the line segment with no atoms
(see Example~\ref{ex:extremal}).  We may therefore assume w.l.o.g.\ that $C\ne \disc$, and consequently that
$T_C\ne\emptyset$.

We claim that $(V_C,c_V)$ together with $\psi_y(x):= \lambda_\S\(\S\cap\comp{x}{y}\)$ satisfies the assumptions
of Proposition~\ref{p:construction}. Indeed, $V_C$ is obviously countable and an algebraic tree by
Lemma~\ref{l:branchtriang}, $\psi_y(x)$ depends on $y$ only through its equivalence class w.r.t.\ $\sim_x$, and
the lengths of all the arcs add up to the total length of $\lambda_\S(\S)=1$. Furthermore, $\psi_x(y) +
\psi_y(x) \ge \lambda_\S(\S)=1$, and Proposition~\ref{p:construction} yields existence and uniqueness of the
desired algebraic measure tree.

\proofcase{(ii)}
Let $\smallx=(T,c,\mu) \in \Tbin$. We construct a sub-triangulation $C$ such that $\tree(C)=\smallx$.
Fix $\rho\in\lf(T,c)$, and recall that $\rho$ induces a partial order relation $\le_\rho$.
We can extend this partial order to a total (planar) order $\le$ by picking for every $v\in \br(T,c)$ an order
of the two components of $T\setminus\{v\}$ that do not contain $\rho$. That is, we define
$S_0(v):=\Sub_v(\rho)$, denote the two remaining components of $T\setminus\{v\}$ by $S_1(v)$, $S_2(v)$, and
define
\begin{equation}\label{e:totord}
	v\le w \quad:\Leftrightarrow\quad v\le_\rho w \text{ or } v\in S_1\(c(x,y,\rho)\),\, w\in S_2\(c(x,y,\rho)\).
\end{equation}
Identify $\S$ with $[0,1]$, where the endpoints are glued. For $a\in [0,1]$ and $b,c>0$ with
	$a+b+c\le 1$, let $\Delta(a,b,c)\subseteq \disc$ be the open triangle with vertices $a,a+b,a+b+c\in \S$,
	$\ell(a,b)\subseteq \disc$ the straight line from $a$ to $a+b$, and $L(a,b)$ the connected component of
	$\disc\setminus \ell(a,b)$ containing $a + \frac b2\in \S$.
The first vertex of the triangle or circular segment corresponding to $v\in \br(T,c)\cup \lfatom(\smallx)$ is given
by the total mass before (w.r.t.\ $\le$ defined in \eqref{e:totord}), i.e.\ by
\begin{equation}\label{e:alpha}
	\alpha(v) := \mu\(\set{u\in T}{u<v}\).
\end{equation}
Define
\begin{equation}\label{e:025}
\begin{aligned}
		\disc\setminus C
 &:=
    \biguplus_{v\in\br(T,c)}\Delta\big(\alpha(v),\mu(S_1(v)),\mu(S_2(v))\big)
	\uplus
    \biguplus_{v\in\lfatom(\smallx)}L\big(\alpha(v),\mu\{v\}\big)
\end{aligned}
\end{equation}
By definition of $C$, $\conv(C)\setminus C$ consists of open triangles, i.e.\ condition (Tri1) is satisfied.
Furthermore, the extreme points of $\conv(C)$ are contained in $\S$, and for $x,y,z \in \tri(C) \cup \circseg(C)$
distinct, there are corresponding $u,v,w\in T$, and a triangle $c_{xyz} \in \tri(C)$ corresponding to
$c(u,v,w)$, which satisfies the requirements of (Tri2)'. Thus, by Proposition~\ref{p:fintriapp}, $C$ is a
sub-triangulation of the circle.
It is straight-forward to check that $\tree(C)=\smallx$.

We defer the proof of continuity of $\tree$ to the next section, where we prove it in Lemma~\ref{l:Fcont}.
\end{proof}

The following is obvious now.
\begin{lemma}[non-atomicity] A sub-triangulation\/ $C$ of the circle is a triangulation of the circle if and only
if, for\/ $(T_C,c_C,\mu_C):=\tree(C)$, the measure\/ $\mu_C$ is non-atomic.
\label{l:CRT}
\end{lemma}

\begin{corollary}[finite tree approximation]\label{c:fintreeapp}
	Let\/ $\smallx=(T,c,\mu) \in \Tbin$.
	Then there is a sequence\/ $(\smallx_n)_{n\in\N}$ of finite algebraic measure trees in\/ $\Tbin$ with\/
	$\smallx_n \to \smallx$ bpdd-Gromov-weakly. Furthermore, if\/ $\mu$ is non-atomic, then\/ $\smallx_n$
	can be chosen as a tree with\/ $n$ leaves and uniform distribution on the leaves.
\end{corollary}
\begin{proof}
	By Theorem~\ref{t:tree}, there is a sub-triangulation $C\in \triang$ with $\tree(C)=\smallx$, and by
	Proposition~\ref{p:fintriapp}, there are finite sub-triangulations $C_n$ with $C_n\to C$. Obviously,
	$\smallx_n:=\tree(C_n)$ is a finite algebraic measure tree and by continuity of $\tree$ we have $\smallx_n
	\to \smallx$. If $\mu$ is non-atomic, then, by Lemma~\ref{l:CRT}, $C$ is a triangulation of the circle,
	and hence, by Proposition~\ref{p:fintriapp}, $C_n$ can be chosen as triangulation of the $n$-gon, which
	means that $\smallx_n$ has $n$ leaves and uniform distribution on them.
\end{proof}

We conclude this section with a few illustrative examples.

\begin{figure}[t]
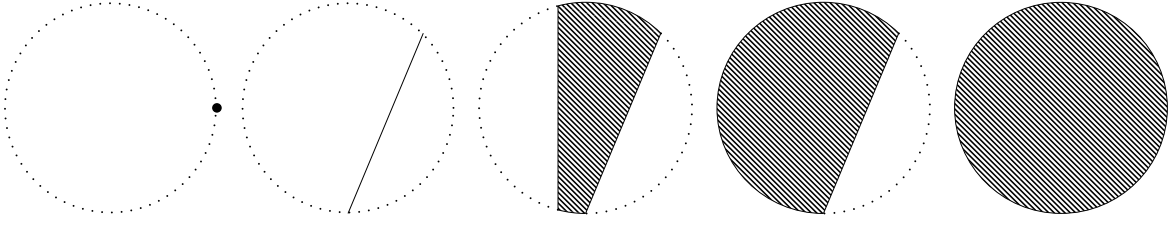

\begin{center}
\ifpdf
	\includegraphics{nobr1}\hfil
	\includegraphics{nobr2}\hfil
	\includegraphics{nobr3}\hfil
	\includegraphics{nobr4}\hfil
	\includegraphics{nobr5}
\else
	\picttrinobr
\fi
\end{center}
\caption{Sub-triangulations of the circle which correspond to the five cases of algebraic measure trees without branch points as explained in Example~\ref{ex:extremal}.}
\label{f:nobranch}
\end{figure}

\begin{example}[coding algebraic measure trees without branch points]\label{ex:extremal}
Let $\smallx$ be an algebraic measure tree without branch points. If $\smallx=\smallx_C$ for some
sub-triangulation $C$, then $\tri(C)= \br(\smallx_C) = \emptyset$ and the following five cases can occur (see Figure~\ref{f:nobranch}):
a) $\smallx_C$ consists of one single point of mass $1$. Then $C=\{x\}$ for some $x\in\S$.
b) $\smallx_C$ consists of an interval with two leaves, where each carries positive mass adding up to $1$,
   in which case $C$ is a single line segment dividing the circle into two arcs with length corresponding to the
   masses of the two leaves.
c) $\smallx_C$ consists of an interval with two leaves, where each has positive mass adding up to $a<1$.
   In this case, $C$ is the area of the disc bounded by two distinct line segments and two arcs (possibly one of
   them degenerated) of $\S$, and the lengths of the remaining two arcs are given by the masses of the leaves.
d) $\smallx_C$ consists of an interval with two leaves, where one has positive mass $a<1$ and the other one has
   zero mass. Then $C$ is a circular segment with arc length $1-a$.
e) $\smallx_C$ consists of an interval with no atoms on the leaves, which implies $C=\disc$.
\end{example}

\begin{example}[a complete binary tree]\label{ex:compbin}
	Let $C$ be the sub-triangulation of the circle drawn in the middle of Figure~\ref{f:nontriang}. Then
	$\#\circseg(C)=\#\lfatom(\tree(C))=1$. We refer to this only leaf with positive mass as the root $\rho$,
	and obtain $\mu(\{\rho\})=\frac13$, corresponding to the length of the dotted arc.
	Moreover, $\tree(C)$ consists of a complete rooted binary tree in the sense of graph theory (with the
	convention that the root has degree one), together with an uncountable set of leaves given by the ends
	at infinity and carrying the remaining $\frac23$ of the mass.
\end{example}

\begin{example}[coding the Brownian CRT]
	Recall the Brownian triangulation $\CCRT$ from Example~\ref{ex:CRT}, which is defined as the limit in
	distribution of uniform random triangulations $C_n$ of the $n$-gon. A realization is shown in the right
	of Figure~\ref{f:nontriang}.
	It is easy to see that $\tree(C_n)$ is the uniform binary tree with $n$ leaves and uniform distribution on
	the leaves. Thus, by Theorem~\ref{t:tree}, the uniform binary tree converges bpdd-Gromov-weakly to
	$\tree(\CCRT)$.
	At this point it is not entirely clear that $\tree(\CCRT)$ is the algebraic measure tree
	induced by the metric measure Brownian CRT. We will see in Section~\ref{s:examples} that this
	is indeed the case.
\end{example}

\section{The subspace of binary algebraic measure trees}
\label{S:topo}
In this section we introduce in Subsections~\ref{s:convshape} and~\ref{s:convmass} with the \emph{sample shape
convergence} and the \emph{sample subtree-mass convergence} two more notions of convergence of algebraic measure
trees which seem more natural when thinking of algebraic trees as combinatorial objects. We then show in
Subsection~\ref{s.equivalence} that on $\Tbin$, both of these notions are equivalent to the bpdd-Gromov-weak
convergence. The main tools are a uniform Glivenko Cantelli argument, and that the coding map sending a
sub-triangulation of the circle to an element in $\Tbin$ is continuous.

\iffalse{
\begin{remark}[convergence of algebraic trees versus convergence of $\R$-trees]\label{r:convcomp}
	Let $\sxh_n \in \H$, and $\smallx_n \in \T$ the corresponding algebraic measure tree.
	Since the equivalence classes of $\T$ are strictly coarser than those of $\H$, it is obvious that
	any kind of convergence of $\smallx_n$ does not imply convergence of $\sxh_n$ in any Hausdorff topology
	on $\H$.

	But also conversely, convergence of $\sxh_n$ in Gromov-weak or Gromov-Hausdorff-weak topology cannot imply
	convergence of $\smallx_n$ in any Hausdorff topology as shown by the following simple example.
	Let $\sxh_n = (T,\frac1n r, \mu)$ for any $(T,r,\mu)\in\H$ with $\mu$ not supported by a single point.
	Then $\sxh_n$ converges to the one-point space. All $\smallx_n$, however, are identical (equivalent as
	algebraic measure trees), but different from  the one-point space.
\end{remark}
}\fi

\subsection{Convergence in distribution of sampled tree shapes}
\label{s:convshape}

The basic idea behind Gromov-weak convergence for metric measure spaces is to sample finite metric sub-spaces
with the sampling measure $\mu$ and then require these to converge in distribution. In this section, we propose
a corresponding construction for binary algebraic measure trees, where we sample finite tree shapes with $\mu$.

First, we have to make precise what we mean by ``tree shape'', which we understand to be a cladogram with the
peculiarity that leaves may carry more than one label. The multi-label case is necessary to allow for
sampling the same point several times due to a possible atom at that point.

\xymatfig{f:shape}
	{&u_1\xyedge[d] &&&&&&&&
  \\
	 \xyedge[r]&\node\xyedge[dr] &     &   &\node\xyedge[ul]\xyedge[ur]       &      &&&&&   u_1\xyedge[dr] & & u_3\xyedge[d] &
  \\
	     &&{v}\xyedge[r]&u_3\xyedge[r]&\node\xyedge[r]\xyedge[u]&u_4&&&&&     &{v_1}\xyedge[r]&{v_2}\xyedge[r]&u_4
   \\
	 &u_2\xyedge[ur] &     &   &  &           &&&&&   u_2\xyedge[ur] & & &}
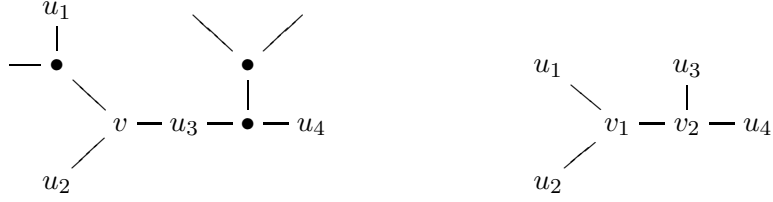
{A tree $T$ and the shape $\shape(u_1,u_2,u_3,u_4)$. Here, we are considering the homomorphism $f\colon C\to
c^3(\{u_1,...,u_4\}^3)$ given by $f(u_i):=u_i$, $i=1,...,4$, and then necessarily $f(v_1)=v$, $f(v_2)=u_3$.
$f$ is clearly no isomorphism, and the cladogram is not isomorphic to the subtree
$c(\{u_1,u_2,u_3,u_4\}^3)$ because $c(u_1,u_4,u_3)=u_3$.}

\begin{definition}[\nclad] For $m\in\N$, an \define{\nclad} is a binary, finite algebraic tree $C=(C,c)$ together with a surjective
	labelling map $\ell\colon \{1,...,m\} \to \lf(C)$.
	Two \nclads\ $(C_1,\ell_1)$ and $(C_2,\ell_2)$ are equivalent if they are label preserving
isomorphic i.e., there exists a tree isomorphism $f\colon C_1\to C_2$ with $f(\ell_1(i))=\ell_2(i)$ for all $i=1,...,m$.
\label{def:cladogram}
\end{definition}

Define
\begin{equation}
	\Clad := \{ \text{isomorphism classes of \nclads} \}.
\end{equation}

In the following we will use cladograms to encode the shape of a subtree spanned by a finite sample of leaves.
\begin{definition}[tree shape]
For a binary algebraic tree $(T,c)$, $m\in\N$, and $u_1,...,u_m\in T\setminus \br(T,c)$,
the \define{tree shape} $\shape(u_1,...,u_m)$
of the $m$-labelled cladogram spanned by $(u_1,...,u_m)$ in $(T,c)$
is the unique (up to isomorphism) \nclad\  $\shape(u_1,...,u_m)=(C,c_C,\ell)$ with
$\lf(C)=\{u_1,...,u_m\}$ and $\ell(i)=u_i$ for all $i=1,...,m$, and such that the identity on $\lf(C)$ extends to a tree
	homomorphism from $C$ onto $c\(\{u_1,...,u_m\}^3\)$.
\label{def:treeshape}
\end{definition}

\begin{remark}[spanned subtree and cladogram are not necessarily isomorphic]
	The tree homomorphism from $\shape(u_1,...,u_m)$ onto $c(\{u_1,...,u_m\}^3)$ does not need to be
	injective. This is the case if (and only if) $u_i\in \openint{u_j}{u_k}$ for some $i,j,k\in\{1,...,m\}$.
	See Figure~\ref{f:shape}.
\label{rem:002}
\end{remark}

\begin{example}[shape of a totally ordered algebraic tree]
	Let $(T,c)$ be a totally ordered algebraic tree, and $u_1,...,u_m\in T$.
	Then $\shape(u_1,...,u_m)$ is a so-called \emph{comb tree} which has a totally ordered spine of binary
	branch points with attached leaves (see Figure~\ref{f:ordered}).
\label{exp:003}
\end{example}

In the following, we build a topology on the convergence of tree shapes of $m$ randomly sampled points. We
therefore need the measurability of the shape map.

\begin{lemma}[measurability of the shape map]
	For every binary algebraic tree\/ $(T,c)$ and\/ $m\in\N$,
	the tree shape map\/ $\shape\colon (T\setminus\br(T,c))^m \to \Clad$ is a measurable function.
\label{l:001}
\end{lemma}
\begin{proof}
	Restricted to the open subset $\bset{v\in (T\setminus\br(T,c))^m}{v_1,...,v_m \text{ distinct}}$,
	$\shape$ is locally constant, hence continuous. The same is true on the set $\bset{v\in
	(T\setminus\br(T,c))^m}{v_1={v_2},\; v_2,..., v_m \text{ distinct}}$, which is an intersection of a closed
	and an open set, hence measurable. We can continue this way to see that $\shape$ is measurable on
	$(T\setminus \br(T,c))^m$.
\end{proof}

\xymatfig{f:ordered}
	{\xyedge[r] & u_1 \xyedge[r] & u_2 \xyedge[r] & u_3 \xyedge[r]  & u_4  \xyedge[r] & &&
 & u_1  \xyedge[r] & \node  \xyedge[r] & \node  \xyedge[r]  & u_4   & &&
 & u_1,u_5  \xyedge[r] & \node  \xyedge[r] & \node  \xyedge[r]  & u_4  
   \\
    &  &    &  &   &  &&
    &  &   u_2\xyedge[u] & u_3\xyedge[u] &    &  &&
    &  &   u_2\xyedge[u] & u_3\xyedge[u] &    
}
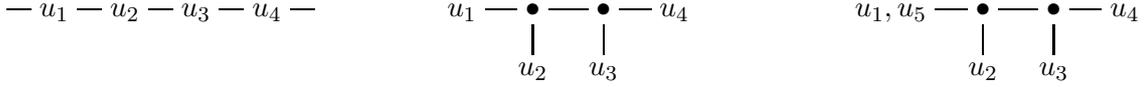
{The left shows a totally ordered binary algebraic tree and four distinct points $u_1,...,u_4$.
The middle shows the shape $\shape(u_1,...,u_4)$ of the cladogram which forms a comb tree. The right illustrates
what happens if a fifth point is equal to $u_1$. Now one of the leaves of $\shape(u_1,...,u_5)$ has two labels.}

\begin{definition}[tree shape distribution]
For $\smallx=(T,c,\mu) \in \Tbin$ and $m\in\N$, the $m$\nbd \define{tree shape distribution} of $\smallx$ is
defined by
\begin{equation}
\label{e:131}
    \shapedist(\smallx):=\mu^{\otimes m} \circ \shape^{-1}\in\CM_1(\Clad).
\end{equation}
\end{definition}

\begin{example}[shape of the linear non-atomic measure tree] \label{exp:005}
	Let $\smallx=(T,c,\mu)$ be the linear non-atomic algebraic measure tree
	(Example~\ref{ex:lintree}).  Then any sample $(u_1,...,u_m)$ with $\mu$ consists of pairwise different
	points, and $\shapedist(\smallx)$ is the mixture of Dirac measures on labelled comb trees where the
	mixture is over all (up to isometry) permutations of the labels.
\end{example}

We refer to the weakest topology on $\Tbin$ such that for every $m\in \N$ the $m$-tree shape distribution is
continuous as sample shape topology.
\begin{definition}[sample shape topology]\label{d:shapeconv}
	The topology induced on $\Tbin$ by the set $\set{\shapedist}{m\in\N}$ of tree shape
	distributions is called \define{sample shape topology}.
\end{definition}
We say that a sequence $(\smallx_n)_{n\in\N}$ is \define{sample shape convergent} to $\smallx$ in $\Tbin$ if it
converges w.r.t.\ the sample shape topology, i.e.\ if $\shapedist(\smallx_n)$ converges to $\shapedist(\smallx)$
as $n\to\infty$ for every $m\in\N$.

In analogy to the set $\ePol$ of polynomials introduced in Remark~\ref{rem:bpddGw}, we also introduce a set of
test functions which evaluate the tree shape distributions. We refer to
$\Phi=\Phi^{m,\varphi}\colon \Tbin\to\R$,
\begin{equation}
\label{e:spol}
		\Phi(\smallx) = \inta{\Clad}{\varphi\,}{\shapedist(\smallx)}
			=\intamu{T^m}{\varphi\circ\shape},
\end{equation}
where $m\in\N$ and $\varphi\colon \Clad\to \R$, as \define{shape polynomial}. We also define
\begin{equation} \label{e:141}
 	\sPol := \{\,\text{shape polynomials on $\Tbin$}\,\}.
\end{equation}
Obviously, the sample shape topology is induced by the set $\sPol$ of shape polynomials.

\begin{proposition}[sample shape implies bpdd-Gromov-weak convergence]\label{p:shapestrongerGw}
	On\/ $\Tbin$, the sample shape topology is stronger than the bpdd-Gromov-weak topology (i.e.\ any
	open set in the bpdd-Gromov-weak topology is open in the sample shape topology).
\end{proposition}

\begin{proof}
	The bpdd-Gromov-weak topology is induced by the set $\ePol$ of polynomials (see Remark~\ref{rem:bpddGw}).
	Because the set of $\phi\in\Cb(\R^{m\times m})$ which are Lipschitz continuous is convergence determining for
	probability measures on $\R^{m\times m}$, the subset of those $\Psi\in\ePol$ with
\begin{equation}\label{e:132}
	\Psi(T,c,\mu) =
		\intamuu{T^m}{\phi\((\nu[u_i,u_j] - \tfrac12\nu\{u_i\} - \tfrac12\nu\{u_j\})_{i,j=1,\ldots,m}\)}{\uu}
\end{equation}
	for some $m\in\N$ and Lipschitz continuous $\phi\in\Cb(\R^{m\times m})$ also induces the
	bpdd-Gromov-weak topology. Therefore, it is enough to show that such a $\Psi$ is continuous on $\Tbin$
	w.r.t.\ the sample shape topology. We do so by showing that the restriction to $\Tbin$ of $\Psi$ is in
	the uniform closure of $\sPol$. Let $L$ be the Lipschitz constant of $\phi$ w.r.t.\ the
	$\ell_\infty$-norm on $\R^{m\times m}$.

	For $n\in\N$ with $3n\ge m$, we define
\begin{equation} \label{e:133}
	\Phi_n(T,c,\mu) := \intamuu[3n]{T^{3n}}
		{\phi\((\nu_{n,\uu}[u_i,u_j] - \tfrac12\nu_{n,\uu}\{u_i\} - \tfrac12\nu_{n,\uu}\{u_j\})_{i,j=1,...,m}\)}{\uu},
\end{equation}
with the empirical branch point distribution
\begin{equation}
\label{e:134}
		\nu_{n,\uu} := \tfrac1n \sum_{k=0}^{n-1} \delta_{c(u_{3k+1},u_{3k+2},u_{3k+3})}.
	\end{equation}
Note that the restriction of $\Phi_n$ to $\Tbin$ belongs to $\sPol$ because whether or not $c(u_{k+1},u_{k+2},u_{k+3})$ lies on $[u_i,u_j]$,
	$k\in\{0,...,n-1\},\,i,j\in\{1,...,m\}$ only depends on the shape $\shape[3n](\uu)$.

Finally, we observe
	\begin{equation}
		\| \Psi - \Phi_n\|_\infty \le \sup_{(T,c,\mu)\in \Tbin} \intamuu[3n]{T^{3n}}
			{  L\cdot 3\sup_{I\in\I} |\nu(I) - \nu_{n,\uu}(I)| }{\uu}
			\le 3L\cdot \eps_n \tno 0,
	\end{equation}
	with $\I:=\{[x,y];\,x,y\in T\}$ and $(\eps_n)_{n\in\N}\tno 0$, where we have used a uniform
	Glivenko-Cantelli estimate which upper bounds the distance of the empirical branch point distribution to
	the branch point distribution.  Such an estimate should be known, but as we could not come up with a
	reference, we show it in Lemma~\ref{l:VCestim} in the appendix.
	We note that $\dimVC(\I)=2$ (compare Example~\ref{exp:004}).
\end{proof}

\begin{corollary}[metrizability] \label{c:shapemetrizable}
	The sample shape topology is metrizable.
\end{corollary}
\begin{proof}
	Because the sample shape topology is induced by a countable family of functions $(\shapedist)_{m\in\N}$
	with values in metrizable spaces, it is pseudo-metrizable. By Proposition~\ref{p:shapestrongerGw}, it is
	stronger than the bpdd-Gromov-weak topology, hence a Hausdorff topology. Therefore, it is metrizable.
\end{proof}

\subsection{Convergence in distribution of sampled subtree masses}
\label{s:convmass}
In this subsection, we introduce yet another notion of convergence of algebraic measure trees which, in contrast to
sampling tree shapes, is based on sampling branch points and evaluating the masses of the subtrees that are
joined at these branch points.
This approach might be more similar to the case of metric measure spaces and distance matrix distributions,
because we sample a tensor of real numbers (masses of subtrees) as opposed to a combinatorial object (tree shape).
Thus, the typical tools of analysis are more readily applicable for the corresponding class of test functions.

Let $(T,c,\mu)\in\Tbin$, and recall from (\ref{e:005}) for $u,v,w\in T$ the subtree component $\Sub_{c(u,v,w)}(x)$ of
$T\setminus\{c(u,v,w)\}$ which contains $x\not=c(u,v,w)$. Here, we always take the component containing $x=u$,
and consider its mass
\begin{equation}\label{e:eta}
	\eta(u,v,w) := \one_{u\ne c(u,v,w)} \cdot \mu\(\Sub_{c(u,v,w)}(u)\).
\end{equation}	

\begin{lemma}[measurability of the subtree masses] \label{l:002}
	For every binary algebraic measure tree\/ $\smallx=(T,c,\mu)\in \Tbin$ and\/ $m\in\N$,
	the function\/ $\eta\colon T^3\to [0,1]$ is measurable.
\end{lemma}

\begin{proof}
	First, we claim that the map $\psi\colon T^2 \to [0,1]$,
	\begin{equation}
		\psi(u,v) := \one_{u\ne v} \cdot \mu\(\Sub_v(u)\)
	\end{equation}
	is lower semi-continuous. Indeed, let $(u_n,v_n)$ be a sequence converging to $(u,v)$. We may assume
	w.l.o.g.\ that $v\ne u$, $u_n\in\Sub_v(u)$, and either $v_n \not\in \Sub_v(u)$ for all $n\in\N$, or
	$v_n \in \Sub_v(u)$ for all $n\in\N$.
	In the first case, $\Sub_v(u) \subseteq \Sub_{v_n}(u_n)$, and hence $\psi(u,v) \le \psi(u_n,v_n)$.
	In the second case, for every $x\in \Sub_v(u)$ and $n \ge n_x$ sufficiently large, we have $u\in
	\Sub_{v_n}(u_n)$ and $v_n\not\in [x,u]$. This means $x\in \Sub_{v_n}(u)=\Sub_{v_n}(u_n)$ and hence
	\begin{equation}\label{e:}
		\psi(u,v) - \liminf_{n\to\infty} \psi(u_n,v_n) \le \lim_{n\to\infty}\mu\(\Sub_v(u)\setminus \Sub_{v_n}(u_n)\) = 0.
	\end{equation}
	Therefore, $\psi$ is lower semi-continuous.
	Because the branch point map $c$ is continuous due to Lemma~\ref{l:ccont}, the same applies to
	$\eta(u,v,w) = \psi((u,\, c(u,v,w))$, and $\eta$ is measurable.
\end{proof}

Given a vector $\uu=(u_1,\ldots,u_m)\in T^m$, $m\in\N$, we consider the masses of all the subtrees we obtain as branch
points of entries of $\uu$. To this end, let
\begin{equation}\label{e:ueta}
	\ueta(u,v,w) := \(\eta(u,v,w),\, \eta(v,u,w),\, \eta(w,u,v)\)
\end{equation}
and define the function\/ $\m\colon T^m\to \tensor$, given by
	\begin{equation} \label{e:m}
	   \m(\uu) := \( \ueta(u_i,u_j,u_k) \)_{1\le i < j < k \le m}
	\end{equation}

\begin{definition}[subtree-mass tensor distribution]\label{d:massdist}
For $\smallx=(T,c,\mu) \in \Tbin$ and $m\in\N$, the $m$\nbd\define{subtree-mass tensor distribution} of $\smallx$ is
defined by
\begin{equation}\label{e:massdist}
   \massdist(\smallx):=\mu^{\otimes m}\circ \m^{-1}\in \CM_1\(\tensor\),
\end{equation}
\end{definition}

\begin{example}[symmetric binary tree] Let for each $n\in\N$, $\smallx_n=(T_n,c_n,\mu_n)$ the symmetric binary
tree with $N=2^n$ leaves and the uniform distribution on the set of leaves. Then the $3$\nbd \define{subtree-mass tensor distribution} of $\smallx_n$ is equal to
\begin{equation}
\label{e:139}
\begin{aligned}
   \massdist[3](\smallx_n) &= \mu_n^{\otimes 3}\circ \m^{-1}
   \\
   &=\sum_{k=1}^{n-1} \,\frac{1-2^{-k}}{2^{k+1}}\,
   	\Bigl( \delta_{(\tfrac{1}{2^{k+1}},\tfrac{1}{2^{k+1}},1-\tfrac{1}{2^k})}
	  + \delta_{(\tfrac{1}{2^{k+1}},1-\tfrac{1}{2^k},\tfrac{1}{2^{k+1}})}
	  + \delta_{(1-\tfrac{1}{2^k},\tfrac{1}{2^{k+1}},\tfrac{1}{2^{k+1}})} \Bigr)
   \\
   &\phantom{{}={}} + \tfrac1N (1-\tfrac1N) \( \delta_{(\frac{1}{N},\frac{1}{N},1)}
   			+ \delta_{(\frac{1}{N},1,\frac{1}{N})} + \delta_{(1,\frac{1}{N},\frac{1}{N})} \)
	+\tfrac1{N^2}\delta_{(\frac{1}{N},\frac{1}{N},\frac{1}{N})}
   \\
   &\tno \sum_{k=1}^\infty \,\frac{1-2^{-k}}{2^{k+1}}\,
   	\Bigl( \delta_{(\tfrac{1}{2^{k+1}},\tfrac{1}{2^{k+1}},1-\tfrac{1}{2^k})}
	  + \delta_{(\tfrac{1}{2^{k+1}},1-\tfrac{1}{2^k},\tfrac{1}{2^{k+1}})}
	  + \delta_{(1-\tfrac{1}{2^k},\tfrac{1}{2^{k+1}},\tfrac{1}{2^{k+1}})} \Bigr)
\end{aligned}
\end{equation}
\label{exp:007}
\end{example}

\begin{remark}[$3$-subtree-mass tensor distribution is not enough]
	It is not enough to consider only the $3$\nbd \define{subtree-mass tensor distribution}. Indeed, $\massdist[3]$ cannot distinguish all
	non-isomorphic binary algebraic measure trees, i.e.\ it does not separate the points of $\Tbin$. To see
	this, take the tree from Figure~\ref{f:3nonsep} with uniform distribution on its $12$ leaves, and the
	same tree with the subtrees marked by $\times$ and $\circ$, respectively, exchanged.
	These two trees are clearly non-isomorphic, and because the
	two marked subtrees have the same number of leaves, every vertex in one tree corresponds to a vertex in
	the other with the same value for $\m$.
\label{rem:004}
\end{remark}
\xymatfig{f:3nonsep}{
     &\lfdr&     &\lfdl&     &     &     &     &\lfd  &\     &      &\lfdl\\
\lfdr&     &\brd &\lfdr&     &\lfdl&\lfdr&     &\Rbrdl&      &\Rbrdl&	  \\
     &\brr &\brdr&     &\brdl&\lfdr&     &\brdl&      &\brul &      &\lful\\
\lfur&     &     &\brrr&     &     &\node&     &      &      &\lful &
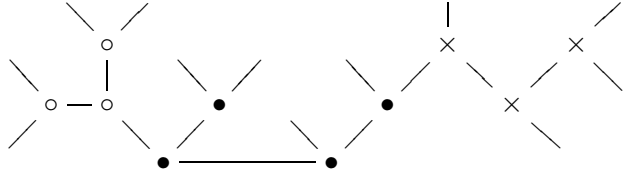
}{$\mu$ is the uniform distribution on the leaves. Swap the $\circ$-part with the $\times$-part to obtain a
non-isomorphic tree giving the same value for $\massdist[3]$.}

We consider the weakest topology on $\Tbin$ such that for every $m\in \N$ the $m$-subtree-mass tensor distribution is
continuous. Here, as usual, we equip $\CM_1(\tensor)$ with the weak topology.
\begin{definition}[sample subtree-mass topology]\label{d:massconv}
	The topology induced on $\Tbin$ by the set $\set{\massdist}{m\in\N}$ of subtree-mass tensor
	distributions is called \define{sample subtree-mass topology}.
\end{definition}
We say that a sequence $(\smallx_n)_{n\in\N}$ is \define{sample subtree-mass convergent} to $\smallx$ in $\Tbin$ if it
converges w.r.t.\ the sample subtree-mass topology, i.e.\ if $\massdist(\smallx_n)$ converges to $\massdist(\smallx)$
as $n\to\infty$ for every $m\in\N$.
To see that the sample subtree-mass topology is a Hausdorff topology on $\Tbin$,  we need the following reconstruction theorem.

\begin{proposition}[reconstruction theorem]
	The set of subtree-mass tensor distributions\/ $\set{\massdist}{m\in\N}$ separates points of\/ $\Tbin$,
	i.e., if\/ $\smallx_1,\smallx_2\in\Tbin$ are such that\/ $\massdist(\smallx_1)=\massdist(\smallx_2)$ for all\/
	$m\in\N$, then\/ $\smallx_1=\smallx_2$.
\label{p:reconstruction}
\end{proposition}

\begin{proof}
	We always assume that the representative $(T,c,\mu)$ of an algebraic measure tree is chosen such that
	$\mu(\Sub_{v}(u))>0$ whenever $u,v\in T$, $u\ne v$.

	Because the set $\set{\shapedist}{m\in\N}$ of tree shape distributions separates points by
	Corollary~\ref{c:shapemetrizable}, it is enough to show that $\shapedist$ is determined by the $m$-subtree-mass tensor
	distribution $\massdist$ for every $m\in\N$. We do so by showing that there exists a (non-continuous) function
	$h\colon \tensor\to\Clad$ such that for every $\smallx=(T,c,\mu)\in\Tbin$ we have $\shape=h\circ \m$ on
	$\(T\setminus \br(T,c)\)^m$.
	This is enough, because $\mu(\br(T,c))=0$ by countability of $\br(T,c)$ and the assumption that $\at(\mu)\subseteq \lf(T,c)$.

	Fix $\uu=(u_1,...,u_m)\in (T\setminus\br(T,c))^m$ and set $C=(C,c_C,\ell):=\shape(\uu)$. For $i\ne j$,
	we have $u_i=u_j$ if and only if $\eta(u_i,u_j,u_k)=\eta(u_j,u_i,u_k)=0$ for any and hence all $k\in
	\{1,...,m\}\setminus \{i,j\}$. Thus, we can determine multiple labels of $C$ by $\m(\uu)$ and may assume
	in the following that $u_1,...,u_m$ are distinct.
	Then, the \nclad\ $C$ is uniquely determined by
	the set of pairs $(\ux_1,\ux_2)$ of triples $\ux_i=(x_{i,1},x_{i,2},x_{i,3}) \in\{u_1,...,u_m\}^3$, $x_{i,j} \ne
	x_{i,k}$ for $j\ne k$, $i=1,2$, such that
	\begin{equation} \label{e:ceq}
		c_C\big(x_{1,1},\,x_{1,2},\,x_{1,3}\big) = c_C\big(x_{2,1},\, x_{2,2},\, x_{2,3}\big).
	\end{equation}
	We claim that \eqref{e:ceq} holds if and only if we can reorder the three entries of $\ux_2$ such that we
	can replace every entry of $\ux_1$ by the corresponding entry of $\ux_2$ and obtain the same masses of
	subtrees. More precisely,
	\begin{equation}\label{e:samem}
		\ueta(x_{1,1},x_{1,2},x_{1,3}) = \ueta(x_{i,1},x_{j,2},x_{k,3}) \quad\forall i,j,k\in\{1,2\}.
	\end{equation}
	Indeed, if $c_C(\ux_1)=c_C(\ux_2)$, then $c(\ux_1)=c(\ux_2)$ by definition of $\shape$.
	Because none of the $u_i$ is a branch point, every component of $T\setminus \{c(\ux_1)\}$ contains
	precisely one of the $x_{1,i}$, as well as one of the $x_{2,i}$. We can reorder the entries of
	$x_2$ such that $x_{1,i}$ is in the same component as $x_{2,i}$, $i=1,\ldots,3$. Then it is easy to
	check that \eqref{e:samem} holds.
	
	Conversely, assume that $c_C(\ux_1) \ne c_C(\ux_2)$. Because the
	restriction of the tree homomorphism $C \to c(\{u_1,\ldots,u_m\}^3)$ to the branch points of $C$ is
	injective, this implies $v_1 := c(\ux_1) \ne c(\ux_2)=:v_2$. There must be an $i$ with
	$x_{1,i} \in \Sub_{v_1}(v_2)$, say $i=3$. Also, $x_{2,j}\in \Sub_{v_1}(v_2)$ for at least two different
	$j$, so at least one which is different from $i$, say $j=2$ (see Figure~\ref{f:twocase}). Then
	$v_3:=c(x_{1,1}, x_{2,2}, x_{1, 3}) \in \Sub_{v_1}(v_2)$, and in particular, $x_{1,1},x_{1,2} \in
	\Sub_{v_3}(x_{1,1})$. Thus $\eta(\ux_1) < \eta(x_{1,1}, x_{2,2}, x_{1,3})$, and \eqref{e:samem} does not hold.
	\xymatfig{f:twocase}{
		x_{1,1}\xyedge[dr] &     & x_{1,3}\xyedge[d] &     & x_{2,2}\xyedge[dl] \\
		        & v_1\xyedge[r] & v_3\xyedge[r]     & v_2\xyedge[dr]           \\
		x_{1,2}\xyedge[ur] &    & & &
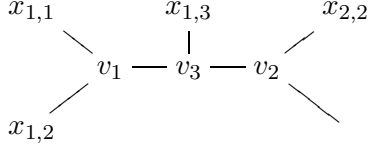
	}{The situation in the proof of Proposition~\ref{p:reconstruction}.}
\end{proof}

\begin{corollary}[metrizability]\label{c:massmetrizable}
	The sample subtree-mass topology is metrizable.
\end{corollary}
\begin{proof}
	Because the sample subtree-mass topology is induced by a countable family of functions $(\massdist)_{m\in\N}$
	with values in metrizable spaces, it is pseudo-metrizable. By Proposition~\ref{p:reconstruction}, it is
	a Hausdorff topology, hence it is metrizable.
\end{proof}

In analogy to the sets $\ePol$ and $\sPol$ of polynomials and shape polynomials, respectively, the sample
subtree-mass topology also comes with a canonical set of test functions.
We call $\Psi\colon \Tbin\to \R$ \define{subtree-mass polynomial} if there is $m\in\N$ and
$\psi\in \Cb(\tensor)$ with
\begin{equation}\label{e:mpol}
	\Psi(\smallx)
		= \inta{\tensor}{\psi\,}{\massdist(\smallx)}
		= \intamu{T^m}{\psi\circ \m}
\end{equation}
We also define
\begin{equation} \label{e:140}
 	\mPol := \{\,\text{subtree-mass polynomials on $\Tbin$}\,\}.
\end{equation}
Obviously, the sample subtree-mass topology is induced by the set $\mPol$ of subtree-mass polynomials.

\begin{proposition}[sample shape convergence implies sample subtree-mass convergence]
	The sample shape topology is stronger than the sample subtree-mass topology.
\label{p:shapestrongermass}
\end{proposition}

\begin{proof}
	The proof is similar to that of Proposition~\ref{p:shapestrongerGw}.
	We will show that each subtree-mass polynomial in $\Psi \in \mPol$,
\begin{equation}\label{e:142}
   \Psi(T,c,\mu) = \intamuu{T^m}{\psi\(\(\ueta(u_i,u_j,u_k)\)_{1\le i<j<k\le m}\)}{\uu},
\end{equation}
	with $m\in\N$ and $\psi\in\C(\tensor)$ Lipschitz continuous w.r.t.\ the
	$\ell_\infty$-Norm on $\tensor$ is in the uniform closure of $\sPol$.
	Let $L$ be the Lipschitz constant of $\Psi$.
For $n\in\N$ with $n\ge m$, we define
\begin{equation}
\label{e:143}
   \Phi_n(T,c,\mu) := \intamuu[n]{T^n}{\psi\(\(\ueta^{\mu_{n,\uu}}(u_i,u_j,u_k)\)_{1\le i<j<k\le m}\)}{\uu},
\end{equation}
where $\ueta^{\mu_{n,\uu}}$ is defined in the same way as $\ueta$, but with $\mu$ replaced by
the empirical sample distribution
\begin{equation} \label{e:144}
	\mu_{n,\underline{u}}:= \tfrac1n\sum_{\ell =1}^{n}\delta_{u_\ell}.
\end{equation}
Note that $\Phi_n\in\sPol$ because whether or not $u_\ell\in \Sub_{c(u_i,u_j,u_k)}(u_i)$ for some
$\ell\in\{1,...,n\},\; i,j,k\in\{1,...,m\}$ depends only on the shape $\shape(\uu)$.

Finally, applying the uniform Glivenko-Cantelli estimate Lemma~\ref{l:VCestim}, we have
\begin{equation}
\label{e:145}
   \begin{aligned}
		\| \Psi -\Phi_n \|_\infty \le \sup_{(T,c,\mu)\in \Tbin}\,
			\intamuu[n]{T^n}{L\cdot\sup_{S\in\Sset} \bigl|\mu(S)-\mu_{n,\uu}(S)\bigr|\,}{\uu} \le L \eps_n \tno 0,
	\end{aligned}
\end{equation}
where $\Sset:=\bset{\Sub_v(u)}{u,v\in T}$ and $(\eps_n)_{n\in\N}\tno 0$.
We note that $\dimVC(\Sset)\le 3$ (compare Example~\ref{exp:006}).
\end{proof}

\subsection{Equivalence and compactness of topologies}
\label{s.equivalence}
In this section, we show that sample shape convergence (Definition~\ref{d:shapeconv}), sample subtree-mass
convergence (Definition~\ref{d:massconv}) and branch point distribution distance Gromov-weak convergence
(Definition~\ref{d:bpddGw}) on $\Tbin$ are equivalent.
While spaces of metric measure spaces are usually far from being locally compact, $\Tbin$ is in this topology
even a compact metrizable space.

\begin{theorem}[equivalence of topologies and compactness]
	The sample shape topology, the sample subtree-mass topology, and the bpdd-Gromov-weak topology
	coincide on\/ $\Tbin$.
	Furthermore, $\Tbin$ is compact and metrizable in this topology.
\label{t:topeq}
\end{theorem}

Because compact subsets of a Hausdorff space are closed, a direct corollary is that unlike the situation in the
space of metric measure trees (with Gromov-weak or Gromov-Hausdorff-weak topology), the set of binary trees is
closed w.r.t.\ the bpdd-Gromov-weak topology. In particular, Gromov(-Hausdorff)-weak convergence does not imply
bpdd-Gromov-weak convergence of the induced trees.

\begin{cor}
	The subspace\/ $\Tbin$ of binary algebraic measure trees with atoms restricted to leaves is closed
	in\/ $\T$ (with bpdd-Gromov-weak topology).
\end{cor}

As a preparation of the proof for the theorem, we show that binary algebraic measure trees depend continuously on
their encoding as sub-triangulations of the circle. Together with Proposition~\ref{p:shapestrongerGw}, this also
finishes the proof of Theorem~\ref{t:tree}. Recall the space $\triang$ of sub-triangulations of the circle
equipped with the Hausdorff metric topology from \eqref{e:tatT}, and the coding map
$\tree\colon \triang \to \Tbin$ from Theorem~\ref{t:tree}.

\begin{lemma}[continuity of the coding map]
	Let\/ $\Tbin$ be equipped with the sample shape topology, and\/ $\triang$ with the
	Hausdorff metric topology. Then the coding map\/ $\tree\colon \triang\to\Tbin$ is continuous.
\label{l:Fcont}
\end{lemma}

\begin{proof} Fix $C\in\triang$ and $m\in\N$. By definition of the sample shape topology, it
	is enough to show that $\shapedist\circ\tree\colon \triang\to\CM_1(\Clad)$ is continuous at $C$.
	Let $U_1,...,U_m$ be i.i.d.\ points on the circle $\S$ chosen with the Lebesgue measure.

Recall from \eqref{e:nabla} the set $\circseg(C)$ of connected components of $\disc\setminus\conv(C)$,
from \eqref{s:comp} the connected component $\comp{x}{y}$ of $\disc\setminus\partial_\disc x$ which contains $y$,
where $x\in\tri(C)\cup \circseg(C)$, and $y\subseteq\disc$ connected and disjoint from $\partial_\disc x$.
Furthermore, recall the set $\Box(C)$ from \eqref{e:Box}, and the subtree components  $\Sub_x(y)$ from \eqref{e:equiv}.
	
	For $\eps>0$, there exists $N=N_{C,m,\eps}\in \N$ and $v_1,...,v_N\in
	\tri(C) \cup \circseg(C)$ distinct such that with probability at least $1-\eps$ the following holds:
\begin{itemize}
\item if $\{U_1,...,U_m\}\cap v\not=\emptyset$ for $v\in\circseg(C)$, then $v\in\{v_1,...,v_N\}$, and
\item if $\{U_1,...,U_m\}\cap \comp{v}{w}\not=\emptyset$ for some $v\in\tri(C)$ and all $w\in\tri(C)\cup\circseg(C)\cup\Box(C)$ with $w\not=v$, then $v\in\{v_1,...,v_N\}$.
\end{itemize}

\iffalse{
	for all $v\in \circseg(C)\setminus\{v_1,...,v_N\}$, we have $\{U_1,...,U_m\}\cap v = \emptyset$,
	and
	for all $v\in \tri(C) \setminus \{v_1,...,v_N\}$ there is $w\in (\tri(C)\cup \circseg(C)\cup
	\Box(C))\setminus\{v\}$ with $\{U_1,...,U_m\}\cap \comp{v}{w} = \emptyset$.
}\fi
	
Put $\eps':=\eps\cdot (12mN)^{-1}$. Then
\begin{equation}
\label{e:444}
   \mathbb{P}\big(\big\{d(U_i,\partial v_j)\ge\eps',\,\forall\,i=1,...,m;j=1,...,N\big\}\big)\ge 1-\eps.
\end{equation}

There is a $\delta=\delta(\eps)>0$ sufficiently small such that for any $C'\in \triang$ with Hausdorff metric
$\dH(C,C')<\delta$ there are distinct $v_1',...,v_N'\in \tri(C')\cup\circseg(C')$ such that
$\dH(v_i,v_i')\le\eps'$ for $i=1,...,N$.
Let $\smallx=(T,c,\mu):=\tree(C)$, and $V_1,...,V_m$ be i.i.d.\ $\mu$\nbd distributed, coupled to
$U_1,...,U_m$ such that $V_k\in\Sub_v(w)$ if and only if $U_k \in \comp{v}{w}$, which is possible due to the
properties of $\tree$ established in Theorem~\ref{t:tree}.
Define $\smallx'$ and $V_1',...,V_m'$ similarly with $C'$ instead of $C$.
Then
\begin{equation}
\label{e:120}
	\mathbb{P}\big(\big\{\shape(V_1,...,V_m)=\shape[T'](V_1',...,V_m')\big\}\big)\ge 1-2\eps,
\end{equation}
which implies that
$\dPr\(\shapedist(\tree(C)),\, \shapedist(\tree(C'))\) \le 2\eps$ (with $\dPr$ denoting the Prokhorov distance).
This shows that $\shapedist\circ\tree$ is continuous at $C$ and, since $m$ and $C$ are arbitrary, that $\tree$
is continuous.
\end{proof}

Now we are in a position to combine our results to a proof of the main theorem of Section~\ref{S:topo}.

\begin{proof}[Proof of Theorem~\ref{t:topeq}]
The space $\triang$ of sub-triangulations of the circle with Hausdorff metric topology is compact according to
Lemma~\ref{l:triangcomp}.
The coding map $\tree \colon \triang \to \Tbin$ is surjective by Theorem~\ref{t:tree},
and continuous when $\Tbin$ is equipped with the sample shape topology by Lemma~\ref{l:Fcont}.
Therefore, the sample shape topology is a compact topology on $\Tbin$. Moreover, the sample shape topology is Hausdorff by
Corollary~\ref{c:shapemetrizable}. As the sample subtree-mass topology is a weaker
Hausdorff topology by Proposition~\ref{p:shapestrongermass} and Corollary~\ref{c:massmetrizable}, it coincides with the sample shape topology.
The same is true for the bpdd-Gromov-weak topology by Proposition~\ref{p:shapestrongerGw}.
\end{proof}

Recall from Remark~\ref{rem:bpddGw} that the set of distance polynomials is convergence determining for measures on $\Tbin$.
It directly follows from the construction that the same is true for the sets of shape polynomials and subtree-mass polynomials.
This property is very useful for proving
convergence in law of random variables.

\begin{cor}[convergence determining classes of functions] \label{c:convdet}
	The sets\/ $\sPol\subseteq \Cb(\Tbin)$ (defined in \eqref{e:spol}) and\/ $\mPol$ (defined in \eqref{e:mpol}) are
convergence determining for measures on\/ $\Tbin$ with bpdd-Gromov-weak topology.
\end{cor}
\begin{proof}
	$\Tbin$ is a compact metrizable space, and both $\sPol$ and $\mPol$ induce the bpdd-Gromov-weak topology
	on $\Tbin$ by Theorem~\ref{t:topeq}. Furthermore, each of $\sPol$ and $\mPol$ is closed under multiplication.
	Thus the claim follows by the Stone-Weierstrass theorem.
\end{proof}

\section{Example: sampling consistent families}
\label{s:examples}
Consider a family $(T_n, c_n)_{n\in\N}$ of random, finite binary (algebraic) trees, where $(T_n,c_n)$ has $n$ leaves.
Let $K_n$ be the Markov kernel that takes such a tree and removes a leaf uniformly chosen at random,
together with the branch point it is attached to, thus obtaining a binary tree with $n-1$ leaves.
We say that the family is \define{sampling consistent} if $K_n(T_n,\,\cdot\,)=\law(T_{n-1})$, where $\law$
denotes the law of a random variable.

\begin{figure}[t]
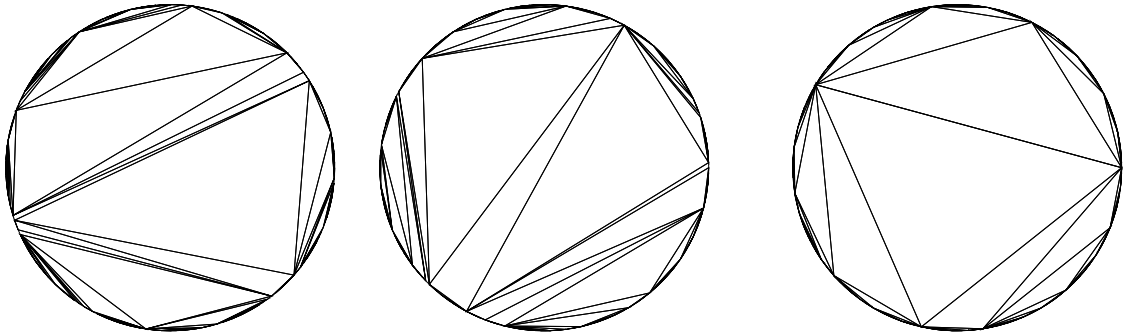

\begin{center}
\ifpdf
	\includegraphics{beta-1-triang}\hfil
	\includegraphics{yule-triang}\hfil
	\includegraphics{beta10-triang}
\else
	\psset{unit=0.139\textwidth}
	\begin{pspicture}(-1.01,-1.01)(1.01,1.01)
\SpecialCoor
\pspolygon(1;1.791045)(1;356.417910)(1;360.000000)
\pspolygon(1;1.791045)(1;10.746269)(1;356.417910)
\pspolygon(1;1.791045)(1;5.373134)(1;10.746269)
\pspolygon(1;1.791045)(1;3.582090)(1;5.373134)
\pspolygon(1;5.373134)(1;8.955224)(1;10.746269)
\pspolygon(1;5.373134)(1;7.164179)(1;8.955224)
\pspolygon(1;10.746269)(1;12.537313)(1;356.417910)
\pspolygon(1;12.537313)(1;318.805970)(1;356.417910)
\pspolygon(1;12.537313)(1;32.238806)(1;318.805970)
\pspolygon(1;12.537313)(1;21.492537)(1;32.238806)
\pspolygon(1;12.537313)(1;17.910448)(1;21.492537)
\pspolygon(1;12.537313)(1;14.328358)(1;17.910448)
\pspolygon(1;14.328358)(1;16.119403)(1;17.910448)
\pspolygon(1;17.910448)(1;19.701493)(1;21.492537)
\pspolygon(1;21.492537)(1;30.447761)(1;32.238806)
\pspolygon(1;21.492537)(1;26.865672)(1;30.447761)
\pspolygon(1;21.492537)(1;23.283582)(1;26.865672)
\pspolygon(1;23.283582)(1;25.074627)(1;26.865672)
\pspolygon(1;26.865672)(1;28.656716)(1;30.447761)
\pspolygon(1;32.238806)(1;198.805970)(1;318.805970)
\pspolygon(1;32.238806)(1;197.014925)(1;198.805970)
\pspolygon(1;32.238806)(1;35.820896)(1;197.014925)
\pspolygon(1;32.238806)(1;34.029851)(1;35.820896)
\pspolygon(1;35.820896)(1;44.776119)(1;197.014925)
\pspolygon(1;35.820896)(1;37.611940)(1;44.776119)
\pspolygon(1;37.611940)(1;39.402985)(1;44.776119)
\pspolygon(1;39.402985)(1;42.985075)(1;44.776119)
\pspolygon(1;39.402985)(1;41.194030)(1;42.985075)
\pspolygon(1;44.776119)(1;159.402985)(1;197.014925)
\pspolygon(1;44.776119)(1;82.388060)(1;159.402985)
\pspolygon(1;44.776119)(1;71.641791)(1;82.388060)
\pspolygon(1;44.776119)(1;69.850746)(1;71.641791)
\pspolygon(1;44.776119)(1;55.522388)(1;69.850746)
\pspolygon(1;44.776119)(1;53.731343)(1;55.522388)
\pspolygon(1;44.776119)(1;48.358209)(1;53.731343)
\pspolygon(1;44.776119)(1;46.567164)(1;48.358209)
\pspolygon(1;48.358209)(1;51.940299)(1;53.731343)
\pspolygon(1;48.358209)(1;50.149254)(1;51.940299)
\pspolygon(1;55.522388)(1;68.059701)(1;69.850746)
\pspolygon(1;55.522388)(1;66.268657)(1;68.059701)
\pspolygon(1;55.522388)(1;62.686567)(1;66.268657)
\pspolygon(1;55.522388)(1;57.313433)(1;62.686567)
\pspolygon(1;57.313433)(1;60.895522)(1;62.686567)
\pspolygon(1;57.313433)(1;59.104478)(1;60.895522)
\pspolygon(1;62.686567)(1;64.477612)(1;66.268657)
\pspolygon(1;71.641791)(1;75.223881)(1;82.388060)
\pspolygon(1;71.641791)(1;73.432836)(1;75.223881)
\pspolygon(1;75.223881)(1;78.805970)(1;82.388060)
\pspolygon(1;75.223881)(1;77.014925)(1;78.805970)
\pspolygon(1;78.805970)(1;80.597015)(1;82.388060)
\pspolygon(1;82.388060)(1;123.582090)(1;159.402985)
\pspolygon(1;82.388060)(1;85.970149)(1;123.582090)
\pspolygon(1;82.388060)(1;84.179104)(1;85.970149)
\pspolygon(1;85.970149)(1;89.552239)(1;123.582090)
\pspolygon(1;85.970149)(1;87.761194)(1;89.552239)
\pspolygon(1;89.552239)(1;103.880597)(1;123.582090)
\pspolygon(1;89.552239)(1;91.343284)(1;103.880597)
\pspolygon(1;91.343284)(1;102.089552)(1;103.880597)
\pspolygon(1;91.343284)(1;93.134328)(1;102.089552)
\pspolygon(1;93.134328)(1;94.925373)(1;102.089552)
\pspolygon(1;94.925373)(1;100.298507)(1;102.089552)
\pspolygon(1;94.925373)(1;96.716418)(1;100.298507)
\pspolygon(1;96.716418)(1;98.507463)(1;100.298507)
\pspolygon(1;103.880597)(1;121.791045)(1;123.582090)
\pspolygon(1;103.880597)(1;114.626866)(1;121.791045)
\pspolygon(1;103.880597)(1;107.462687)(1;114.626866)
\pspolygon(1;103.880597)(1;105.671642)(1;107.462687)
\pspolygon(1;107.462687)(1;112.835821)(1;114.626866)
\pspolygon(1;107.462687)(1;109.253731)(1;112.835821)
\pspolygon(1;109.253731)(1;111.044776)(1;112.835821)
\pspolygon(1;114.626866)(1;120.000000)(1;121.791045)
\pspolygon(1;114.626866)(1;118.208955)(1;120.000000)
\pspolygon(1;114.626866)(1;116.417910)(1;118.208955)
\pspolygon(1;123.582090)(1;155.820896)(1;159.402985)
\pspolygon(1;123.582090)(1;152.238806)(1;155.820896)
\pspolygon(1;123.582090)(1;127.164179)(1;152.238806)
\pspolygon(1;123.582090)(1;125.373134)(1;127.164179)
\pspolygon(1;127.164179)(1;128.955224)(1;152.238806)
\pspolygon(1;128.955224)(1;136.119403)(1;152.238806)
\pspolygon(1;128.955224)(1;134.328358)(1;136.119403)
\pspolygon(1;128.955224)(1;130.746269)(1;134.328358)
\pspolygon(1;130.746269)(1;132.537313)(1;134.328358)
\pspolygon(1;136.119403)(1;137.910448)(1;152.238806)
\pspolygon(1;137.910448)(1;141.492537)(1;152.238806)
\pspolygon(1;137.910448)(1;139.701493)(1;141.492537)
\pspolygon(1;141.492537)(1;146.865672)(1;152.238806)
\pspolygon(1;141.492537)(1;143.283582)(1;146.865672)
\pspolygon(1;143.283582)(1;145.074627)(1;146.865672)
\pspolygon(1;146.865672)(1;148.656716)(1;152.238806)
\pspolygon(1;148.656716)(1;150.447761)(1;152.238806)
\pspolygon(1;152.238806)(1;154.029851)(1;155.820896)
\pspolygon(1;155.820896)(1;157.611940)(1;159.402985)
\pspolygon(1;159.402985)(1;170.149254)(1;197.014925)
\pspolygon(1;159.402985)(1;161.194030)(1;170.149254)
\pspolygon(1;161.194030)(1;162.985075)(1;170.149254)
\pspolygon(1;162.985075)(1;168.358209)(1;170.149254)
\pspolygon(1;162.985075)(1;164.776119)(1;168.358209)
\pspolygon(1;164.776119)(1;166.567164)(1;168.358209)
\pspolygon(1;170.149254)(1;193.432836)(1;197.014925)
\pspolygon(1;170.149254)(1;189.850746)(1;193.432836)
\pspolygon(1;170.149254)(1;171.940299)(1;189.850746)
\pspolygon(1;171.940299)(1;186.268657)(1;189.850746)
\pspolygon(1;171.940299)(1;182.686567)(1;186.268657)
\pspolygon(1;171.940299)(1;179.104478)(1;182.686567)
\pspolygon(1;171.940299)(1;177.313433)(1;179.104478)
\pspolygon(1;171.940299)(1;173.731343)(1;177.313433)
\pspolygon(1;173.731343)(1;175.522388)(1;177.313433)
\pspolygon(1;179.104478)(1;180.895522)(1;182.686567)
\pspolygon(1;182.686567)(1;184.477612)(1;186.268657)
\pspolygon(1;186.268657)(1;188.059701)(1;189.850746)
\pspolygon(1;189.850746)(1;191.641791)(1;193.432836)
\pspolygon(1;193.432836)(1;195.223881)(1;197.014925)
\pspolygon(1;198.805970)(1;308.059701)(1;318.805970)
\pspolygon(1;198.805970)(1;202.388060)(1;308.059701)
\pspolygon(1;198.805970)(1;200.597015)(1;202.388060)
\pspolygon(1;202.388060)(1;204.179104)(1;308.059701)
\pspolygon(1;204.179104)(1;261.492537)(1;308.059701)
\pspolygon(1;204.179104)(1;241.791045)(1;261.492537)
\pspolygon(1;204.179104)(1;238.208955)(1;241.791045)
\pspolygon(1;204.179104)(1;236.417910)(1;238.208955)
\pspolygon(1;204.179104)(1;234.626866)(1;236.417910)
\pspolygon(1;204.179104)(1;225.671642)(1;234.626866)
\pspolygon(1;204.179104)(1;222.089552)(1;225.671642)
\pspolygon(1;204.179104)(1;214.925373)(1;222.089552)
\pspolygon(1;204.179104)(1;209.552239)(1;214.925373)
\pspolygon(1;204.179104)(1;207.761194)(1;209.552239)
\pspolygon(1;204.179104)(1;205.970149)(1;207.761194)
\pspolygon(1;209.552239)(1;213.134328)(1;214.925373)
\pspolygon(1;209.552239)(1;211.343284)(1;213.134328)
\pspolygon(1;214.925373)(1;220.298507)(1;222.089552)
\pspolygon(1;214.925373)(1;218.507463)(1;220.298507)
\pspolygon(1;214.925373)(1;216.716418)(1;218.507463)
\pspolygon(1;222.089552)(1;223.880597)(1;225.671642)
\pspolygon(1;225.671642)(1;229.253731)(1;234.626866)
\pspolygon(1;225.671642)(1;227.462687)(1;229.253731)
\pspolygon(1;229.253731)(1;231.044776)(1;234.626866)
\pspolygon(1;231.044776)(1;232.835821)(1;234.626866)
\pspolygon(1;238.208955)(1;240.000000)(1;241.791045)
\pspolygon(1;241.791045)(1;243.582090)(1;261.492537)
\pspolygon(1;243.582090)(1;259.701493)(1;261.492537)
\pspolygon(1;243.582090)(1;247.164179)(1;259.701493)
\pspolygon(1;243.582090)(1;245.373134)(1;247.164179)
\pspolygon(1;247.164179)(1;256.119403)(1;259.701493)
\pspolygon(1;247.164179)(1;252.537313)(1;256.119403)
\pspolygon(1;247.164179)(1;248.955224)(1;252.537313)
\pspolygon(1;248.955224)(1;250.746269)(1;252.537313)
\pspolygon(1;252.537313)(1;254.328358)(1;256.119403)
\pspolygon(1;256.119403)(1;257.910448)(1;259.701493)
\pspolygon(1;261.492537)(1;306.268657)(1;308.059701)
\pspolygon(1;261.492537)(1;286.567164)(1;306.268657)
\pspolygon(1;261.492537)(1;281.194030)(1;286.567164)
\pspolygon(1;261.492537)(1;266.865672)(1;281.194030)
\pspolygon(1;261.492537)(1;265.074627)(1;266.865672)
\pspolygon(1;261.492537)(1;263.283582)(1;265.074627)
\pspolygon(1;266.865672)(1;279.402985)(1;281.194030)
\pspolygon(1;266.865672)(1;275.820896)(1;279.402985)
\pspolygon(1;266.865672)(1;274.029851)(1;275.820896)
\pspolygon(1;266.865672)(1;272.238806)(1;274.029851)
\pspolygon(1;266.865672)(1;270.447761)(1;272.238806)
\pspolygon(1;266.865672)(1;268.656716)(1;270.447761)
\pspolygon(1;275.820896)(1;277.611940)(1;279.402985)
\pspolygon(1;281.194030)(1;284.776119)(1;286.567164)
\pspolygon(1;281.194030)(1;282.985075)(1;284.776119)
\pspolygon(1;286.567164)(1;304.477612)(1;306.268657)
\pspolygon(1;286.567164)(1;302.686567)(1;304.477612)
\pspolygon(1;286.567164)(1;288.358209)(1;302.686567)
\pspolygon(1;288.358209)(1;291.940299)(1;302.686567)
\pspolygon(1;288.358209)(1;290.149254)(1;291.940299)
\pspolygon(1;291.940299)(1;295.522388)(1;302.686567)
\pspolygon(1;291.940299)(1;293.731343)(1;295.522388)
\pspolygon(1;295.522388)(1;299.104478)(1;302.686567)
\pspolygon(1;295.522388)(1;297.313433)(1;299.104478)
\pspolygon(1;299.104478)(1;300.895522)(1;302.686567)
\pspolygon(1;308.059701)(1;309.850746)(1;318.805970)
\pspolygon(1;309.850746)(1;315.223881)(1;318.805970)
\pspolygon(1;309.850746)(1;313.432836)(1;315.223881)
\pspolygon(1;309.850746)(1;311.641791)(1;313.432836)
\pspolygon(1;315.223881)(1;317.014925)(1;318.805970)
\pspolygon(1;318.805970)(1;322.388060)(1;356.417910)
\pspolygon(1;318.805970)(1;320.597015)(1;322.388060)
\pspolygon(1;322.388060)(1;354.626866)(1;356.417910)
\pspolygon(1;322.388060)(1;338.507463)(1;354.626866)
\pspolygon(1;322.388060)(1;336.716418)(1;338.507463)
\pspolygon(1;322.388060)(1;334.925373)(1;336.716418)
\pspolygon(1;322.388060)(1;333.134328)(1;334.925373)
\pspolygon(1;322.388060)(1;324.179104)(1;333.134328)
\pspolygon(1;324.179104)(1;325.970149)(1;333.134328)
\pspolygon(1;325.970149)(1;329.552239)(1;333.134328)
\pspolygon(1;325.970149)(1;327.761194)(1;329.552239)
\pspolygon(1;329.552239)(1;331.343284)(1;333.134328)
\pspolygon(1;338.507463)(1;340.298507)(1;354.626866)
\pspolygon(1;340.298507)(1;342.089552)(1;354.626866)
\pspolygon(1;342.089552)(1;349.253731)(1;354.626866)
\pspolygon(1;342.089552)(1;345.671642)(1;349.253731)
\pspolygon(1;342.089552)(1;343.880597)(1;345.671642)
\pspolygon(1;345.671642)(1;347.462687)(1;349.253731)
\pspolygon(1;349.253731)(1;352.835821)(1;354.626866)
\pspolygon(1;349.253731)(1;351.044776)(1;352.835821)
\pspolygon(1;356.417910)(1;358.208955)(1;360.000000)
\end{pspicture}
 \hfil
	\begin{pspicture}(-1.01,-1.01)(1.01,1.01)
\SpecialCoor
\pspolygon(1;1.791045)(1;241.791045)(1;360.000000)
\pspolygon(1;1.791045)(1;60.895522)(1;241.791045)
\pspolygon(1;1.791045)(1;17.910448)(1;60.895522)
\pspolygon(1;1.791045)(1;14.328358)(1;17.910448)
\pspolygon(1;1.791045)(1;12.537313)(1;14.328358)
\pspolygon(1;1.791045)(1;5.373134)(1;12.537313)
\pspolygon(1;1.791045)(1;3.582090)(1;5.373134)
\pspolygon(1;5.373134)(1;10.746269)(1;12.537313)
\pspolygon(1;5.373134)(1;8.955224)(1;10.746269)
\pspolygon(1;5.373134)(1;7.164179)(1;8.955224)
\pspolygon(1;14.328358)(1;16.119403)(1;17.910448)
\pspolygon(1;17.910448)(1;19.701493)(1;60.895522)
\pspolygon(1;19.701493)(1;25.074627)(1;60.895522)
\pspolygon(1;19.701493)(1;21.492537)(1;25.074627)
\pspolygon(1;21.492537)(1;23.283582)(1;25.074627)
\pspolygon(1;25.074627)(1;48.358209)(1;60.895522)
\pspolygon(1;25.074627)(1;34.029851)(1;48.358209)
\pspolygon(1;25.074627)(1;28.656716)(1;34.029851)
\pspolygon(1;25.074627)(1;26.865672)(1;28.656716)
\pspolygon(1;28.656716)(1;30.447761)(1;34.029851)
\pspolygon(1;30.447761)(1;32.238806)(1;34.029851)
\pspolygon(1;34.029851)(1;41.194030)(1;48.358209)
\pspolygon(1;34.029851)(1;37.611940)(1;41.194030)
\pspolygon(1;34.029851)(1;35.820896)(1;37.611940)
\pspolygon(1;37.611940)(1;39.402985)(1;41.194030)
\pspolygon(1;41.194030)(1;44.776119)(1;48.358209)
\pspolygon(1;41.194030)(1;42.985075)(1;44.776119)
\pspolygon(1;44.776119)(1;46.567164)(1;48.358209)
\pspolygon(1;48.358209)(1;55.522388)(1;60.895522)
\pspolygon(1;48.358209)(1;50.149254)(1;55.522388)
\pspolygon(1;50.149254)(1;53.731343)(1;55.522388)
\pspolygon(1;50.149254)(1;51.940299)(1;53.731343)
\pspolygon(1;55.522388)(1;59.104478)(1;60.895522)
\pspolygon(1;55.522388)(1;57.313433)(1;59.104478)
\pspolygon(1;60.895522)(1;225.671642)(1;241.791045)
\pspolygon(1;60.895522)(1;137.910448)(1;225.671642)
\pspolygon(1;60.895522)(1;66.268657)(1;137.910448)
\pspolygon(1;60.895522)(1;62.686567)(1;66.268657)
\pspolygon(1;62.686567)(1;64.477612)(1;66.268657)
\pspolygon(1;66.268657)(1;82.388060)(1;137.910448)
\pspolygon(1;66.268657)(1;71.641791)(1;82.388060)
\pspolygon(1;66.268657)(1;69.850746)(1;71.641791)
\pspolygon(1;66.268657)(1;68.059701)(1;69.850746)
\pspolygon(1;71.641791)(1;77.014925)(1;82.388060)
\pspolygon(1;71.641791)(1;75.223881)(1;77.014925)
\pspolygon(1;71.641791)(1;73.432836)(1;75.223881)
\pspolygon(1;77.014925)(1;80.597015)(1;82.388060)
\pspolygon(1;77.014925)(1;78.805970)(1;80.597015)
\pspolygon(1;82.388060)(1;123.582090)(1;137.910448)
\pspolygon(1;82.388060)(1;103.880597)(1;123.582090)
\pspolygon(1;82.388060)(1;102.089552)(1;103.880597)
\pspolygon(1;82.388060)(1;87.761194)(1;102.089552)
\pspolygon(1;82.388060)(1;84.179104)(1;87.761194)
\pspolygon(1;84.179104)(1;85.970149)(1;87.761194)
\pspolygon(1;87.761194)(1;91.343284)(1;102.089552)
\pspolygon(1;87.761194)(1;89.552239)(1;91.343284)
\pspolygon(1;91.343284)(1;98.507463)(1;102.089552)
\pspolygon(1;91.343284)(1;93.134328)(1;98.507463)
\pspolygon(1;93.134328)(1;94.925373)(1;98.507463)
\pspolygon(1;94.925373)(1;96.716418)(1;98.507463)
\pspolygon(1;98.507463)(1;100.298507)(1;102.089552)
\pspolygon(1;103.880597)(1;114.626866)(1;123.582090)
\pspolygon(1;103.880597)(1;105.671642)(1;114.626866)
\pspolygon(1;105.671642)(1;111.044776)(1;114.626866)
\pspolygon(1;105.671642)(1;107.462687)(1;111.044776)
\pspolygon(1;107.462687)(1;109.253731)(1;111.044776)
\pspolygon(1;111.044776)(1;112.835821)(1;114.626866)
\pspolygon(1;114.626866)(1;121.791045)(1;123.582090)
\pspolygon(1;114.626866)(1;120.000000)(1;121.791045)
\pspolygon(1;114.626866)(1;116.417910)(1;120.000000)
\pspolygon(1;116.417910)(1;118.208955)(1;120.000000)
\pspolygon(1;123.582090)(1;130.746269)(1;137.910448)
\pspolygon(1;123.582090)(1;128.955224)(1;130.746269)
\pspolygon(1;123.582090)(1;127.164179)(1;128.955224)
\pspolygon(1;123.582090)(1;125.373134)(1;127.164179)
\pspolygon(1;130.746269)(1;136.119403)(1;137.910448)
\pspolygon(1;130.746269)(1;134.328358)(1;136.119403)
\pspolygon(1;130.746269)(1;132.537313)(1;134.328358)
\pspolygon(1;137.910448)(1;152.238806)(1;225.671642)
\pspolygon(1;137.910448)(1;143.283582)(1;152.238806)
\pspolygon(1;137.910448)(1;141.492537)(1;143.283582)
\pspolygon(1;137.910448)(1;139.701493)(1;141.492537)
\pspolygon(1;143.283582)(1;150.447761)(1;152.238806)
\pspolygon(1;143.283582)(1;148.656716)(1;150.447761)
\pspolygon(1;143.283582)(1;146.865672)(1;148.656716)
\pspolygon(1;143.283582)(1;145.074627)(1;146.865672)
\pspolygon(1;152.238806)(1;223.880597)(1;225.671642)
\pspolygon(1;152.238806)(1;154.029851)(1;223.880597)
\pspolygon(1;154.029851)(1;216.716418)(1;223.880597)
\pspolygon(1;154.029851)(1;171.940299)(1;216.716418)
\pspolygon(1;154.029851)(1;166.567164)(1;171.940299)
\pspolygon(1;154.029851)(1;164.776119)(1;166.567164)
\pspolygon(1;154.029851)(1;161.194030)(1;164.776119)
\pspolygon(1;154.029851)(1;155.820896)(1;161.194030)
\pspolygon(1;155.820896)(1;159.402985)(1;161.194030)
\pspolygon(1;155.820896)(1;157.611940)(1;159.402985)
\pspolygon(1;161.194030)(1;162.985075)(1;164.776119)
\pspolygon(1;166.567164)(1;168.358209)(1;171.940299)
\pspolygon(1;168.358209)(1;170.149254)(1;171.940299)
\pspolygon(1;171.940299)(1;202.388060)(1;216.716418)
\pspolygon(1;171.940299)(1;173.731343)(1;202.388060)
\pspolygon(1;173.731343)(1;186.268657)(1;202.388060)
\pspolygon(1;173.731343)(1;182.686567)(1;186.268657)
\pspolygon(1;173.731343)(1;175.522388)(1;182.686567)
\pspolygon(1;175.522388)(1;177.313433)(1;182.686567)
\pspolygon(1;177.313433)(1;179.104478)(1;182.686567)
\pspolygon(1;179.104478)(1;180.895522)(1;182.686567)
\pspolygon(1;182.686567)(1;184.477612)(1;186.268657)
\pspolygon(1;186.268657)(1;198.805970)(1;202.388060)
\pspolygon(1;186.268657)(1;193.432836)(1;198.805970)
\pspolygon(1;186.268657)(1;191.641791)(1;193.432836)
\pspolygon(1;186.268657)(1;189.850746)(1;191.641791)
\pspolygon(1;186.268657)(1;188.059701)(1;189.850746)
\pspolygon(1;193.432836)(1;197.014925)(1;198.805970)
\pspolygon(1;193.432836)(1;195.223881)(1;197.014925)
\pspolygon(1;198.805970)(1;200.597015)(1;202.388060)
\pspolygon(1;202.388060)(1;204.179104)(1;216.716418)
\pspolygon(1;204.179104)(1;209.552239)(1;216.716418)
\pspolygon(1;204.179104)(1;205.970149)(1;209.552239)
\pspolygon(1;205.970149)(1;207.761194)(1;209.552239)
\pspolygon(1;209.552239)(1;211.343284)(1;216.716418)
\pspolygon(1;211.343284)(1;214.925373)(1;216.716418)
\pspolygon(1;211.343284)(1;213.134328)(1;214.925373)
\pspolygon(1;216.716418)(1;222.089552)(1;223.880597)
\pspolygon(1;216.716418)(1;220.298507)(1;222.089552)
\pspolygon(1;216.716418)(1;218.507463)(1;220.298507)
\pspolygon(1;225.671642)(1;231.044776)(1;241.791045)
\pspolygon(1;225.671642)(1;229.253731)(1;231.044776)
\pspolygon(1;225.671642)(1;227.462687)(1;229.253731)
\pspolygon(1;231.044776)(1;234.626866)(1;241.791045)
\pspolygon(1;231.044776)(1;232.835821)(1;234.626866)
\pspolygon(1;234.626866)(1;236.417910)(1;241.791045)
\pspolygon(1;236.417910)(1;238.208955)(1;241.791045)
\pspolygon(1;238.208955)(1;240.000000)(1;241.791045)
\pspolygon(1;241.791045)(1;345.671642)(1;360.000000)
\pspolygon(1;241.791045)(1;248.955224)(1;345.671642)
\pspolygon(1;241.791045)(1;247.164179)(1;248.955224)
\pspolygon(1;241.791045)(1;243.582090)(1;247.164179)
\pspolygon(1;243.582090)(1;245.373134)(1;247.164179)
\pspolygon(1;248.955224)(1;256.119403)(1;345.671642)
\pspolygon(1;248.955224)(1;254.328358)(1;256.119403)
\pspolygon(1;248.955224)(1;250.746269)(1;254.328358)
\pspolygon(1;250.746269)(1;252.537313)(1;254.328358)
\pspolygon(1;256.119403)(1;309.850746)(1;345.671642)
\pspolygon(1;256.119403)(1;300.895522)(1;309.850746)
\pspolygon(1;256.119403)(1;282.985075)(1;300.895522)
\pspolygon(1;256.119403)(1;277.611940)(1;282.985075)
\pspolygon(1;256.119403)(1;272.238806)(1;277.611940)
\pspolygon(1;256.119403)(1;259.701493)(1;272.238806)
\pspolygon(1;256.119403)(1;257.910448)(1;259.701493)
\pspolygon(1;259.701493)(1;266.865672)(1;272.238806)
\pspolygon(1;259.701493)(1;263.283582)(1;266.865672)
\pspolygon(1;259.701493)(1;261.492537)(1;263.283582)
\pspolygon(1;263.283582)(1;265.074627)(1;266.865672)
\pspolygon(1;266.865672)(1;268.656716)(1;272.238806)
\pspolygon(1;268.656716)(1;270.447761)(1;272.238806)
\pspolygon(1;272.238806)(1;274.029851)(1;277.611940)
\pspolygon(1;274.029851)(1;275.820896)(1;277.611940)
\pspolygon(1;277.611940)(1;279.402985)(1;282.985075)
\pspolygon(1;279.402985)(1;281.194030)(1;282.985075)
\pspolygon(1;282.985075)(1;295.522388)(1;300.895522)
\pspolygon(1;282.985075)(1;288.358209)(1;295.522388)
\pspolygon(1;282.985075)(1;286.567164)(1;288.358209)
\pspolygon(1;282.985075)(1;284.776119)(1;286.567164)
\pspolygon(1;288.358209)(1;291.940299)(1;295.522388)
\pspolygon(1;288.358209)(1;290.149254)(1;291.940299)
\pspolygon(1;291.940299)(1;293.731343)(1;295.522388)
\pspolygon(1;295.522388)(1;299.104478)(1;300.895522)
\pspolygon(1;295.522388)(1;297.313433)(1;299.104478)
\pspolygon(1;300.895522)(1;304.477612)(1;309.850746)
\pspolygon(1;300.895522)(1;302.686567)(1;304.477612)
\pspolygon(1;304.477612)(1;308.059701)(1;309.850746)
\pspolygon(1;304.477612)(1;306.268657)(1;308.059701)
\pspolygon(1;309.850746)(1;325.970149)(1;345.671642)
\pspolygon(1;309.850746)(1;322.388060)(1;325.970149)
\pspolygon(1;309.850746)(1;313.432836)(1;322.388060)
\pspolygon(1;309.850746)(1;311.641791)(1;313.432836)
\pspolygon(1;313.432836)(1;318.805970)(1;322.388060)
\pspolygon(1;313.432836)(1;317.014925)(1;318.805970)
\pspolygon(1;313.432836)(1;315.223881)(1;317.014925)
\pspolygon(1;318.805970)(1;320.597015)(1;322.388060)
\pspolygon(1;322.388060)(1;324.179104)(1;325.970149)
\pspolygon(1;325.970149)(1;334.925373)(1;345.671642)
\pspolygon(1;325.970149)(1;329.552239)(1;334.925373)
\pspolygon(1;325.970149)(1;327.761194)(1;329.552239)
\pspolygon(1;329.552239)(1;331.343284)(1;334.925373)
\pspolygon(1;331.343284)(1;333.134328)(1;334.925373)
\pspolygon(1;334.925373)(1;338.507463)(1;345.671642)
\pspolygon(1;334.925373)(1;336.716418)(1;338.507463)
\pspolygon(1;338.507463)(1;343.880597)(1;345.671642)
\pspolygon(1;338.507463)(1;340.298507)(1;343.880597)
\pspolygon(1;340.298507)(1;342.089552)(1;343.880597)
\pspolygon(1;345.671642)(1;354.626866)(1;360.000000)
\pspolygon(1;345.671642)(1;347.462687)(1;354.626866)
\pspolygon(1;347.462687)(1;351.044776)(1;354.626866)
\pspolygon(1;347.462687)(1;349.253731)(1;351.044776)
\pspolygon(1;351.044776)(1;352.835821)(1;354.626866)
\pspolygon(1;354.626866)(1;358.208955)(1;360.000000)
\pspolygon(1;354.626866)(1;356.417910)(1;358.208955)
\end{pspicture}
 \hfil
	\providecommand{\pstrrootedge}{\psline[linestyle=dashed, linewidth=0.3pt, arrows=*-o]}
\providecommand{\pstrintedge}{\psline[linestyle=dashed, linewidth=0.3pt, arrows=*-]}
\providecommand{\pstrextedge}{\psline[linestyle=dashed, linewidth=0.3pt, arrows=o-]}
\providecommand{\addTriangCommand}{\pscircle[linestyle=dotted, linewidth=0.5pt](0,0){1}}
\begin{pspicture}(-1.01,-1.01)(1.01,1.01)
\addTriangCommand
\SpecialCoor
\pspolygon(1;0.000000)(1;149.400000)(1;360.000000)
\pspolygon(1;0.000000)(1;63.000000)(1;149.400000)
\pspolygon(1;0.000000)(1;27.000000)(1;63.000000)
\pspolygon(1;0.000000)(1;12.600000)(1;27.000000)
\pspolygon(1;0.000000)(1;1.800000)(1;12.600000)
\pspolygon(1;1.800000)(1;7.200000)(1;12.600000)
\pspolygon(1;1.800000)(1;3.600000)(1;7.200000)
\pspolygon(1;3.600000)(1;5.400000)(1;7.200000)
\pspolygon(1;7.200000)(1;10.800000)(1;12.600000)
\pspolygon(1;7.200000)(1;9.000000)(1;10.800000)
\pspolygon(1;12.600000)(1;18.000000)(1;27.000000)
\pspolygon(1;12.600000)(1;16.200000)(1;18.000000)
\pspolygon(1;12.600000)(1;14.400000)(1;16.200000)
\pspolygon(1;18.000000)(1;21.600000)(1;27.000000)
\pspolygon(1;18.000000)(1;19.800000)(1;21.600000)
\pspolygon(1;21.600000)(1;25.200000)(1;27.000000)
\pspolygon(1;21.600000)(1;23.400000)(1;25.200000)
\pspolygon(1;27.000000)(1;43.200000)(1;63.000000)
\pspolygon(1;27.000000)(1;32.400000)(1;43.200000)
\pspolygon(1;27.000000)(1;28.800000)(1;32.400000)
\pspolygon(1;28.800000)(1;30.600000)(1;32.400000)
\pspolygon(1;32.400000)(1;36.000000)(1;43.200000)
\pspolygon(1;32.400000)(1;34.200000)(1;36.000000)
\pspolygon(1;36.000000)(1;39.600000)(1;43.200000)
\pspolygon(1;36.000000)(1;37.800000)(1;39.600000)
\pspolygon(1;39.600000)(1;41.400000)(1;43.200000)
\pspolygon(1;43.200000)(1;54.000000)(1;63.000000)
\pspolygon(1;43.200000)(1;52.200000)(1;54.000000)
\pspolygon(1;43.200000)(1;46.800000)(1;52.200000)
\pspolygon(1;43.200000)(1;45.000000)(1;46.800000)
\pspolygon(1;46.800000)(1;50.400000)(1;52.200000)
\pspolygon(1;46.800000)(1;48.600000)(1;50.400000)
\pspolygon(1;54.000000)(1;59.400000)(1;63.000000)
\pspolygon(1;54.000000)(1;57.600000)(1;59.400000)
\pspolygon(1;54.000000)(1;55.800000)(1;57.600000)
\pspolygon(1;59.400000)(1;61.200000)(1;63.000000)
\pspolygon(1;63.000000)(1;99.000000)(1;149.400000)
\pspolygon(1;63.000000)(1;79.200000)(1;99.000000)
\pspolygon(1;63.000000)(1;72.000000)(1;79.200000)
\pspolygon(1;63.000000)(1;64.800000)(1;72.000000)
\pspolygon(1;64.800000)(1;66.600000)(1;72.000000)
\pspolygon(1;66.600000)(1;70.200000)(1;72.000000)
\pspolygon(1;66.600000)(1;68.400000)(1;70.200000)
\pspolygon(1;72.000000)(1;77.400000)(1;79.200000)
\pspolygon(1;72.000000)(1;75.600000)(1;77.400000)
\pspolygon(1;72.000000)(1;73.800000)(1;75.600000)
\pspolygon(1;79.200000)(1;82.800000)(1;99.000000)
\pspolygon(1;79.200000)(1;81.000000)(1;82.800000)
\pspolygon(1;82.800000)(1;93.600000)(1;99.000000)
\pspolygon(1;82.800000)(1;86.400000)(1;93.600000)
\pspolygon(1;82.800000)(1;84.600000)(1;86.400000)
\pspolygon(1;86.400000)(1;90.000000)(1;93.600000)
\pspolygon(1;86.400000)(1;88.200000)(1;90.000000)
\pspolygon(1;90.000000)(1;91.800000)(1;93.600000)
\pspolygon(1;93.600000)(1;97.200000)(1;99.000000)
\pspolygon(1;93.600000)(1;95.400000)(1;97.200000)
\pspolygon(1;99.000000)(1;131.400000)(1;149.400000)
\pspolygon(1;99.000000)(1;120.600000)(1;131.400000)
\pspolygon(1;99.000000)(1;115.200000)(1;120.600000)
\pspolygon(1;99.000000)(1;108.000000)(1;115.200000)
\pspolygon(1;99.000000)(1;102.600000)(1;108.000000)
\pspolygon(1;99.000000)(1;100.800000)(1;102.600000)
\pspolygon(1;102.600000)(1;104.400000)(1;108.000000)
\pspolygon(1;104.400000)(1;106.200000)(1;108.000000)
\pspolygon(1;108.000000)(1;113.400000)(1;115.200000)
\pspolygon(1;108.000000)(1;111.600000)(1;113.400000)
\pspolygon(1;108.000000)(1;109.800000)(1;111.600000)
\pspolygon(1;115.200000)(1;117.000000)(1;120.600000)
\pspolygon(1;117.000000)(1;118.800000)(1;120.600000)
\pspolygon(1;120.600000)(1;127.800000)(1;131.400000)
\pspolygon(1;120.600000)(1;122.400000)(1;127.800000)
\pspolygon(1;122.400000)(1;126.000000)(1;127.800000)
\pspolygon(1;122.400000)(1;124.200000)(1;126.000000)
\pspolygon(1;127.800000)(1;129.600000)(1;131.400000)
\pspolygon(1;131.400000)(1;144.000000)(1;149.400000)
\pspolygon(1;131.400000)(1;142.200000)(1;144.000000)
\pspolygon(1;131.400000)(1;135.000000)(1;142.200000)
\pspolygon(1;131.400000)(1;133.200000)(1;135.000000)
\pspolygon(1;135.000000)(1;138.600000)(1;142.200000)
\pspolygon(1;135.000000)(1;136.800000)(1;138.600000)
\pspolygon(1;138.600000)(1;140.400000)(1;142.200000)
\pspolygon(1;144.000000)(1;147.600000)(1;149.400000)
\pspolygon(1;144.000000)(1;145.800000)(1;147.600000)
\pspolygon(1;149.400000)(1;257.400000)(1;360.000000)
\pspolygon(1;149.400000)(1;221.400000)(1;257.400000)
\pspolygon(1;149.400000)(1;189.000000)(1;221.400000)
\pspolygon(1;149.400000)(1;172.800000)(1;189.000000)
\pspolygon(1;149.400000)(1;165.600000)(1;172.800000)
\pspolygon(1;149.400000)(1;158.400000)(1;165.600000)
\pspolygon(1;149.400000)(1;156.600000)(1;158.400000)
\pspolygon(1;149.400000)(1;151.200000)(1;156.600000)
\pspolygon(1;151.200000)(1;154.800000)(1;156.600000)
\pspolygon(1;151.200000)(1;153.000000)(1;154.800000)
\pspolygon(1;158.400000)(1;163.800000)(1;165.600000)
\pspolygon(1;158.400000)(1;162.000000)(1;163.800000)
\pspolygon(1;158.400000)(1;160.200000)(1;162.000000)
\pspolygon(1;165.600000)(1;169.200000)(1;172.800000)
\pspolygon(1;165.600000)(1;167.400000)(1;169.200000)
\pspolygon(1;169.200000)(1;171.000000)(1;172.800000)
\pspolygon(1;172.800000)(1;176.400000)(1;189.000000)
\pspolygon(1;172.800000)(1;174.600000)(1;176.400000)
\pspolygon(1;176.400000)(1;185.400000)(1;189.000000)
\pspolygon(1;176.400000)(1;183.600000)(1;185.400000)
\pspolygon(1;176.400000)(1;181.800000)(1;183.600000)
\pspolygon(1;176.400000)(1;178.200000)(1;181.800000)
\pspolygon(1;178.200000)(1;180.000000)(1;181.800000)
\pspolygon(1;185.400000)(1;187.200000)(1;189.000000)
\pspolygon(1;189.000000)(1;214.200000)(1;221.400000)
\pspolygon(1;189.000000)(1;196.200000)(1;214.200000)
\pspolygon(1;189.000000)(1;192.600000)(1;196.200000)
\pspolygon(1;189.000000)(1;190.800000)(1;192.600000)
\pspolygon(1;192.600000)(1;194.400000)(1;196.200000)
\pspolygon(1;196.200000)(1;199.800000)(1;214.200000)
\pspolygon(1;196.200000)(1;198.000000)(1;199.800000)
\pspolygon(1;199.800000)(1;208.800000)(1;214.200000)
\pspolygon(1;199.800000)(1;207.000000)(1;208.800000)
\pspolygon(1;199.800000)(1;205.200000)(1;207.000000)
\pspolygon(1;199.800000)(1;203.400000)(1;205.200000)
\pspolygon(1;199.800000)(1;201.600000)(1;203.400000)
\pspolygon(1;208.800000)(1;210.600000)(1;214.200000)
\pspolygon(1;210.600000)(1;212.400000)(1;214.200000)
\pspolygon(1;214.200000)(1;217.800000)(1;221.400000)
\pspolygon(1;214.200000)(1;216.000000)(1;217.800000)
\pspolygon(1;217.800000)(1;219.600000)(1;221.400000)
\pspolygon(1;221.400000)(1;239.400000)(1;257.400000)
\pspolygon(1;221.400000)(1;234.000000)(1;239.400000)
\pspolygon(1;221.400000)(1;230.400000)(1;234.000000)
\pspolygon(1;221.400000)(1;226.800000)(1;230.400000)
\pspolygon(1;221.400000)(1;225.000000)(1;226.800000)
\pspolygon(1;221.400000)(1;223.200000)(1;225.000000)
\pspolygon(1;226.800000)(1;228.600000)(1;230.400000)
\pspolygon(1;230.400000)(1;232.200000)(1;234.000000)
\pspolygon(1;234.000000)(1;235.800000)(1;239.400000)
\pspolygon(1;235.800000)(1;237.600000)(1;239.400000)
\pspolygon(1;239.400000)(1;250.200000)(1;257.400000)
\pspolygon(1;239.400000)(1;243.000000)(1;250.200000)
\pspolygon(1;239.400000)(1;241.200000)(1;243.000000)
\pspolygon(1;243.000000)(1;246.600000)(1;250.200000)
\pspolygon(1;243.000000)(1;244.800000)(1;246.600000)
\pspolygon(1;246.600000)(1;248.400000)(1;250.200000)
\pspolygon(1;250.200000)(1;255.600000)(1;257.400000)
\pspolygon(1;250.200000)(1;253.800000)(1;255.600000)
\pspolygon(1;250.200000)(1;252.000000)(1;253.800000)
\pspolygon(1;257.400000)(1;279.000000)(1;360.000000)
\pspolygon(1;257.400000)(1;273.600000)(1;279.000000)
\pspolygon(1;257.400000)(1;268.200000)(1;273.600000)
\pspolygon(1;257.400000)(1;262.800000)(1;268.200000)
\pspolygon(1;257.400000)(1;259.200000)(1;262.800000)
\pspolygon(1;259.200000)(1;261.000000)(1;262.800000)
\pspolygon(1;262.800000)(1;266.400000)(1;268.200000)
\pspolygon(1;262.800000)(1;264.600000)(1;266.400000)
\pspolygon(1;268.200000)(1;271.800000)(1;273.600000)
\pspolygon(1;268.200000)(1;270.000000)(1;271.800000)
\pspolygon(1;273.600000)(1;275.400000)(1;279.000000)
\pspolygon(1;275.400000)(1;277.200000)(1;279.000000)
\pspolygon(1;279.000000)(1;338.400000)(1;360.000000)
\pspolygon(1;279.000000)(1;311.400000)(1;338.400000)
\pspolygon(1;279.000000)(1;300.600000)(1;311.400000)
\pspolygon(1;279.000000)(1;293.400000)(1;300.600000)
\pspolygon(1;279.000000)(1;284.400000)(1;293.400000)
\pspolygon(1;279.000000)(1;282.600000)(1;284.400000)
\pspolygon(1;279.000000)(1;280.800000)(1;282.600000)
\pspolygon(1;284.400000)(1;288.000000)(1;293.400000)
\pspolygon(1;284.400000)(1;286.200000)(1;288.000000)
\pspolygon(1;288.000000)(1;291.600000)(1;293.400000)
\pspolygon(1;288.000000)(1;289.800000)(1;291.600000)
\pspolygon(1;293.400000)(1;297.000000)(1;300.600000)
\pspolygon(1;293.400000)(1;295.200000)(1;297.000000)
\pspolygon(1;297.000000)(1;298.800000)(1;300.600000)
\pspolygon(1;300.600000)(1;306.000000)(1;311.400000)
\pspolygon(1;300.600000)(1;302.400000)(1;306.000000)
\pspolygon(1;302.400000)(1;304.200000)(1;306.000000)
\pspolygon(1;306.000000)(1;309.600000)(1;311.400000)
\pspolygon(1;306.000000)(1;307.800000)(1;309.600000)
\pspolygon(1;311.400000)(1;327.600000)(1;338.400000)
\pspolygon(1;311.400000)(1;320.400000)(1;327.600000)
\pspolygon(1;311.400000)(1;315.000000)(1;320.400000)
\pspolygon(1;311.400000)(1;313.200000)(1;315.000000)
\pspolygon(1;315.000000)(1;318.600000)(1;320.400000)
\pspolygon(1;315.000000)(1;316.800000)(1;318.600000)
\pspolygon(1;320.400000)(1;325.800000)(1;327.600000)
\pspolygon(1;320.400000)(1;324.000000)(1;325.800000)
\pspolygon(1;320.400000)(1;322.200000)(1;324.000000)
\pspolygon(1;327.600000)(1;334.800000)(1;338.400000)
\pspolygon(1;327.600000)(1;331.200000)(1;334.800000)
\pspolygon(1;327.600000)(1;329.400000)(1;331.200000)
\pspolygon(1;331.200000)(1;333.000000)(1;334.800000)
\pspolygon(1;334.800000)(1;336.600000)(1;338.400000)
\pspolygon(1;338.400000)(1;352.800000)(1;360.000000)
\pspolygon(1;338.400000)(1;345.600000)(1;352.800000)
\pspolygon(1;338.400000)(1;342.000000)(1;345.600000)
\pspolygon(1;338.400000)(1;340.200000)(1;342.000000)
\pspolygon(1;342.000000)(1;343.800000)(1;345.600000)
\pspolygon(1;345.600000)(1;351.000000)(1;352.800000)
\pspolygon(1;345.600000)(1;349.200000)(1;351.000000)
\pspolygon(1;345.600000)(1;347.400000)(1;349.200000)
\pspolygon(1;352.800000)(1;358.200000)(1;360.000000)
\pspolygon(1;352.800000)(1;356.400000)(1;358.200000)
\pspolygon(1;352.800000)(1;354.600000)(1;356.400000)
\end{pspicture} \fi
\end{center}
\caption{Realisations of $\beta$-splitting trees for (from left to right)
	$\beta=-1$, $\beta=0$ (Yule tree), $\beta=10$.}
\label{f:betasplit}
\end{figure}
\begin{example}[$\beta$-splitting trees]
	For every $\beta\in [-2,\infty]$, let $T^\beta_n$ be the $\beta$\nbd splitting tree on $n$ leaves
	from \cite{Aldous1996} (with forgotten labels). For $-2<\beta<\infty$, the $\beta$\nbd splitting tree
	$T^\beta_n$ can be constructed recursively as follows.
	$T^\beta_2$ consists of two leaves connected by a distinguished root edge. If $n>2$, choose $i\in
	\{1,\ldots,n-1\}$ with probability
	\begin{equation}
		q_n^\beta(i) = \frac1{a_n(\beta)} \binom ni \int_0^1 x^{i+\beta}(1-x)^{n-i+\beta} \,\dx x,
	\end{equation}
	where $a_n(\beta)$ is a normalisation constant. Then construct two independent $\beta$\nbd splitting trees $T^\beta_i$ and
	$T^\beta_{n-i}$, introduce a new branch point in the middle of each of the two root edges, and connect
	these new branch points with the new root edge to obtain $T^\beta_n$.
	
	It is easy to see (and observed in \cite{Aldous1996}) that $(T^\beta_n)_{n\in\N}$ is sampling consistent.
	Note the special cases $\beta=-2$ which is the {\em comb tree}, $\beta=-\frac32$ which is the \emph{uniform cladogram},
	$\beta=0$ which is the \emph{Yule tree} and $\beta=\infty$ which is the {\em symmetric binary tree}.
See Figure~\ref{f:betasplit} for triangulations of a realization of
	$\beta$\nbd splitting trees for different values of $\beta$ and large $n$. The Aldous Brownian CRT, which is the
	limit for $\beta=-\frac32$, is shown in Figure~\ref{f:nontriang}. 
\label{exp:009}
\end{example}

\begin{lemma}[convergence of sampling consistent families]\label{l:sampconsist}
	Let\/ $((T_n,c_n))_{n\in\N}$ be a sampling consistent family of random binary trees, and\/ $\mu_n$ the uniform
	distribution on\/ $\lf(T_n,c_n)$. Then we have the convergence in law
	\begin{equation}
		(T_n,c_n,\mu_n) \toln (T,c,\mu)	\quad\text{on\/ $\Tbin$ with bpdd-Gromov-weak  topology}
	\end{equation}
	for some random algebraic measure tree\/ $(T,c,\mu) \in \Tbin$ with non-atomic measure\/ $\mu$.
\end{lemma}
\begin{proof}
	Recall the $m$\nbd tree shape distribution $\shapedist$ from Definition~\ref{d:shapeconv}.
	Let $n,m\in\N$ with $m<n$ and define
	\begin{equation}
		\eps_{n,m} := \mu_n^{\otimes m}\bset{x\in T^m}{x_1,\ldots,x_m \text{ not distinct}}
			\le \tfrac{m^2}{n}.
	\end{equation}
	Because $(T_n)$ is sampling consistent, we obtain for the annealed shape distribution
	\begin{equation}\label{e:shapeconv}
		\Exp\(\shapedist(T_n,c_n,\mu_n)\) = (1-\eps_{n,m}) \law(T_m^*) + \eps_{n,m} \mu_{n,m},
	\end{equation}
	where $T_m^*$ is obtained from $T_m$ by randomly labelling the leaves, and $\mu_{n,m} \in \MC_1(\Clad)$
	is some law of \nclad s supported by cladograms where at least one leaf has more than one label.
	This shows that, for every fixed $m$, the expected $m$\nbd tree shape distribution converges as
	$n\to\infty$. Because the $m$\nbd tree shape distribution is convergence determining for the bpdd-Gromov-weak
	topology by Corollary~\ref{c:convdet}, all limit points of $\law(T_n,c_n,\mu_n)$ in $\MC_1(\Tbin)$ coincide.
	According to Theorem~\ref{t:topeq}, $\Tbin$, and hence $\MC_1(\Tbin)$, is compact and thus a unique
	limit exists. That the limiting measure is non-atomic is obvious, because the probability that a sampled
	shape is single-labelled tends to one by \eqref{e:shapeconv}.
\end{proof}

In the parameter range $\beta \in \ropenint{-2}{-1}$, the height (in graph distance) of the $\beta$-splitting
tree with $n$ leaves is asymptotically of power-law order $\Theta(n^{-\beta -1})$.
In this case, after rescaling edge-lengths with the factor $n^{\beta+1}$, Gromov-Hausdorff convergence in law to
a fragmentation tree is shown in \cite[Corollary~16]{HaasMiermontPitmanWinkel08}.
In the case $\beta>-1$, the height of the tree is only of logarithmic order $\Theta(\log(n))$, and it is easy to
see that no non-trivial Gromov-Hausdorff scaling limit (with uniform edge rescalings) exists. Seen as algebraic
measure trees, however, it easily follows from sampling consistency that the bpdd-Gromov-weak limit exist in the
full parameter range $\beta\in [-2,\infty]$.

\begin{example}[$\beta$-splitting trees continued]
	By Lemma~\ref{l:sampconsist}, for every $\beta \in [-2,\infty]$, the sequence $(T_n^\beta, c_n^\beta,
	\mu_n^\beta)_{n\in\N}$ of increasing $\beta$\nbd splitting trees converges in distribution to some
	limiting random algebraic measure tree $(T^\beta, c^\beta, \mu^\beta)$. In the case of the uniform cladogram
	($\beta=-\frac32$), the limit is the Brownian algebraic continuum random tree which can be obtained as
	tree $\tree(\CCRT)$ coded by the Brownian triangulation (see Example~\ref{ex:CRT}), or as the algebraic
	measure tree induced by the metric measure Brownian CRT which is known to have uniform shape
	distribution (\cite{Aldous1993}).
	In the case of the comb tree ($\beta=-2$), the limit is the unit interval with Lebesgue measure (a
	coding triangulation is shown in the very right of Figure~\ref{f:nobranch}).
\label{exp:008}
\end{example}

\appendix

\section{A uniform Glivenko-Cantelli theorem}
In Subsections~\ref{s:convshape} and~\ref{s:convmass} we made use of uniform estimates of the speed of convergence
in the approximation of the branch point distribution and the measure of a algebraic measure tree by empirical
distribution. Such uniform Glivenko-Cantelli estimates under a bound on the Vapnik-Chervonenkis dimension (VC-dimension) of the type presented below should be
well-known. As we did not find it explicitly in sufficient generality in the literature, we will present it here.

We recall the
definition of VC-dimension, going back to the seminal work of Vapnik and Chervonenkis,
\cite{VapnikChervonenkis71}.
Let $E$ be a non-empty set and $\J$ a non-empty collection of subsets of $E$. For $n\in\N$ and $x\in E^n$, put
\begin{equation}
	\J(x) := \bset{(\one_I(x_1),\ldots,\one_I(x_n))}{I\in\J}\subseteq\{0,1\}^n.
\end{equation}
Then obviously, $1\le\#\J(x)\le 2^n$.
\begin{definition}[Vapnik-Chervonenkis dimension]
The \emph{Vapnik-Chervonenkis dimension} of $\J$ is defined as
\begin{equation}
	\dimVC(\J) := \sup \bset{n\in\N}{\max_{x\in E^n} \#\J(x) = 2^n}.
\end{equation}
\label{def:VCdim}
\end{definition}

\begin{example}[collection of intervals of an algebraic tree]
	Let $(T,c)$ be a separable algebraic tree with $\#T>2$, and
\begin{equation}
\label{e:135}
   \J:=\I:=\bset{[u,v]}{u,v\in T}.
\end{equation}
	For $x_1,x_2,u \in T$ distinct, we have $\#\J(x) \ge \#\{[u,u],\, [x_1,x_1],\, [x_2,x_2],\,
	[x_1,x_2]\}=2^2$, hence $\dimVC(\I) \ge 2$. Conversely, for $x\in T^3$, either there is $u,v\in T$ with
	$x_1,x_2,x_3 \in [u,v]$. Then w.l.o.g.\ $x_2\in [x_1,x_3]$ and $(1,0,1)\not\in \I(x)$. Or there is no
	such $u,v \in T$, which means $(1,1,1) \not\in \I(x)$. Therefore,
	\begin{equation}
\label{e:VCpath}
		\dimVC(\I) = 2.
	\end{equation}
\label{exp:004}
\end{example}

Recall the notion $\Sub_x(y)$ of the equivalence class of $T\setminus\{x\}$ containing $y$.
\begin{example}[collection of subtrees branching of a branch point]
	Let $(T,c)$ be a separable algebraic tree, and
\begin{equation}
\label{e:136}
   \J:=\Sset:=\bset{\Sub_v(u)}{u,v\in T}.
\end{equation}
We claim that
\begin{equation}\label{e:VCsubtree}
		\dimVC(\Sset) \le  3.
\end{equation}
For this upper bound, let $x=(x_1,x_2,x_3,x_4)\in T^4$.
By the $4$-point condition of the branch point map, we can assume w.l.o.g.\ that
\begin{equation} \label{e:034}
   c(x_1,x_2,x_3)=c(x_1,x_2,x_4).
\end{equation}
In this case, it is not possible to cover $\{x_1,x_3\}$ but neither $x_2$ nor $x_4$ with a single subtree in
$\Sset$, which proves the claim.
\label{exp:006}
\end{example}

The constant in front in the following Glivenko-Cantelli lemma is clearly not optimal.
For us it is only important that it is universal and not depending on the measure space $(E,\mu)$.

\begin{lemma}[rate of convergence in Glivenko-Cantelli]
	Let\/ $E$ be a Polish space, $\mu$ a probability measure on\/ $E$, $(X_n)_{n\in\N}$ i.i.d.\
	$\mu$\nbd distributed, and\/ $\mu_n=\frac1n \sum_{k=1}^n \delta_{X_k}$ the empirical measure.
	Then, for every\/ $\J \subseteq \B(E)$ with\/ $\dimVC(\J)<\infty$ and\/ $n>1$,
	\begin{equation}
		\Exp\Bigl(\sup_{I\in\J}\bigl| \mu(I) - \mu_n(I)\bigr|\Bigr)
			\le 96 \sqrt{\frac{\dimVC(\J)}{n}}.
	\end{equation}
\label{l:VCestim}
\end{lemma}

\begin{proof}
	By the Kuratowski isomorphism theorem, all uncountable Polish spaces are Borel-isomorphic. Therefore, we
	may assume w.l.o.g.\ that $E=\R$. Theorem~3.2 in \cite{DevroyeLugosi01} yields
	\begin{equation}\label{e:DevLug}
		\Delta := \Exp\(\sup_{I\in\J}\bigl| \mu(I) - \mu_n(I)\bigr|\) \le
			\frac{24}{\sqrt{n}} \sup_{x\in \R^n} \int_0^1 \sqrt{\log\(2 N(r,\J(x))\)} \,\dx r,
	\end{equation}
	where $N(r,\J(x))$ is the covering number of $\J(x)$ w.r.t.\ the metric $\frac1{\sqrt{n}} \cdot d_{\ell^2}$, where $d_{\ell^2}$ is the
	Euclidean metric on $\{0,1\}^n$.
	This covering number can be upper-bounded in terms of the separation number $M(r, \J)$ w.r.t.\ the
	metric $\frac1n\cdot d_{\ell^1}$ used by Haussler in \cite{Haussler95}, and Theorem~1 there yields
	\begin{equation}\label{e:Haussler}
		N\(r,\J(x)\) \le M\(r^2, \J(x)\) \le e(\dimVC(\J)+1) \Bigl( \frac{2e}{r^2} \Bigr)^{\dimVC(\J)},
	\end{equation}
	provided that $nr^2\in\N$. For $r^2\le\frac1n$, we use the trivial estimate $M(r^2,\J(x)) \le 2^n$.
	For general $r^2\ge \frac1n$, we estimate $M(r^2,\J(x)) \le M(\frac1n \floor{nr^2},\J(x))$, and
	inserting \eqref{e:Haussler} into \eqref{e:DevLug} yields
	\begin{equation}
	\begin{aligned}
	  \Delta &\le \frac{24}{\sqrt{n}} \Bigl( \sqrt{\tfrac{n+1}n}
		+ \int_{\frac1{\sqrt{n}}}^1  \sqrt{\log(2e(\dimVC(\J)+1))
			+ \dimVC(\J) \log\(2e(r^2-\tfrac1n)^{-1}\)} \,\dx r \Bigr)\\
	    & \le \frac{24}{\sqrt{n}}\sqrt{\dimVC(\J)} \Bigl(
		\sqrt{\tfrac{n+1}{n}} +  \int_0^1 \sqrt{3 + \log(2e) -2\log(r)} \,\dx r \Bigr),
	\end{aligned}
	\end{equation}
	where we used that $\log(2e(d+1)) \le 3d$ for $d\ge 1$, and $r^2 - \frac1n \ge (r-\frac1{\sqrt{n}})^2$.
	The last bracket is less than $4$ for $n>1$, and the claim follows.
\end{proof}

\addtocontents{toc}{\protect\setcounter{tocdepth}{-1}}

\bibliographystyle{amsalpha}
\providecommand{\bysame}{\leavevmode\hbox to3em{\hrulefill}\thinspace}
\providecommand{\MR}{\relax\ifhmode\unskip\space\fi MR }
\providecommand{\MRhref}[2]{
  \href{http://www.ams.org/mathscinet-getitem?mr=#1}{#2}
}
\providecommand{\href}[2]{#2}

\end{document}